\documentclass[mathpazo]{jms}

\usepackage{amssymb}
\usepackage{amsmath}
\usepackage{amsthm}
\usepackage{color}
\usepackage{mathrsfs}

\newtheorem{corollary}[theorem]{Corollary}

\def\red{\color{red}}

\def\rr{{\mathbb R}}
\def\rn{{\mathbb{R}^n}}
\def\rrm{{\mathbb{R}^m}}
\def\zz{{\mathbb Z}}
\def\nn{{\mathbb N}}
\def\cp{{\mathbb P}}
\def\cs{{\mathcal S}}
\def\CO{{\mathcal O}}

\def\sa{\sigma}
\def\fz{\infty}
\def\az{\alpha}
\def\oz{\omega}
\def\bz{\beta}
\def\dz{\delta}
\def\lz{\lambda}
\def\vaz{\varepsilon}
\def\pa{\partial}

\def\lf{\left}
\def\r{\right}
\def\hs{\hspace{0.25cm}}
\def\ls{\lesssim}
\def\gs{\gtrsim}
\def\noz{\nonumber}
\def\wz{\widetilde}
\def\com{\complement}

\def\lg{\langle}
\def\rg{\rangle}

\def\loc{{\mathop\mathrm{\,loc\,}}}
\def\supp{\mathop\mathrm{\,supp\,}}
\def\esup{\mathop\mathrm{\,ess\,sup\,}}
\def\gfz{\genfrac{}{}{0pt}{}}

\def\va{\vec{a}}
\def\vp{\vec{p}}
\def\vq{\vec{q}}
\def\vm{{\mathcal{M}^{p}_{\vq}(\rn)}}
\def\Aw{[\oz]_{\mathbf{A}_p(\rn)}}
\def\Al{\mathbf{A}_1(\rn)}
\def\Alw{[\oz]_{\mathbf{A}_1(\rn)}}
\def\Ap{\mathbf{A}_p(\rn)}
\def\HL{M_{{\rm HL}}}
\def\vh{{H_{\va}^{\vp}(\rn)}}
\def\vAh{{H_{A}^{p}(\rn)}}
\def\vah{{H_{\va}^{\vp,\,r,\,s}(\rn)}}

\def\vfah{{H_{\va,\,{\rm fin}}^{\vp,\,r,\,s}(\rn)}}
\def\vfahfz{{H_{\va,\,{\rm fin}}^{\vp,\,\fz,\,s}(\rn)}}
\def\lv{{L^{\vp}(\rn)}}
\def\lvq{{L^{\vq}(\rn)}}
\def\mlv{{L^{\vp,\fz}(\rn\times\rrm)}}

\def\lq{\mathcal{L}_{\vp,\,q,\,s}^{\va}(\rn)}
\def\lr{\mathcal{L}_{\vp,\,r',\,s}^{\va}(\rn)}

\def\Qkk{B_k^*}
\def\Bik{{B_i^k}}
\def\Qik{{Q_i^k}}

\begin{document}
\title{On Function Spaces with Mixed Norms --- A Survey}


\author[L.~Huang and D.~Yang]{Long Huang and Dachun Yang\corrauth}
\address{Laboratory of Mathematics and Complex Systems
(Ministry of Education of China),
School of Mathematical Sciences, Beijing Normal University,
Beijing 100875, People's Republic of China}
\emails{{\tt longhuang@mail.bnu.edu.cn} (L.~Huang),
{\tt dcyang@bnu.edu.cn} (D.~Yang)/{\red August 8, 2019}/Final version}


\begin{abstract}
The targets of this article are threefold. The first one
is to give a survey on the recent developments
of function spaces with mixed norms,
including mixed Lebesgue spaces, iterated weak
Lebesgue spaces, weak mixed-norm Lebesgue spaces and
mixed Morrey spaces as well as anisotropic mixed-norm Hardy
spaces. The second one is to provide a detailed proof for a useful inequality
about mixed Lebesgue norms and the Hardy--Littlewood maximal
operator and also to improve
some known results on the maximal function characterizations of anisotropic
mixed-norm Hardy spaces and the boundedness of Calder\'on--Zygmund operators
from these anisotropic
mixed-norm Hardy spaces to themselves or to mixed Lebesgue spaces.
The last one is to correct
some errors and seal some gaps existing in the known articles.
\end{abstract}


\maketitle

\tableofcontents

\section{Introduction}\label{sec1}

In 1961, the mixed Lebesgue space $\lv$, with
$\vp\in (0,\fz]^n$, as a natural generalization
of the classical Lebesgue space $L^p(\rn)$ via replacing
the constant exponent $p$ by an exponent vector $\vp$,
was investigated by Benedek and Panzone \cite{bp61}.
Indeed, the origin of these mixed Lebesgue spaces can be traced
back to the interesting article of H\"{o}rmander \cite{h60}
on the estimates for translation invariant operators, in 1960.
Later on, in 1965, Galmarino and Panzone \cite{gp65}
extended the mixed Lebesgue
space $\lv$ to the mixed Lebesgue space $L^P(\rn)$ with
the exponent $P$ being a sequence, namely, an $\fz$-tuple.
Since the early 1960s, lots of nice work have been done in the study
of the boundedness of operators on mixed norm spaces; see,
for instance, Benedek et al. \cite{bcp62},
Lizorkin \cite{l70}, Adams and Bagby \cite{ab74}, Schmeisser
\cite{sc84}, Rubio de Francia et al. \cite{rrt86}
and Fernandez \cite{f87} as well as Stefanov and Torres \cite{st04}.
Recently, the Plancherel--Polya inequality
on mixed Lebesgue spaces $\lv$ and the wavelet characterization of
$\lv$ were studied by Torres and Ward \cite{tw15}; the smoothing
properties of bilinear operators and Leibniz-type rules in
mixed Lebesgue spaces $\lv$ were considered by Hart et al.
\cite{htw17}; the boundedness of the multilinear strong maximal operator
from the product of mixed Lebesgue spaces to mixed Lebesgue spaces
was obtained by Liu et al. \cite{ltxy}.
In addition, more recently, C\'{o}rdoba and
Latorre Crespo in \cite{cl17} revisited
some classical conjectures in harmonic analysis in the
setting of mixed norm spaces. To be exact, they established the
sharp boundedness for the restriction of the Fourier transform
to compact hypersurfaces of revolution and studied an extension
of the disc multiplier in the mixed norm setting.
For more progresses about the mixed Lebesgue space, we refer
the reader to \cite{ai16,cs,cs1,drw16,ho18,ig86,ma82,sj91}.

On another hand, motivated by the aforementioned work of
Benedek and Panzone \cite{bp61} on mixed Lebesgue spaces
$\lv$, numerous other function spaces with mixed norms were
introduced and studied. For instance,
Besov spaces, Sobolev spaces and Bessel potential
spaces with mixed norms were investigated by
Besov et al. \cite{bin78,bin79}
in the 1970s; inhomogeneous Triebel--Lizorkin spaces
with mixed norms were also studied by Besov et al.
in \cite{bin96}; parabolic function spaces
with mixed norms were considered by Gopala Rao \cite{g91}.
Particularly, in 1977, Fernandez \cite{f77} first introduced
the Lorentz spaces with mixed norms. Later, an interpolation
result on these Lorentz spaces with mixed norms was obtained
by Milman \cite{m81}. Moreover, Lorentz--Marcinkiewicz
spaces with mixed norms and Orlicz spaces with mixed norms
were considered by Milman in \cite{mm78} and \cite{mm81},
respectively; Banach
function spaces with mixed norms were studied by Blozinski \cite{b81};
anisotropic mixed-norm Hardy spaces were introduced by Clenathous
et al. in \cite{cgn17}; mixed Lebesgue spaces
with variable exponents were considered by Ho \cite{ho18};
Morrey spaces with mixed norms were investigated by Nogayama
\cite{tn,tn19}. Indeed, the function spaces with mixed norms
have attracted considerable attention and have rapidly been
developed. For more progresses about various function
spaces with mixed norms and their applications in the boundedness
of different operators, we refer the reader to
\cite{cs,cgn17bs,cgn17-2,gjn17,gn16,htw17,ho16,jms13,jms14,jms15,js07}.

In the last two decades, due to the wider usefulness of
function spaces with mixed norms within the context of partial
differential equations, there has been a renewed interest in the study
of them. More precisely, since the function spaces with
mixed norms have finer structures than the corresponding
classical function spaces, they naturally arise
in the studies on the solutions of partial differential equations
used to model physical processes involving in both space
and time variables, such as the heat or the wave equations (particularly,
the very useful Strichartz estimates); see, for instance,
\cite{ai16,k04,kpv93,kim08,tao06}. This is also based on the fact that,
while treating some linear or nonlinear equations, functions with different
orders of integrability in different variables  give more precise
information on the parameters involved in the estimates and
further induce a better regularity
(of traces) of solutions; see, for instance, \cite{dk18,gs91,w02}.
Another recent interest in developing the theory of function spaces
with mixed norms comes from
bilinear estimates and their vector-valued extensions
which have proved useful in partial differential equations involving functions
in $n$ dimension space variable $x$ and one dimension time variable $t$;
see, for instance, \cite{bm16,bm17,fk00,htw17,tw15}.
In particular, in order to obtain
the smoothing properties of bilinear operators and Leibniz-type rules in
mixed Lebesgue spaces, Hart et al. \cite{htw17}
introduced the mixed-norm Hardy space $H^{p,q}(\mathbb{R}^{n+1})$
with $p,$ $q \in(0,\fz)$ via the Littlewood--Paley g-function.
As was mentioned in their article \cite[p.\,8586]{htw17}, the space
$H^{p,q}(\mathbb{R}^{n+1})$ plays an important role in overcoming the
difficulty caused by full derivatives both in the space variable $x$
and the time variable $t$ in the mixed Lebesgue spaces. For more progresses
about the applications of function spaces with mixed norms in partial
differential equations, we refer the reader to \cite{cmr13,kim08,k07,ms97}.

The purposes of this article are threefold. The first one
is to give a survey on the recent developments
of function spaces with mixed norms,
including mixed Lebesgue spaces, iterated weak
Lebesgue spaces, weak mixed-norm Lebesgue spaces and
mixed Morrey spaces as well as anisotropic mixed-norm Hardy
spaces. To be precise, the main results
that we review include: the interpolation theorems,
the dual inequality of Stein type, the Fefferman--Stein vector-valued inequality
as well as the boundedness of fractional integral operators and geometric inequalities
on these three kinds of (weak) Lebesgue spaces with mixed norms,
the boundedness of maximal operators, Calder\'{o}n--Zygmund operators
and fractional integral operators on mixed Morrey spaces $\vm$,
a necessary and sufficient condition for the boundedness
of the commutators of fractional integral operators on $\vm$,
various real-variable characterizations of the anisotropic
mixed-norm Hardy spaces $\vh$, with $\va\in[1,\fz)^n$ and
$\vp\in(0,\fz)^n$, and their dual spaces as well as
applications to the boundedness of Calder\'{o}n--Zygmund operators.
The second purpose is to provide a detailed proof for an extended inequality
on the central Hardy--Littlewood maximal operator
on mixed Lebesgue norms stated
by Bagby in \cite{b75} but without giving a proof and also improve
some known results on the maximal function characterizations of $\vh$
given in \cite[Theorem 3.1]{cgn17} and the boundedness of Calder\'{o}n--Zygmund
operators in \cite[Theorems 6.8 and 6.9]{hlyy}. The last purpose is to correct
some errors and seal some gaps existing in the proof of the
Lusin area function characterizations of $\vh$,
namely, the proof of the sufficiency of \cite[Theorem 4.1]{hlyy}.

The organization of this survey is as follows.

In Section \ref{sec2}, we first recall the notions of the mixed Lebesgue
space $\lv$ with $\vp\in(0,\fz]^n$, the iterated weak Lebesgue space
$$L^{q_2,\fz}(L^{q_1,\fz})(\rn\times\rrm)\quad\mathrm{with}\quad q_1,\ q_2\in(0,\fz),$$
the weak mixed-norm Lebesgue space $L^{\vec q,\fz}(\rr^{2n})$ with $\vq\in(0,\fz]^2$
and their basic properties which include the completeness, the corresponding
H\"{o}lder inequalities and interpolation theorems. Then we present an
extended very useful inequality about mixed Lebesgue norms and the central Hardy--Littlewood maximal
operator, which was stated by Bagby in \cite{b75} but without giving a proof.
We provide a detailed proof of this important inequality in Subsection \ref{2s2} below.
Finally, we give a survey on applications of these Lebesgue spaces
with mixed norms, which include the dual inequality of Stein type and
Fefferman--Stein vector-valued inequality on mixed
Lebesgue spaces $\lv$ proved by Nogayama in \cite{tn} as well as
the boundedness of fractional integral operators and geometric inequalities
on iterated weak Lebesgue spaces $L^{q_2,\fz}(L^{q_1,\fz})(\rr^{2n})$
and weak mixed-norm Lebesgue spaces $L^{\vec q,\fz}(\rr^{2n})$
obtained by Chen and Sun in \cite{cs}.

The aim of Section \ref{sec3} is the summarization of
mixed Morrey spaces $\vm$, with $\vq\in (0,\fz]^n$ and $p\in(0,\fz]$,
and their applications to the boundedness of operators.
To this end, we first recall the notion and some
examples of mixed Morrey spaces $\vm$ from Nogayama \cite{tn}
and then we further show some basic properties about these spaces,
including the completeness as well as the embedding
theorem, obtained in the same aforementioned article. Moreover,
the behaviors of the Hardy--Littlewood maximal operator,
the iterated maximal operator, Calder\'{o}n--Zygmund operators and fractional integral
operators on mixed Morrey spaces $\vm$ are discussed. At the end of Section
\ref{sec3}, we review a necessary and sufficient condition for the boundedness
of commutators of fractional integral operators $I_{\az}$ with $\az\in(0,n)$
on $\vm$ established by Nogayama in \cite{tn19}.

In Section \ref{sec4}, via recalling the notions of anisotropic
quasi-homogeneous norms in \cite{bil66, f66} (see also \cite{sw78})
and anisotropic non-tangential grand maximal functions in \cite{cgn17},
we first give the definition of anisotropic mixed-norm Hardy
spaces $\vh$ appearing in \cite{cgn17} (see also \cite{hlyy}), where
$\va\in[1,\fz)^n$ and $\vp\in(0,\fz)^n$, and present some basic facts
about them. Then various real-variable characterizations of
the spaces $\vh$, respectively, in terms of various maximal
functions, atoms, finite atoms and the
Lusin area function as well as the Littlewood--Paley $g$-function or
the Littlewood--Paley $g_{\lambda}^\ast$-function, established in
\cite{cgn17,hlyy}, are displayed. Furthermore, as the applications
of these various real-variable characterizations, the dual spaces of $\vh$,
a criterion on the boundedness of sublinear operators from $\vh$ into a
quasi-Banach space and the boundedness of anisotropic convolutional
$\delta$-type and non-convolutional $\bz$-order Calder\'on--Zygmund operators
from $\vh$ to itself [or to $\lv$], obtained in \cite{hlyy}, are presented.
Some errors and gaps existing in the proof of the sufficiency of
\cite[Theorem 4.1]{hlyy} are also corrected and sealed in Subsection \ref{4s2.1}.
In addition, by providing a new proof, we improve the maximal
function characterizations of $\vh$ given in \cite[Theorem 3.1]{cgn17}.
The revised versions of \cite[Theorems 6.8 and 6.9]{hlyy} on the boundedness
of anisotropic $\bz$-order Calder\'{o}n--Zygmund operators from $\vh$
to itself [or to $\lv$] are also obtained.

Finally, we make some conventions on notation.
We always let
$\mathbb{N}:=\{1,2,\ldots\}$,
$\mathbb{Z}_+:=\{0\}\cup\mathbb{N}$
and $\vec0_n$ be the \emph{origin} of $\rn$.
For any multi-index
$\az:=(\az_1,\ldots,\az_n)\in(\mathbb{Z}_+)^n=:\mathbb{Z}_+^n$,
let
$|\az|:=\az_1+\cdots+\az_n$ and
$$\partial^{\az}:=\lf(\frac{\partial}{\partial x_1}\r)^{\az_1} \cdots
\lf(\frac{\partial}{\partial x_n}\r)^{\az_n}.$$
We denote by $C$ a \emph{positive constant}
which is independent of the main parameters,
but may vary from line to line.
The notation $f\ls g$ means $f\le Cg$ and, if $f\ls g\ls f$,
we then write $f\sim g$. We also use the following
convention: If $f\le Cg$ and $g=h$ or $g\le h$, we then write $f\ls g\sim h$
or $f\ls g\ls h$, \emph{rather than} $f\ls g=h$
or $f\ls g\le h$. For any $p\in[1,\fz]$, we denote by $p'$
its \emph{conjugate index}, namely, $1/p+1/p'=1$.
Moreover, if $\vec{p}:=(p_1,\ldots,p_n)\in[1,\fz]^n$, we denote by
$\vec{p}':=(p_1',\ldots,p_n')$ its \emph{conjugate index}, namely,
for any $i\in\{1,\ldots,n\}$, $$\frac1{p_i}+\frac1{p_i'}=1.$$ In addition,
for any set $E\subset\rn$, we denote by $E^\complement$ the
set $\rn\setminus E$, by ${\mathbf 1}_E$ its \emph{characteristic function}
and $|E|$ its \emph{n-dimensional Lebesgue measure}.
For any $\ell\in\mathbb{R}$, we denote by $\lfloor \ell\rfloor$
the \emph{largest integer not greater than $\ell$}. In
what follows, the \emph{symbol} $C^{\fz}(\rn)$
denotes the set of all \emph{infinitely differentiable functions} on $\rn$.

\section{(Weak) Lebesgue spaces with mixed norms}\label{sec2}

The aims of this section are twofold. The first one is devoted to
summarizing three kinds of (weak) Lebesgue spaces with mixed norms,
which include the mixed Lebesgue space $\lv$ with
$\vp\in(0,\fz]^n$, the iterated weak
Lebesgue space $L^{q_2,\fz}(L^{q_1,\fz})(\rn\times\rrm)$ with
$q_1,\ q_2\in(0,\fz)$ and the weak mixed-norm Lebesgue space
$L^{\vec q,\fz}(\rr^{2n})$ with $\vq\in(0,\fz]^2$ (see Subsection \ref{2s1} below),
then we further review some properties of these spaces as
well as their applications to the dual inequality of Stein type,
the Fefferman--Stein vector-valued inequality, the
boundedness of fractional integrals and geometric inequalities
(see Subsection \ref{2s3} below).
The second one is to recall an extended inequality about mixed Lebesgue
norms and the central Hardy--Littlewood maximal operator, and further
to provide a detailed proof for it (see Subsection \ref{2s2} below).

\subsection{(Weak) Lebesgue spaces with mixed norms}\label{2s1}

In this subsection, we first recall the definitions of three kinds of (weak)
Lebesgue spaces with mixed norms and then display some basic properties of them,
including the completeness, the corresponding H\"{o}lder inequalities
(see Subsection \ref{2s1.1} below) as well as interpolation theorems
(see Subsection \ref{2s1.2} below). To this end, we first present the
notion of mixed Lebesgue spaces $\lv$ with $\vp\in(0,\fz]^n$,
which was originally introduced in \cite{bp61}.

\begin{definition}\label{x2d1}
Let $\vp:=(p_1,\ldots,p_n)\in (0,\fz]^n$. The \emph{mixed Lebesgue space} $\lv$ is
defined to be the set of all measurable functions $f$ such that their quasi-norms
$$\|f\|_{\lv}:=\lf\{\int_{\mathbb R}\cdots\lf[\int_{\mathbb R}|f(x_1,\ldots,x_n)|^{p_1}\,dx_1\r]
^{\frac{p_2}{p_1}}\cdots\, dx_n\r\}^{\frac{1}{p_n}}<\fz$$
with the usual modifications made when $p_i=\fz$ for some $i\in \{1,\ldots,n\}$.
\end{definition}

\begin{remark}\label{x2r1}
Obviously, when $\vp:=(\overbrace{p,\ldots,p}^{n\ \rm times})$
with some $p\in(0,\fz]$, the space $\lv$ coincides
with the classical Lebesgue space $L^p(\rn)$ and,
in this case, they have the same quasi-norms.
\end{remark}

If $\vp \in [1,\fz]^n$, Benedek and Panzone proved
that the mixed Lebesgue space $\lv$ is complete in
\cite[p.\,304, Theorem 1.b)]{bp61}, which is stated as follows.

\begin{theorem}\label{x2t1}
Let $\vp \in [1,\fz]^n$, then $(\lv,\|\cdot\|_{\lv})$
becomes a Banach space.
\end{theorem}

Moreover, it was also shown by Benedek and Panzone in
\cite[p.\,304, Theorem 1.a)]{bp61} that the dual space
of $\lv$ with any given $\vp\in[1,\fz)$ is $L^{\vp'}(\rn)$ as follows.

\begin{theorem}\label{x2t18}
Let $\vp\in[1,\fz)^n$ and $\vp'$ denote the conjugate
exponent of $\vp$, namely, for any $i\in \{1,\ldots,n\}$,
$1/p_i+1/p_i'=1$. Then the dual space of $\lv$, denoted by
$(\lv)^*$, is $L^{\vp'}(\rn)$ in the following sense:
$J$ is a continuous linear functional
on $\lv$ if and only if there exists a uniquely
$h\in L^{\vp'}(\rn)$ such that, for any $f\in\lv$,
$$J(f)=\int_{\rn}f(x)h(x)\,dx$$
and $\|J\|_{(\lv)^\ast}=\|h\|_{L^{\vp'}(\rn)}$.
\end{theorem}

Now, we recall the notions of iterated weak Lebesgue spaces and
weak mixed-norm Lebesgue spaces given in \cite[Definition 1.1]{cs}.

\begin{definition}\label{x2d2}
Let $\vp:=(p_1,p_2)\in(0,\fz]^2$.
The \emph{iterated weak Lebesgue space}
$L^{p_2,\fz}(L^{p_1,\fz})(\rn\times\rrm)$
with $n,\ m\in\nn$ is
defined to be the set of all measurable functions $f$ such that
\begin{align*}
\|f\|_{L^{p_2,\fz}(L^{p_1,\fz})(\rn\times\rrm)}
&:=\sup_{\bz\in(0,\fz)}\bz\lf|\lf\{y\in\rrm:\
\sup_{\az\in(0,\fz)}\az\lf|\lf\{x\in\rn:\ \lf|f(x,y)\r|>\az\r\}\r|
^{1/p_1}>\bz\r\}\r|^{1/p_2}\\
&\ <\fz
\end{align*}
and the \emph{weak mixed-norm Lebesgue space} $L^{\vp,\fz}(\rn)$ is
defined to be the set of all measurable functions $g$ such that
\begin{align*}
\|g\|_{L^{\vp,\fz}(\rn)}:=\sup_{\az\in(0,\fz)}\az\lf\|\mathbf{1}_{\lf\{x\in\rn:\
|g(x)|>\az\r\}}\r\|_{\lv}<\fz.
\end{align*}
\end{definition}

The completeness of iterated weak Lebesgue spaces
and weak mixed-norm Lebesgue spaces was proved in
\cite[Theorem 2.10]{cs} as follows.

\begin{theorem}\label{x2t5}
Let $\vp:=(p_1,p_2)\in(0,\fz]^2$.
Then the weak spaces $L^{p_2,\fz}(L^{p_1,\fz})(\rn\times \rrm)$
and $L^{\vp,\fz}(\rn)$ are both complete.
\end{theorem}

The iterated weak Lebesgue space
and the weak mixed-norm Lebesgue space are two different spaces
and one of them can not cover each other, which was obtained
in \cite[Theorem 2.2]{cs}.

\begin{theorem}\label{x2t6}
Let $\vp:=(p_1,p_2)\in(0,\fz)^2$. Then, for any $x\in\rn$ and $y\in\rrm$,
one has
\begin{enumerate}
\item[{\rm(i)}] $L^{p_2,\fz}(L^{p_1,\fz})(\rn\times\rrm)\nsubseteq\mlv$.

\item[{\rm (ii)}]$\mlv\nsubseteq L^{p_2,\fz}(L^{p_1,\fz})(\rn\times\rrm)$.

\item[{\rm(iii)}] $L^{p_2,\fz}(L^{p_1,\fz})(\rn\times\rrm)\nsubseteq
L^{p_1,\fz}(L^{p_2,\fz})(\rrm\times\rn)$ .

\item[{\rm (iv)}]
$L^{p_1,\fz}(L^{p_2,\fz})(\rrm\times\rn)\nsubseteq
L^{p_2,\fz}(L^{p_1,\fz})(\rn\times\rrm)$.

\item[{\rm(v)}] $L^{\vp}(\rn\times\rrm)\subsetneq \mlv\bigcap
L^{p_2,\fz}(L^{p_1,\fz})(\rn\times\rrm)$.
\end{enumerate}
\end{theorem}

For any given $\vp:=(p_1,p_2)\in(0,\fz)^2$, the
\emph{space $L^{p_2,\fz}(L^{p_1})(\rn\times \rrm)$} is defined to be the set of
all measurable functions $f$ such that
$$\|f\|_{L^{p_2,\fz}(L^{p_1})(\rn\times\rrm)}
:=\sup_{\az\in(0,\fz)}\az\lf|\lf\{y\in\rrm:\
\lf\|f(\cdot,y)\r\|_{L^{p_1}(\rn)}>\az\r\}\r|^{1/p_2}<\fz$$
and the
\emph{space $L^{p_2}(L^{p_1,\fz})(\rn\times \rrm)$} is defined to be the set of
all measurable functions $g$ such that
$$\|g\|_{L^{p_2}(L^{p_1,\fz})(\rn\times\rrm)}
:=\lf\|\sup_{\az\in(0,\fz)}\az\lf|\lf\{x\in\rn:\
\lf|g(x,\cdot)\r|>\az\r\}\r|
^{1/p_1}\r\|_{L^{p_2}(\rrm)}<\fz.$$
The relations among these three mixed norms of $L^{p_2,\fz}(L^{p_1})(\rn\times \rrm)$,
$L^{p_2}(L^{p_1,\fz})(\rn\times \rrm)$ and $\mlv$ are as follows,
which is just \cite[Theorem 2.4]{cs}.

\begin{theorem}\label{x2t8}
Let $\vp:=(p_1,p_2)\in(0,\fz)^2$. Then
\begin{enumerate}
\item[{\rm(i)}] $L^{p_2}(L^{p_1,\fz})(\rn\times\rrm)\subset \mlv$
and, for any measurable function $f$ defined on $\rn\times\rrm$,
$$\|f\|_{\mlv}\le \|f\|_{L^{p_2}(L^{p_1,\fz})(\rn\times\rrm)}.$$
\item[{\rm(ii)}] $L^{p_2,\fz}(L^{p_1})(\rn\times\rrm)\nsubseteq
\mlv$ and $\mlv\nsubseteq L^{p_2,\fz}(L^{p_1})(\rn\times\rrm)$.
\end{enumerate}
\end{theorem}

\subsubsection{H\"{o}lder inequalities}\label{2s1.1}

It is known that the H\"{o}lder inequality holds true for
classical Lebesgue spaces and weak Lebesgue spaces.
We now discuss the H\"{o}lder inequality
on the above three kinds of (weak) Lebesgue spaces.
First, for mixed Lebesgue spaces $\lv$, we have the following
conclusion, which was obtained by Benedek and Panzone in
\cite[(1)]{bp61}.

\begin{theorem}\label{x2t7}
Let $\vp\in [1,\fz]^n$. Then, for any measurable function
$f$ and $g$, one has
\begin{align*}
\lf\|fg\r\|_{L^1(\rn)}\le \|f\|_{\lv}\|g\|_{L^{\vp'}(\rn)},
\end{align*}
where $\vp'$ denotes the conjugate vector of $\vp$,
namely, for any $i\in \{1,\ldots,n\}$, $1/p_i+1/p_i'=1$.
\end{theorem}

Let $\vp:=(p_1,\ldots,p_n)\in(0,\fz]^n$,
$\vq:=(q_1,\ldots,q_n)\in(0,\fz]^n$ and
$\vec r:=(r_1,\ldots,r_n)\in(0,\fz]^n$ satisfy that,
for any $i\in \{1,\ldots,n\}$,
$$\frac1{r_i}=\frac1{p_i}+\frac1{q_i}.$$
Then, via the inequality in Theorem \ref{x2t7}, we find that,
for any measurable functions $f$ and $g$,
\begin{align}\label{x2e12}
\lf\|fg\r\|_{L^{\vec r}(\rn)}\le \|f\|_{\lv}\|g\|_{L^{\vq}(\rn)}.
\end{align}

Recall that, applying the H\"{o}lder inequality for
classical weak Lebesgue spaces, the corresponding
extended inequality \eqref{x2e12} for iterated weak Lebesgue
spaces was shown by Chen and Sun in \cite[Theorem 2.16]{cs}.

\begin{theorem}\label{x2t8'}
Let $\vp:=(p_1,p_2)$, $\vq:=(q_1,q_2)$ and $\vec r:=(r_1,r_2)$
satisfy that, for any $i\in \{1,2\}$, $p_i,\ q_i,\ r_i\in(0,\fz]$
and $$\frac1{r_i}=\frac1{p_i}+\frac1{q_i}.$$
Then there exists a positive constant $C$,
depending on $\vp,\ \vq$ and $\vec r$, such that,
for any measurable functions $f$ and $g$
defined on $\rn\times\rrm$,
$$\|fg\|_{L^{r_2,\fz}(L^{r_1,\fz})(\rn\times\rrm)}
\le C\|f\|_{L^{p_2,\fz}(L^{p_1,\fz})(\rn\times\rrm)}
\|g\|_{L^{q_2,\fz}(L^{q_1,\fz})(\rn\times\rrm)},$$
where $C:=\prod_{i=1}^{2}(p_i/r_i)^{1/p_i}(q_i/r_i)^{1/q_i}$
if $p_i,\ q_i\neq\fz$ for any $i\in \{1,2\}$ and $C:=1$
if $p_i=\fz$ or $q_i=\fz$ for some $i\in \{1,2\}$.
\end{theorem}

However, for weak mixed-norm Lebesgue spaces, the corresponding
H\"{o}lder inequality usually does not hold true. Particularly,
we state a result in \cite[Theorem 2.17]{cs} as follows.

\begin{theorem}\label{x2t9}
Let $\vp:=(p_1,p_2)$, $\vec q:=(q_1,q_2)$ and $\vec r:=(r_1,r_2)$
satisfy that, for any $i\in \{1,2\}$, $p_i,\ q_i,\ r_i\in(0,\fz]$ and
$$\frac1{r_i}=\frac1{p_i}+\frac1{q_i}.$$
Then there exists a positive constant $C$,
depending on $\vp,\ \vq$ and $\vec r$,
such that, for any measurable functions $f$ and $g$ defined on $\rn$,
$$\|fg\|_{L^{\vec r,\fz}(\rn)}
\le C\|f\|_{L^{\vp,\fz}(\rn)}
\|g\|_{L^{\vec q,\fz}(\rn)}$$
holds true if and only if $p_1q_2=p_2q_1$. Moreover,
when this condition holds true, then
\begin{align*}
C:=\left\{
\begin{array}{cl}
\max\lf\{1,2^{1/r_1-1/r_2}\r\}\frac{p_2^{1/p_2} q_2^{1/q_2}}{r_2^{1/r_2}}
&{\rm when}\hspace{0.5cm} p_1,\ p_2,\ q_1,\ q_2\in(0,\fz),\\
\hspace{-0.2cm}\max\lf\{1,2^{1/r_1-1}\r\}\frac{p_1^{r_1/p_1} q_1^{r_1/q_1}}{r_1}
&{\rm when}\hspace{0.5cm} p_2=q_2=\fz,~p_1,\ q_1\in(0,\fz),\\
\hspace{-3cm}\frac{p_2^{1/p_2} q_2^{1/q_2}}{r_2^{1/r_2}}
&{\rm when}\hspace{0.5cm} p_1=q_1=\fz,~p_2,\ q_2\in(0,\fz),\\
\hspace{-4cm}1
&{\rm when}\hspace{0.5cm} \vp=(\fz,\fz)~{\rm or}~\vec{q}=(\fz,\fz).\\
\end{array}\r.
\end{align*}
\end{theorem}

\subsubsection{Interpolations}\label{2s1.2}

Interpolation theorem has proved an important
and useful tool in many applications due to the fact that
it is applicable to allow one to pass from hypotheses for
certain exponents $p$ (for instance $p=2$ and $p=\fz$)
to conclusions involving a range of $p$ [for instance $p\in(2,\fz)$].
In this subsection, we mainly review some interpolation
results about the above (weak) Lebesgue spaces with mixed norms.

Let the \emph{symbol} $V(\rn)$ denote the set of all
simple functions on $\rn$. Note that $V(\rn)\subset\lv$
and $V(\rn)$ is dense in $\lv$ for any given $\vp\in[1,\fz)^n$
(see \cite[p.\,313]{bp61}). We first display the Riesz--Thorin
interpolation theorem on mixed Lebesgue spaces, which was
established by Benedek and Panzone in \cite[p.\,316, Theorem 2]{bp61}.

\begin{theorem}\label{x2t10}
Let $\theta\in(0,1)$, $\vp:=(p_1,\ldots,p_n)$, $\vq:=(q_1,\ldots,q_n)$,
$\vp^{(j)}:=(p_1^{(j)},\ldots,q_n^{(j)})$ and
$\vq^{(j)}:=(q_1^{(j)},\ldots,q_n^{(j)})$ for any $j\in\{1,2\}$
satisfy that, for any $i\in\{1,\ldots,n\}$ and $j\in\{1,2\}$,
$p_i,\ q_i,\ p^{(j)}_i,\ q^{(j)}_i\in [1,\fz]$,
$$\frac1{p_i}=\frac{\theta}{p_i^{(1)}}+\frac{1-\theta}{p_i^{(2)}}
\quad{\rm and}\quad
\frac1{q_i}=\frac{\theta}{q_i^{(1)}}+\frac{1-\theta}{q_i^{(2)}}.$$
Let $T$ be a linear operator, defined on $V(\rn)$, satisfy that
there exist two positive constants $M_1$ and $M_2$ such that,
for any $j\in\{1,2\}$ and $f\in V(\rn)$,
$$\lf\|T(f)\r\|_{L^{\vp^{(j)}}(\rn)}\le M_{j}\lf\|f\r\|_{L^{\vq^{(j)}}(\rn)}.$$
Then, for any $f\in V(\rn)$,
$$\lf\|T(f)\r\|_{L^{\vp}(\rn)}\le M_1^{1-\theta}M_2^\theta\lf\|f\r\|_{L^{\vq}(\rn)}.$$
Furthermore, if $\vq\in[1,\fz)^n$, the linear operator $T$ can be uniquely
extended to the space $L^{\vq}(\rn)$.
\end{theorem}

\begin{remark}
We point out that Theorem \ref{x2t10} plays
a curial role in the proof of Theorem \ref{x2t2} below.
\end{remark}

For weak mixed-norm Lebesgue spaces, Chen and Sun
obtained the following interpolation theorem
in \cite[Theorem 2.21]{cs}.

\begin{theorem}\label{x2t11}
Let $\vp:=(p_1,p_2)$, $\vec q:=(q_1,q_2)$ and $\vec r:=(r_1,r_2)$
satisfy that, for any $i\in\{1,2\}$, $p_i,\ q_i,\ r_i\in(0,\fz)$
and $$\frac 1 {r_i}=\frac{\theta}{p_i}+\frac{1-\theta}{q_i}$$
with constant $\theta\in(0,1)$.
Then, for any measurable function $f$
defined on $\rn\times\rrm$, one has
$$\|f\|_{L^{\vec r,\fz}(\rn\times\rrm)}
\le \|f\|_{L^{\vp,\fz}(\rn\times\rrm)}^{\theta}
\|f\|_{L^{\vec q,\fz}(\rn\times\rrm)}^{1-\theta}$$
and
$$\|f\|_{L^{\vec r}(\rn\times\rrm)}
\le \lf(\frac{r_1}{r_1-p_1}+\frac{r_1}{q_1-r_1}\r)^{1/r_1}
\|f\|_{L^{p_2}(L^{p_1,\fz})(\rn\times\rrm)}^{\theta}
\|f\|_{L^{q_2}(L^{q_1,\fz})(\rn\times\rrm)}^{1-\theta}.$$
However, if $1/p_1+1/p_2=1/q_1+1/q_2$, then, for any multiple
index $\vec r\in(0,\fz)^2$, $L^{\vp,\fz}(\rn\times\rrm)\cap
L^{\vec q,\fz}(\rn\times\rrm)\nsubseteq L^{\vec r}(\rn\times\rrm)$.
\end{theorem}

Moreover, in \cite[Theorem 2.22]{cs}, Chen and Sun provided
the following interpolation theorem for iterated weak Lebesgue spaces.

\begin{theorem}\label{x2t12}
Let $\vec r:=(r_1,r_2)\in(0,\fz)^2$ and
$p_1,\ q_1,\ p_{2,1},\ p_{2,2},\ q_{2,1},\ q_{2,2}\in(0,\fz)$ satisfy that
$$\frac 1 {r_1}=\frac{\theta}{p_1}+\frac{1-\theta}{q_1}$$
and
$$\frac 1 {r_2}=\frac{\theta\xi}{p_{2,1}}+\frac{(1-\theta)\xi}{p_{2,2}}
+\frac{\theta(1-\xi)}{q_{2,1}}+\frac{(1-\theta)(1-\xi)}{q_{2,2}},$$
where the constants $\theta,\ \xi\in(0,1)$.
Then there exists a positive constant $C$ such that,
for any measurable function $f$
defined on $\rn\times\rrm$,
\begin{align*}
\|f\|_{L^{\vec r}(\rn\times\rrm)}
&\le C\|f\|_{L^{p_{2,1},\fz}(L^{p_1,\fz})(\rn\times\rrm)}^{\xi\theta}
\|f\|_{L^{p_{2,2},\fz}(L^{q_1,\fz})(\rn\times\rrm)}^{\xi(1-\theta)}\\
&\hs\hs\times\|f\|_{L^{q_{2,1},\fz}(L^{p_1,\fz})(\rn\times\rrm)}^{(1-\xi)\theta}
\|f\|_{L^{q_{2,2},\fz}(L^{q_1,\fz})(\rn\times\rrm)}^{(1-\xi)(1-\theta)}.
\end{align*}
\end{theorem}

\subsection{An extended inequality on the central
Hardy--Littlewood maximal operator}\label{2s2}

In this subsection, we first recall an extended inequality about
mixed Lebesgue norms and the central Hardy--Littlewood maximal operator,
which was stated by Bagby \cite{b75} without giving a proof. For its importance
and also the convenience of the reader, we provide a detailed proof for it in this subsection.

To begin with, we state this inequality as follows.
In what follows, the \emph{symbol} $L_{\rm loc}^1(E)$
denotes the collection of all locally integrable functions on set $E$.

\begin{theorem}\label{x2t2}
Let $m\in \zz_+$ and $f\in L_{\rm loc}^1(\rr^{n}\times\rr^{m})$.
For any $s\in \rr^{n}$ and $t\in\rr^{m}$,
define $$f^*(s,t):=\sup_{r\in (0,\fz)}\frac{1}{|B(s,r)|}
\int_{B(s,r)}|f(y,t)|\,dy,$$
here and thereafter, for any $s\in \rn$ and $r\in(0,\fz)$,
$B(s,r):=\{z\in\rn:\ |z-s|<r\}$.
If $m\in\nn$, then, for any given $\vec{p}_m
:=(p_1,\ldots,p_m)\in(1,\fz)^m$ and any $s\in\rn$, define
\begin{align*}
T_{L^{\vec{p}_{m}}(\rr^{m})}(f)(s)
:=\left\{\int_{\rr}\cdots\left[\int_{\rr}
|f(s,t_1\ldots,t_{m})|^{p_1}\,dt_1\right]^{\frac{p_2}{p_1}}
\cdots \,dt_{m}\right\}^{\frac{1}{p_{m}}},
\end{align*}
where, if $m\equiv 0$, then let $\rr^0:=\emptyset$ and
$T_{L^{\vec{p}_m}(\rr^m)}(f)(s):=|f(s)|$
for any $s\in\rn$.
Thus, for any given $q\in (1,\fz)$, there exists a positive constant $C$,
depending on $q$, such that
\begin{align}\label{x2e2}
\int_{\rr^{n}}\lf[T_{L^{\vec{p}_{m}}(\rr^{m})}(f^*)(s)\r]^q \,ds
\le C \int_{\rr^{n}}\lf[T_{L^{\vec{p}_{m}}(\rr^{m})}(f)(s)\r]^q \,ds.
\end{align}
\end{theorem}

\begin{remark}\label{x2r2}
We should point out that Theorem \ref{x2t2} plays a key role in
the proofs of both the boundedness of the Hardy--Littlewood maximal
operator on mixed Lebesgue spaces
(see \cite[Lemma 3.5]{hlyy} or \cite[Theorem 1.2]{tn}) and
Fefferman--Stein vector-valued inequality on mixed Lebesgue spaces
(see \cite[p.\,679]{js08} or \cite[Theorem 1.7]{tn}), which are known to be fundamental
tools in developing a real-variable theory of related function spaces.
\end{remark}

Recall that the
\emph{centered Hardy--Littlewood maximal function}
$f^*$ of $f\in L^1_{{\rm loc}}(\rn)$ is defined by setting,
for any $x\in\rn$,
\begin{align}\label{x2e1}
f^*(x):=\sup_{r\in(0,\fz)}
\frac1{|B(x,r)|}\int_{B(x,r)}|f(y)|\,dy.
\end{align}

The succeeding two conclusions are just, respectively,
\cite[p.\,304, Theorem 2]{bp61} and
\cite[Lemma 1]{fs71}, which are used later to show Theorem \ref{x2t2}.

\begin{theorem}\label{x2t3}
Let $\vp\in[1,\fz]^n$ and $f$ be a measurable function on $\rn$.
Then
\begin{align*}
\|f\|_{\lv}=\sup_{g\in U_{\vp'}}\lf|\int_{\rn}f(x)g(x)\,dx\r|,
\end{align*}
where $\vp'$ denotes the conjugate vector of $\vp$,
namely, for any $i\in \{1,\ldots,n\}$, $1/p_i+1/p_i'=1$
and $U_{\vp'}$ the unit sphere of $L^{\vp'}(\rn)$.
\end{theorem}

\begin{theorem}\label{x2t4}
For any $r\in(1,\fz)$, there exists a positive constant $C$,
depending only on $r$, such that, for any positive real-valued
functions $f$ and $\phi$ on $\rn$,
\begin{align*}
\int_{\rn}\lf[f^\ast(x)\r]^r \phi(x)\,dx
\le C\int_{\rn}\lf[f(x)\r]^r \phi^\ast(x)\,dx,
\end{align*}
where $f^\ast$ and $\phi^\ast$ respectively denote the centered
Hardy--Littlewood maximal functions of $f$ and $\phi$ as in \eqref{x2e1}.
\end{theorem}

Via Theorems \ref{x2t3} and \ref{x2t4}, we now show Theorem \ref{x2t2}.

\begin{proof}[Proof of Theorem \ref{x2t2}]
We prove this theorem by induction. Let $m\in \zz_+$ and
$f\in L_{\rm loc}^1(\rr^{n}\times\rr^{m})$.
We perform induction on $m$. If $m:=0$, then
$f\in L_{\rm loc}^1(\rr^{n})$ and the desired inequality
\eqref{x2e2} becomes, for any given $q\in (1,\fz)$,
$$\int_{\rr^{n}}\lf|f^\ast(s)\r|^q \,ds
\ls \int_{\rr^{n}}\lf|f(s)\r|^q \,ds.$$
This is just the well-known boundedness of the centered Hardy--Littlewood
maximal operator on $L^q(\rr^{n})$ with any given $q\in (1,\fz)$
(see, for instance, \cite[p.\,13, Theorem 1.(c)]{s93}).

Now assume \eqref{x2e2} holds true for $m:=k$
with some fixed $k\in\nn$, namely,
for any given $\vec{p}_{k}:=(p_1,\ldots,p_{k})\in(1,\fz)^{k}$
and $q\in (1,\fz)$,
\begin{align}\label{x2e3}
\int_{\rr^{n}}\lf[T_{L^{\vec{p}_{k}}(\rr^{k})}(f^\ast)(s)\r]^q \,ds
\ls \int_{\rr^{n}}\lf[T_{L^{\vec{p}_{k}}(\rr^{k})}(f)(s)\r]^q \,ds.
\end{align}
To complete the proof of Theorem \ref{x2t2}, it suffices to show
that, for $m:=k+1$, \eqref{x2e2} also holds true. To this end,
we only need to prove, for any given $\vec{p}_k\in(1,\fz)^k$,
$r\in(1,\fz]$ and $q\in(1,\fz)$,
\begin{align}\label{x2e4}
&\int_{\rr^{n}}\lf\{T_{L^{\vec{p}_{k}}(\rr^{k})}\lf(\lf[\int_{\rr}
\lf|f^\ast(s,y)\r|^r\,dy\r]^{1/r}\r)\r\}^q \,ds\noz\\
&\hs\hs\ls \int_{\rr^{n}}\lf\{T_{L^{\vec{p}_{k}}(\rr^{k})}\lf(\lf[\int_{\rr}
\lf|f(s,y)\r|^r\,dy\r]^{1/r}\r)\r\}^q \,ds
\end{align}
with the usual modifications made when $r=\fz$.
Indeed, for any given $r\in(1,\fz)$, any $s\in\rn$ and $\wz{t}\in\rr^k$, let
$$J_r(f)(s,\wz t):=\lf[\int_{\rr}
\lf|f(s,y,\wz t)\r|^r\,dy\r]^{1/r}$$
and $J_{\fz}(f)(s,\wz t):=\esup_{y\in\rr}\lf|f(s,y,\wz t)\r|$.
Then we can rewrite \eqref{x2e4} simply as
\begin{align}\label{x2e5}
\int_{\rr^{n}}\lf[T_{L^{\vec{p}_{k}}(\rr^{k})}\lf(J_r\lf(f^\ast\r)\r)(s)\r]^q \,ds
\ls \int_{\rr^{n}}\lf[T_{L^{\vec{p}_{k}}(\rr^{k})}\lf(J_r(f)\r)(s)\r]^q \,ds.
\end{align}

We now prove \eqref{x2e5} by three steps.

\emph{Step 1.} In this step, we show that \eqref{x2e5} holds true
for  $r=\fz$. To do this, first notice that,
for any $s\in\rn$, $\wz{t}\in\rr^k$ and almost every $y\in\rr$,
$$\lf|f(s,y,\wz t)\r|\le J_{\fz}(f)(s,\wz t).$$
Thus, for any $s\in\rn$, $\wz{t}\in\rr^k$ and almost every $y\in\rr$, we have
$$f^\ast(s,y,\wz t)\le \lf[J_{\fz}(f)\r]^\ast(s,\wz t)$$
and hence
\begin{align*}
J_{\fz}(f^\ast)(s,\wz t)\le\lf[J_{\fz}(f)\r]^\ast(s,\wz t).
\end{align*}
From this and \eqref{x2e3}, it follows that,
for any given $\vec{p}_k\in(1,\fz)^k$ and $q\in(1,\fz)$,
\begin{align*}
\int_{\rr^{n}}\lf[T_{L^{\vec{p}_{k}}(\rr^{k})}
\lf(J_{\fz}\lf(f^\ast\r)\r)(s)\r]^q \,ds
&\le\int_{\rr^{n}}\lf[T_{L^{\vec{p}_{k}}(\rr^{k})}
\lf(\lf[J_{\fz}(f)\r]^\ast\r)(s)\r]^q \,ds\\
&\ls\int_{\rr^{n}}\lf[T_{L^{\vec{p}_{k}}(\rr^{k})}
\lf(J_{\fz}(f)\r)(s)\r]^q \,ds,
\end{align*}
which implies that \eqref{x2e5} holds true for $r=\fz$.

\emph{Step 2.} In this step, we prove that \eqref{x2e5} holds true
for any given $r\in(1,\min\{q,p_1,\ldots,p_k\})$.
Indeed, note that, for any $s\in\rn$,
\begin{align*}
T_{L^{\vec{p}_{k}}(\rr^{k})}\lf(J_r\lf(f^\ast\r)\r)(s)
&=T_{L^{\vec{p}_{k}}(\rr^{k})}\lf(\lf\{\int_{\rr}
\lf[f^\ast(s,y)\r]^r\,dy\r\}^{1/r}\r)\\
&=\lf[T_{L^{\vec{p}_{k}/r}(\rr^{k})}\lf(\int_{\rr}
\lf[f^\ast(s,y)\r]^r\,dy\r)\r]^{1/r},
\end{align*}
where $\vec{p}_{k}/r:=(p_1/r,\ldots,p_{k}/r)\in(1,\fz)^{k}$.
Therefore, for any given $q\in (1,\fz)$, by the fact that
$r\in(1,\min\{q,p_1,\ldots,p_k\})$
and Theorem \ref{x2t3}, we have
\begin{align}\label{x2e6}
\int_{\rn}\lf[T_{L^{\vec{p}_{k}}(\rr^{k})}
\lf(J_r\lf(f^\ast\r)\r)(s)\r]^q\,ds
&=\int_{\rn}\lf[T_{L^{\vec{p}_{k}/r}(\rr^{k})}\lf(\int_{\rr}
\lf[f^\ast(s,y)\r]^r\,dy\r)\r]^{q/r}\,ds\noz\\
&=\lf\|\lf\|\int_{\rr}\lf[f^\ast(\cdot,y,\cdot)\r]^r\,dy\r
\|_{L^{\vec{p}_{k}/r}(\rr^{k})}\r\|_{L^{q/r}(\rr^n)}^{q/r}\noz\\
&=\sup_{\phi}\lf|\int_{\rn}\int_{\rr^k}\lf\{\int_{\rr}
\lf[f^\ast(s,y,\wz t)\r]^r\,dy\r\}
\phi(s,\wz t)\,d\wz t\,ds\r|^{q/r},
\end{align}
where the supremum is taken over all $\phi$ belonging to
$$\lf\{\zeta\in L^{(q/r)'}(L^{(\vp_k/r)'})(\rr^{n+k}):\
\int_{\rn}\lf[T_{L^{(\vec{p}_{k}/r)'}
(\rr^{k})}(\zeta)(s)\r]^{\frac q{q-r}}\,ds=1\r\}.$$
This, together with the Tonelli theorem and Theorem \ref{x2t4},
implies that
\begin{align*}
&\lf|\int_{\rn}\int_{\rr^k}\lf\{\int_{\rr}
\lf[f^\ast(s,y,\wz t)\r]^r\,dy\r\}
\phi(s,\wz t)\,d\wz t\,ds\r|\\
&\hs\hs\le\int_{\rr^k}\int_{\rr}\int_{\rn}
\lf[f^\ast(s,y,\wz t)\r]^r
\lf|\phi(s,\wz t)\r|\,ds\,dy\,d\wz t\ls\int_{\rr^k}\int_{\rr}\int_{\rn}
\lf|f(s,y,\wz t)\r|^r
\phi^\ast(s,\wz t)\,ds\,dy\,d\wz t.
\end{align*}
From this, the Tonelli theorem again, the fact that
$r\in(1,\min\{q,p_1,\ldots,p_k\})$, Theorem \ref{x2t7} and
the H\"{o}lder inequality, it follows that
\begin{align*}
&\lf|\int_{\rn}\int_{\rr^k}\lf\{\int_{\rr}
\lf[f^\ast(s,y,\wz t)\r]^r\,dy\r\}
\phi(s,\wz t)\,d\wz t\,ds\r|\\
&\hs\hs\ls \int_{\rn}\int_{\rr^k}\lf[\int_{\rr}
\lf|f(s,y,\wz t)\r|^r\,dy\r]
\phi^\ast(s,\wz t)\,d\wz t\,ds\\
&\hs\hs\ls \int_{\rn}T_{L^{\vec{p}_{k}/r}(\rr^{k})}
\lf(\int_{\rr}\lf|f(s,y)\r|^r\,dy\r)
T_{L^{(\vec{p}_{k}/r)'}(\rr^{k})}\lf(\phi^\ast\r)(s)\,ds\\
&\hs\hs\ls \lf\{\int_{\rn}\lf[T_{L^{\vec{p}_{k}/r}(\rr^{k})}\lf(\int_{\rr}
\lf|f(s,y)\r|^r\,dy\r)\r]^{q/r}\,ds\r\}^{r/q}\\
&\hs\hs\hs\times\lf\{\int_{\rn}\lf[T_{L^{(\vec{p}_{k}/r)'}(\rr^{k})}
\lf(\phi^\ast\r)(s)\r]^{q/{(q-r)}}\,ds\r\}^{1-r/q},
\end{align*}
which, combined with \eqref{x2e3} and the fact that
$\int_{\rn}\lf[T_{L^{(\vec{p}_{k}/r)'}
(\rr^{k})}(\phi)(s)\r]^{\frac q{q-r}}\,ds=1$,
further implies that
\begin{align*}
&\lf|\int_{\rn}\int_{\rr^k}\lf\{\int_{\rr}
\lf[f^\ast(s,y,\wz t)\r]^r\,dy\r\}
\phi(s,\wz t)\,d\wz t\,ds\r|\\
&\hs\hs\ls \lf\{\int_{\rn}\lf[T_{L^{\vec{p}_{k}/r}(\rr^{k})}\lf(\int_{\rr}
\lf|f(s,y)\r|^r\,dy\r)\r]^{q/r}\,ds\r\}^{r/q}\\
&\hs\hs\hs\times\lf\{\int_{\rn}\lf[T_{L^{(\vec{p}_{k}/r)'}(\rr^{k})}
\lf(\phi\r)(s)\r]^{q/{(q-r)}}\,ds\r\}^{1-r/q}\\
&\hs\hs\ls \lf\{\int_{\rn}\lf[T_{L^{\vec{p}_{k}/r}(\rr^{k})}\lf(\int_{\rr}
\lf|f(s,y)\r|^r\,dy\r)\r]^{q/r}\,ds\r\}^{r/q}\\
&\hs\hs\sim \lf\{\int_{\rn}\lf[T_{L^{\vec{p}_{k}}(\rr^{k})}
\lf(J_r(f)\r)(s)\r]^{q}\,ds\r\}^{r/q}.
\end{align*}
By this and \eqref{x2e6}, we conclude that
\eqref{x2e5} holds true for any given
$r\in(1,\min\{q,p_1,\ldots,p_k\})$.

\emph{Step 3.} In this step, based on the obtained results in
Steps 1 and 2 above, we complete the proof of \eqref{x2e5}
via an interpolation procedure. To this end, recall that
Theorem \ref{x2t10}, the Riesz--Thorin interpolation
theorem on mixed Lebesgue spaces, has been established
by Benedek and Panzone in \cite[p.\,316, Theorem 2]{bp61}.
However, this interpolation theorem is
only applicable to linear operators and, obviously,
the Hardy--Littlewood maximal operator as in \eqref{x2e1} is only sublinear.
Thus, for any $i\in\zz$, we define a linear operator $\Gamma$ by setting,
for any $f\in L_{\rm loc}^1(\rn\times\rr\times\rr^k)$, $s\in\rn$, $y\in\rr$
and $\wz{t}\in\rr^k$,
\begin{align*}
\lf(\Gamma (f)\r)_i(s,y,\wz t):=\frac{1}{|B(s,2^i)|}
\int_{B(s,2^i)}f(z,y,\wz t)\,dz,
\end{align*}
where, for any $s\in \rr^{n}$ and $i\in\zz$,
$B(s,2^i):=\{z\in\rr^{n}:\ |z-s|<2^i\}$,
and let
\begin{align*}
I_{\fz}\lf[\{\Gamma (f)\}\r](s,y,\wz t)
:=\sup_{i\in\zz}\lf|\lf(\Gamma f\r)_i(s,y,\wz t)\r|.
\end{align*}
Then it is easy to see that,
for any $s\in\rn$, $y\in\rr$ and $\wz{t}\in\rr^k$,
\begin{align}\label{x2e7}
I_{\fz}\lf[\{\Gamma (f)\}\r](s,y,\wz t)
\le f^\ast(s,y,\wz t).
\end{align}
In addition, notice that, for any $r\in(0,\fz)$,
there exists some $i_r\in\zz$ such that $r\in [2^{i_r-1},2^{i_r})$.
Therefore, for any $r\in(0,\fz)$, $s\in\rn$, $y\in\rr$ and $\wz{t}\in\rr^k$,
\begin{align*}
\frac{1}{|B(s,r)|}\int_{B(s,r)}\lf|f(z,y,\wz t)\r|\,dz
&\le \frac{1}{ r^n\upsilon_n}\int_{B(s,2^{i_r})}\lf|f(z,y,\wz t)\r|\,dz\le \frac{1}{2^{(i_r-1)n}\upsilon_n }\int_{B(s,2^{i_r})}\lf|f(z,y,\wz t)\r|\,dz\\
&=2^n I_{\fz}\lf[\{\Gamma (f)\}\r](s,y,\wz t),
\end{align*}
here and thereafter, $\upsilon_n$ denotes the Lebesgue
measure of the unit ball on $\rn$, namely, $\upsilon_n
:=|B(\vec{0}_n,1)|$ with $B(\vec{0}_n,1):=\{y\in\rn:\
|y|<1\}$.
Thus, for any $s\in\rn$, $y\in\rr$ and $\wz{t}\in\rr^k$,
$$f^\ast(s,y,\wz t)\le 2^n I_{\fz}\lf[\{\Gamma (f)\}\r](s,y,\wz t),$$
which, together with \eqref{x2e7}, further implies that
\begin{align*}
I_{\fz}\lf[\{\Gamma (f)\}\r](s,y,\wz t)
\sim f^\ast(s,y,\wz t).
\end{align*}
Thus, \eqref{x2e5} is equivalent to, for any given $\vec{p}_k\in(1,\fz)^k$,
$r\in(1,\fz]$ and $q\in(1,\fz)$,
\begin{align}\label{x2e15}
\int_{\rr^{n}}\lf[T_{L^{\vec{p}_{k}}(\rr^{k})}
\lf(J_r \lf(I_{\fz}\lf[\{\Gamma (f)\}\r]\r)\r)(s)\r]^q\,ds
\ls \int_{\rr^{n}}\lf[T_{L^{\vec{p}_{k}}(\rr^{k})}\lf(J_r(f)\r)(s)\r]^q\,ds.
\end{align}
For any given $\vec{p}_k\in(1,\fz)^k$,
$r\in(1,\fz]$ and $q\in(1,\fz)$, let
$$\lf\|\Gamma (|f|)\r\|_{\vec {Q}_r(\vec{p}_k,q)}:=\int_{\rr^{n}}\lf[T_{L^{\vec{p}_{k}}(\rr^{k})}
\lf(J_r \lf(I_{\fz}\lf[\{\Gamma (f)\}\r]\r)\r)(s)\r]^q\,ds$$
and
$$\lf\|f\r\|_{\vec {V}_r(\vec{p}_k,q)}
:=\int_{\rr^{n}}\lf[T_{L^{\vec{p}_{k}}(\rr^{k})}\lf(J_r(f)\r)(s)\r]^q\,ds.$$
Then \eqref{x2e15} becomes, for any given $\vec{p}_k\in(1,\fz)^k$,
$r\in(1,\fz]$ and $q\in(1,\fz)$,
$$\lf\|\Gamma (|f|)\r\|_{\vec {Q}_r(\vec{p}_k,q)}\ls \lf\|f\r\|_{\vec {V}_r(\vec{p}_k,q)}.$$

On another hand, by the conclusions of Steps 1 and 2, we know that,
for any given $\vec{p}_k\in(1,\fz)^k$ and $q\in(1,\fz)$,
$$\lf\|\Gamma (|f|)\r\|_{\vec {Q}_{\fz}(\vec{p}_k,q)}
\ls \lf\|f\r\|_{\vec {V}_{\fz}(\vec{p}_k,q)}$$
and, for any given $r\in(1,\min\{q,p_1,\ldots,p_k\})$,
\begin{align*}
\lf\|\Gamma (|f|)\r\|_{\vec {Q}_r(\vec{p}_k,q)}
\ls \lf\|f\r\|_{\vec {V}_r(\vec{p}_k,q)},
\end{align*}
which, combined with Theorem \ref{x2t10}, further implies
that, for any given $r\in(1,\fz]$,
$$\lf\|\Gamma (|f|)\r\|_{\vec {Q}_r(\vec{p}_k,q)}
\ls \lf\|f\r\|_{\vec {V}_r(\vec{p}_k,q)}.$$
This implies that \eqref{x2e15} holds true and
hence finishes the proof of Theorem \ref{x2t2}.
\end{proof}

\subsection{Applications}\label{2s3}

This subsection is devoted to a survey of some applications
which include the dual inequality of Stein type, the
Fefferman--Stein vector-valued inequality on mixed
Lebesgue spaces $\lv$ proved by Nogayama in \cite{tn} as well as
the boundedness of fractional integrals and geometric inequalities
on iterated weak Lebesgue spaces
$L^{q_2,\fz}(L^{q_1,\fz})(\rr^{2n})$
and weak mixed-norm Lebesgue spaces $L^{\vec q,\fz}(\rr^{2n})$
obtained by Chen and Sun in \cite{cs}.
For this purpose, we first present the notion of
iterated maximal operators.

For any
$k\in\{1,\ldots,n\}$, the \emph{maximal function} $M_k(f)$
of any $f\in L_{\rm loc}^1(\rn)$ for the $k$th variable is defined by setting,
for any $x:=(x_1,\ldots,x_n)\in \rn$,
\begin{align}\label{x2e14}
M_k(f)(x):=\sup_{I\in \mathbb{I}_{x_k}}\frac{1}{|I|}\int_I
|f(x_1,\ldots,x_{k-1},y_k,x_{k+1},\ldots,x_n)|\,dy_k,
\end{align}
where, for any $k\in \{1,\ldots,n\}$, ${\mathbb I}_{x_k}$
denotes the set of all intervals in $\rr_{x_k}$ containing $x_k$.
Moreover, for any given $t\in(0,\fz)$, the
\emph{iterated maximal function} $\mathfrak{M}_t(f)$ of any
$f\in L_{\rm loc}^1(\rn)$ is defined by setting, for any $x\in\rn$,
\begin{align}\label{x2e13}
\mathfrak{M}_t(f)(x):=\lf[M_n\cdots \lf(M_1(|f|^t)\r)(x)\r]^{1/t}.
\end{align}

Now we state a result given in \cite[Theorem 1.6]{tn}
about the dual inequality of Stein type on mixed
Lebesgue spaces, which extends the corresponding
result of Fefferman and Stein \cite[Lemma 1]{fs71}.
To this end, we first recall the following notion
of Muckenhoupt weights.

\begin{definition}\label{x3d2}
Let $p\in(1,\fz)$. The \emph{weight class} $\Ap$ is defined
to be the set of all non-negative locally integrable functions
$\oz$ on $\rn$ such that
\begin{align*}
\Aw:=\sup_{Q\in\rn}\frac{1}{|Q|^p}\int_Q
\oz(x)\,dx \lf(\int_Q
\lf[\oz(y)\r]^{-p'/p}\,dy\r)^{p/p'}<\fz,
\end{align*}
where the supremum is taken over all closed cubes $Q\subset \rn$
and $1/p+1/p'=1$.

When $p=1$, the \emph{weight class} $\Al$ is defined
to be the set of all non-negative locally integrable functions
$\oz$ on $\rn$ such that
\begin{align*}
\Alw:=\sup_{Q\in\rn}\frac{1}{|Q|}\int_Q
\oz(x)\,dx \lf(\esup_{y\in Q}
\lf[\oz(y)\r]^{-1}\r)<\fz,
\end{align*}
where the supremum is taken over all closed cubes $Q\subset \rn$.
Moreover, for any $E\subset \rn$, let $\oz(E):=\int_E\oz(x)\,dx$.
\end{definition}

\begin{theorem}\label{x3t4}
Let $\vp:=(p_1,\ldots,p_n)\in [1,\fz)^n$,
$t\in(0,\min\{p_1,\ldots,p_n\})$ and,
for any $j\in\{1,\ldots,n\}$, $(\oz_j)^t\in \mathbf{A}_{p_j}(\rr)$.
Then there exists a positive constant $C$ such that,
for any measurable function $f$,
$$\lf\|\mathfrak{M}_t(f)\cdot\bigotimes_{j=1}^n(\oz_j)^{\frac1{p_j}}\r\|_{\lv}
\le C\lf\|f\cdot\bigotimes_{j=1}^n\lf[M_j(\oz_j)\r]^{\frac1{p_j}}\r\|_{\lv},$$
where, for any $x=(x_1,\ldots,x_n)\in\rn$,
$(\bigotimes_{j=1}^n\oz_j)(x):=\prod_{j=1}^n\oz_j(x_j)$
and $M_j$ as well as $\mathfrak{M}_t$ are, respectively,
as in \eqref{x2e14} and \eqref{x2e13}.
\end{theorem}

Moreover, the following Fefferman--Stein vector-valued inequality
of iterated maximal operators $\mathfrak{M}_t$
on mixed Lebesgue spaces was also shown by Nogayama in
\cite[Theorem 1.7]{tn} (see also \cite[p.\,679]{js08}).

\begin{theorem}\label{x3t5}
Let $\vp\in(0,\fz)^n$, $u\in(0,\fz]$ and
$t\in(0,\min\{p_1,\ldots,p_n,u\})$.
Then there exists a positive constant $C$ such that,
for any sequence $\{f_j\}_{j\in\nn}$ of measurable functions,
\begin{align*}
\lf\|\lf(\sum_{j\in\nn}\lf[\mathfrak{M}_t (f_j)\r]
^u \r)^{\frac{1}{u}}\r\|_{\lv}
\le C\lf\|\lf(\sum_{j\in\nn}\lf|f_j\r|^u \r)^{\frac{1}{u}}\r\|_{\lv},
\end{align*}
where $\mathfrak{M}_t$ is as in \eqref{x2e13}.
\end{theorem}

To present some applications of iterated weak Lebesgue spaces
and weak mixed-norm Lebesgue spaces,
we recall the following conclusions given in \cite[Theorem 3.1]{cs},
which imply some boundedness of non-negative measurable
functions from $L^{\fz}(\rr^{2n})$ to $L^{\vec q,\fz}(\rr^{2n})$
or to $L^{q_2,\fz}(L^{q_1,\fz})(\rr^{2n})$.

\begin{theorem}\label{x2t13}
Let $\vec q:=(q_1,q_2)\in(0,\fz]^2$. Then there exists a positive constant
$C$, depending only on $\vec q$ and $n$, such that,
for any non-negative measurable function $F$ on $\rr^{2n}$,
$$\|F\|_{L^{\vec q,\fz}(\rr^{2n})}
\le C\sup_{x,y\in\rn}\lf\{F(x,y)\lf(|x+y|+|x-y|\r)^{n/q_1+n/q_2}\r\}$$
and
$$\|F\|_{L^{q_2,\fz}(L^{q_1,\fz})(\rr^{2n})}
\le C\sup_{x,y\in\rn}\lf\{F(x,y)\lf(|x+y|+|x-y|\r)^{n/q_1+n/q_2}\r\}.$$
\end{theorem}

Next we state the following conclusions of \cite[Theorem 3.2]{cs},
which further induce two geometric inequalities as Theorem \ref{x2c1}
below.

\begin{theorem}\label{x2t14}
Let $\vec q:=(q_1,q_2)\in(0,\fz)^2$. Then there exists a constant
$C$, depending only on $\vec q$ and $n$, such that,
for any non-negative measurable function $F$ on $\rr^{2n}$,
$$\|F\|_{L^{q_2,\fz}(L^{q_1})(\rr^{2n})}
\le C\sup_{x,y\in\rn}\lf\{F(x,y)\lf(|x+y|+|x-y|\r)^{n/q_1+n/q_2}\r\}.$$

Moreover, on the endpoint cases, for any given $q_1\in(0,\fz]$,
there exists a positive
constant $C$, depending only on $q_1$ and $n$, such that,
for any non-negative measurable function $F$ on $\rr^{2n}$,
$$\|F\|_{L^{q_1,\fz}(L^{\fz})(\rr^{2n})}
\le C\sup_{x,y\in\rn}\{F(x,y)\lf(|x+y|+|x-y|\r)^{n/q_1}\}$$
and
$$\|F\|_{L^{\fz}(L^{q_1,\fz})(\rr^{2n})}
\le C\sup_{x,y\in\rn}\{F(x,y)\lf(|x+y|+|x-y|\r)^{n/q_1}\}.$$
\end{theorem}

As a consequence of Theorem \ref{x2t14},
the following geometric inequalities were obtained by
Chen and Sun in \cite[Corollary 3.3]{cs}.

\begin{theorem}\label{x2c1}
Let $\vp:=(p_1,p_2)\in(0,\fz]^2$.
Then there exists a positive constant $C$,
depending only on $\vp$ and $n$, such that,
for any non-negative function $f\in L^{p_1,\fz}(\rn)$
and non-negative function $g\in L^{p_2,\fz}(\rn)$,
$$\|f\|_{L^{p_1,\fz}(\rn)} \|g\|_{L^{p_2,\fz}(\rn)}
\le C\sup_{x,y\in\rn}\lf\{f(x)g(y)|x+y|^{n/p_1+n/p_2}\r\}.$$
Moreover, there exists a positive constant $C$ such that,
for any non-negative function $f\in L^{p_1}(\rn)$
and non-negative function $g\in L^{p_2}(\rn)$,
\begin{align}\label{x2e8}
\|f\|_{L^{p_1}(\rn)} \|g\|_{L^{p_2}(\rn)}
\le C\sup_{x,y\in\rn}\lf\{f(x)g(y)|x-y|^{n/p_1+n/p_2}\r\}.
\end{align}
\end{theorem}

\begin{remark}
Notice that, if $f\equiv g\equiv\mathbf{1}_E$
with some measurable set $E\in\rn$, then \eqref{x2e8} becomes
\begin{align*}
|E|
\le C\sup_{x,\,y\in E}|x-y|^n,
\end{align*}
which means the ``volume" of the set $E$, namely, $|E|$,
can be controlled by its ``diameter" $\sup_{x,\,y\in E}|x-y|$.
\end{remark}

Let $\gamma\in(0,\fz)$ and $f$ be a measurable function
defined on $\rr^{2n}$. The \emph{linear operators $L_{\gamma}$
and $T_{\gamma}$} are defined, respectively, by setting,
for any $x,\ y\in\rn$,
$$L_{\gamma}f(x,y):=\frac{f(x,y)}{|x-y|^{\gamma}}$$
and
$$T_{\gamma}f(x,y):=\frac{f(x,y)}{(|x+y|+|x-y|)^{\gamma}}.$$

Recall that, in \cite[Theorem 3.7]{cs}, the boundedness of
$T_{\gamma}$ and $T_{\gamma}^{-1}$
on iterated weak Lebesgue spaces and weak mixed-norm Lebesgue spaces
was shown as follows.

\begin{theorem}\label{x2t15}
Let $\gamma\in(0,\fz)$, $\vp:=(p_1,p_2)$ and $\vec q:=(q_1,q_2)$.
\begin{enumerate}
\item[{\rm(i)}] If $0<q_1\le p_1\le \fz$ and
$0<q_2\le p_2\le \fz$ satisfy the homogeneity condition
\begin{align}\label{2`1}
\frac1{q_1}+\frac1{q_2}=\frac1{p_1}+\frac1{p_2}+\frac{\gamma}{n},
\end{align}
then there exists a positive constant $C$, depending only on
$\vp,\ \vec q$ and $n$, such that, for any
non-negative measurable function $f$ on $\rr^{2n}$,
$$\|T_{\gamma}f\|_{L^{q_2,\fz}(L^{q_1,\fz})(\rr^{2n})}
\le C \|f\|_{L^{p_2,\fz}(L^{p_1,\fz})(\rr^{2n})}.$$

\item[{\rm(ii)}] If $0<p_1\le q_1\le \fz$ and
$0<p_2\le q_2\le \fz$ satisfy the homogeneity condition
\begin{align}\label{2`2}
\frac1{p_1}+\frac1{p_2}=\frac1{q_1}+\frac1{q_2}+\frac{\gamma}{n},
\end{align}
then there exists a positive constant $C$, depending only on
$\vp,\ \vec q$ and $n$, such that, for any
non-negative measurable function $f$ on $\rr^{2n}$, $$\lf\|T_{\gamma}^{-1}f\r\|_{L^{q_2,\fz}(L^{q_1,\fz})(\rr^{2n})}
\le C \|f\|_{L^{p_2,\fz}(L^{p_1,\fz})(\rr^{2n})}.$$

\item[{\rm(iii)}] If $0<q_1\le p_1\le \fz$ and
$0<q_2\le p_2\le \fz$ satisfy $p_1q_2=p_2q_1$ and
the homogeneity condition \eqref{2`1},
then there exists a positive constant $C$, depending only on
$\vp,\ \vec q$ and $n$, such that, for any
non-negative measurable function $f$ on $\rr^{2n}$,
$$\|T_{\gamma}f\|_{L^{\vec q,\fz}(\rr^{2n})}
\le C \|f\|_{L^{\vp,\fz}(\rr^{2n})}.$$

\item[{\rm(iv)}] If $0<p_1\le q_1\le \fz$ and
$0<p_2\le q_2\le \fz$ satisfy $p_1q_2=p_2q_1$ and
the homogeneity condition \eqref{2`2},
then there exists a positive constant $C$, depending only on
$\vp,\ \vec q$ and $n$, such that, for any
non-negative measurable function $f$ on $\rr^{2n}$,
$$\lf\|T_{\gamma}^{-1}f\r\|_{L^{\vec q,\fz}(\rr^{2n})}
\le C \|f\|_{L^{\vp,\fz}(\rr^{2n})}.$$
\end{enumerate}
\end{theorem}

Via Theorems \ref{x2t8} and \ref{x2t15}, Chen and Sun
\cite[Theorem 3.12]{cs} also obtained the following inequalities.

\begin{theorem}\label{x2t16}
Let $\gamma\in(0,\fz)$.
If $0<r<p_1\le \fz$ and $p_2\in(0,\fz]$
satisfy the homogeneity condition
$\frac1{r}=\frac1{p_1}+\frac{\gamma}{n}$,
then there exists a positive constant $C$, depending only on
$p_1,\ p_2,\ r$ and $n$, such that, for any
non-negative measurable function $f$ on $\rr^{2n}$,
$$\|L_{\gamma}f\|_{L^{p_2}(L^{r,\fz})(\rr^{2n})}
\le C \|f\|_{L^{p_2}(L^{p_1,\fz})(\rr^{2n})}$$
and
$$\|L_{\gamma}f\|_{L^{p_2,\fz}(L^{r,\fz})(\rr^{2n})}
\le C \|f\|_{L^{p_2,\fz}(L^{p_1,\fz})(\rr^{2n})}.$$
If $0<p_1<r\le \fz$ and $p_2\in(0,\fz]$
satisfy the homogeneity condition
$\frac1{p_1}=\frac1{r}+\frac{\gamma}{n}$,
then there exists a positive constant $C$, depending only on
$p_1,\ p_2,\ r$ and $n$, such that, for any
non-negative measurable function $f$ on $\rr^{2n}$,
$$\lf\|L_{\gamma}^{-1}f\r\|_{L^{p_2}(L^{r,\fz})(\rr^{2n})}
\ge C \|f\|_{L^{p_2}(L^{p_1,\fz})(\rr^{2n})}$$
and
$$\lf\|L_{\gamma}^{-1}f\r\|_{L^{p_2,\fz}(L^{r,\fz})(\rr^{2n})}
\ge C \|f\|_{L^{p_2,\fz}(L^{p_1,\fz})(\rr^{2n})}.$$
\end{theorem}

In addition, it was shown, via Theorem \ref{x2t15},
by Chen and Sun in \cite[Corollary 3.8]{cs}
that the Hardy--Littlewood--Sobolev inequality and its
reverse version hold true as follows.

\begin{theorem}\label{x2t17}
Let $\vp:=(p_1,p_2)\in(1,\fz)^2$ with $\frac1{p_1}+\frac1{p_2}>1$.
Then there exists a positive constant $C$, depending only on $\vp$ and $n$,
such that, for any non-negative functions $f\in L^{p_1}(\rn)$
and $g\in L^{p_2}(\rn)$,
\begin{align*}
\int_{\rn}\int_{\rn}f(x)g(x)|x-y|^{-n(2-1/p_1-1/p_2)}\,dx\,dy
\le C\|f\|_{L^{p_1}(\rn)}\|g\|_{L^{p_2}(\rn)}.
\end{align*}
If $\vp\in(0,1)^2$,
then there exists a positive constant $C$, depending only on $\vp$ and $n$,
such that, for any non-negative functions $f\in L^{p_1}(\rn)$
and $g\in L^{p_2}(\rn)$,
\begin{align*}
\int_{\rn}\int_{\rn}f(x)g(x)|x-y|^{n(1/p_1+1/p_2-2)}\,dx\,dy
\ge C\|f\|_{L^{p_1}(\rn)}\|g\|_{L^{p_2}(\rn)}.
\end{align*}
\end{theorem}

\section{Mixed Morrey spaces $\vm$} \label{sec3}

In this section, we first recall the notion of mixed Morrey spaces
$\vm$, with $\vq\in (0,\fz]^n$ and $p\in(0,\fz]$, and then discuss
some basic properties of them (see Subsection \ref{3s1} below) as well as the
boundedness of maximal operators on $\vm$ (see Subsection \ref{3s2} below).
Finally, we review the boundedness of Calder\'{o}n--Zygmund operators
and fractional integral operators on $\vm$ as well as
a necessary and sufficient condition for the boundedness
of the commutators of fractional integral operators on $\vm$
(see Subsection \ref{3s3} below).

\subsection{Definition and some basic properties}\label{3s1}

In this subsection, we present the definition and some examples
of mixed Morrey spaces and then recall some basic properties
about these spaces. To this end, we begin with recalling the notion of
mixed Morrey spaces given in \cite[Definition 1.3]{tn}.

\begin{definition}\label{x3d1}
Let $\vq:=(q_1,\ldots,q_n)\in (0,\fz]^n$ and $p\in(0,\fz]$
satisfy $$\sum_{j=1}^n\frac1 {q_j}\ge \frac{n}{p}.$$
The \emph{mixed Morrey space} $\vm$ is
defined to be the set of all measurable functions $f$
such that their quasi-norms
$$\|f\|_{\vm}:=\sup\lf\{|Q|^{\frac 1{p}-\frac{1}{n}(\sum_{j=1}^n\frac{1}{q_j})}
\lf\|f \mathbf{1}_{Q}\r\|_{\lvq}:\ Q~{\rm is~a ~cube~in}~\rn\r\}<\fz.$$
\end{definition}

\begin{remark}\label{x3r1}
Obviously, when $\vq:=(\overbrace{q,\ldots,q}^{n\ \rm times})$
with some $q\in(0,\fz]$, then the mixed Morrey space $\vm$ coincides with
the classical Morrey space ${\mathcal{M}^{p}_{q}(\rn)}$
of \cite{m38} and, in this case, they have the same quasi-norms.
In addition, when $\vq:=(q_1,\ldots,q_n)\in (0,\fz]^n$ and $p\in(0,\fz]$
satisfy
$$\sum_{j=1}^n\frac1 {q_j}=\frac{n}{p},$$
then $\vm=L^{\vq}(\rn)$ with equivalent quasi-norms (see \cite[Section 3]{tn}).
\end{remark}
We review the following completeness of mixed Morrey spaces $\vm$,
which was proved by Nogayama \cite[Remark 3.1]{tn}.

\begin{theorem}\label{x3t1}
Let $\vq\in[1,\fz]^n$ and $p\in[1,\fz]$.
Then the mixed Morrey space $\vm$ becomes a Banach space.
\end{theorem}

Next we borrow two examples given in \cite[Examples 3.3 and 3.5]{tn}
as follows to show some simple elements in $\vm$.

\begin{example}
Let $\vq:=(q_1,\ldots,q_n)\in (0,\fz]^n$, $p\in(0,\fz]$ and
$q_+:=\max\{q_1,\ldots,q_n\}$ satisfy $q_+\in(0,p)$. Then,
for any $x\in\rn\setminus\{\vec{0}_n\}$,
let $f(x):=|x|^{-\frac{n}{p}}$; then $$f\in\vm.$$
\end{example}

\begin{example}
Let $\vq:=(q_1,\ldots,q_n)\in (0,\fz]^n$,
$\vp:=(p_1,\ldots,p_n)\in(0,\fz]^n$ and $p\in(0,\fz]$
satisfy that, for any $j\in\{1,\ldots,n\}$,
$q_j\in(0,p_j)$ if $p_j\in(0,\fz)$, and $q_j\in(0,\fz]$ if $p_j=\fz$
as well as $$\sum_{j=1}^n\frac1 {p_j}=\frac{n}{p}.$$ Then,
for any $x:=(x_1,\ldots,x_n)\in\rn\setminus\{\vec{0}_n\}$,
let $f(x):=\prod_{j=1}^n\lf|x_j\r|^{-\frac{1}{p_j}}$; then $$f\in\vm.$$
\end{example}

We now state the embedding properties of mixed Morrey spaces given in
\cite[Proposition 3.2]{tn} as follows.

\begin{theorem}\label{x3t2}
Let $\vq:=(q_1,\ldots,q_n)\in(0,\fz]^n$,
$\vec r:=(r_1,\ldots,r_n)\in(0,\fz]^n$ and $p\in(0,\fz)$
satisfy that, for any $j\in\{1,\ldots,n\}$, $q_j\in(0,r_j]$
and $$\sum_{j=1}^n\frac1 {r_j}\ge\frac{n}{p}.$$
Then $${\mathcal{M}^{p}_{\vec r}(\rn)}\subset\vm$$
and the embedding is continuous.
\end{theorem}

\begin{remark}\label{x3r2}
Let $\vq:=(q_1,\ldots,q_n)\in(0,\fz]^n$
and $p\in[\max\{q_1,\ldots,q_n\},\fz)$.
Then, from Theorem \ref{x3t2}, it follows that,
$$\mathcal{M}^{p}_{\max\{q_1,\ldots,q_n\}}(\rn)\subset\vm
\subset \mathcal{M}^{p}_{\min\{q_1,\ldots,q_n\}}(\rn).$$
\end{remark}

\subsection{Some maximal inequalities on $\vm$}\label{3s2}

The aim of this subsection is the summarization of
conclusions on the boundedness of several maximal operators
on mixed Morrey spaces. We first present the following notion of
the uncentered Hardy--Littlewood maximal operator.

The \emph{uncentered Hardy--Littlewood maximal function} $M_{{\rm HL}}(f)$
of $f\in L_{\rm loc}^1(\rn)$ is defined by setting, for any $x\in\rn$,
\begin{align}\label{x3e1}
M_{{\rm HL}}(f)(x):=\sup_{x\in Q}
\frac1{|Q|}\int_Q|f(y)|\,dy,
\end{align}
where the supremum is taken over all cubes $Q\subset\rn$
containing $x$.

Then we display the boundedness of the iterated maximal operator $\mathfrak{M}_t$
on mixed Morrey spaces $\vm$ established in \cite[Theorem 1.4]{tn} as follows.

\begin{theorem}\label{x3t3}
Let $\vq:=(q_1,\ldots,q_n)\in (0,\fz)^n$ and $p\in(0,\fz)$
satisfy
\begin{align*}
\sum_{j=1}^n\frac1 {q_j}\ge \frac{n}{p}\quad{\rm and}\quad
\frac{n-1}{n}p<\max\lf\{q_1,\ldots,q_n\r\}.
\end{align*}
If $t\in(0,\min\{q_1,\ldots,q_n,p\})$, then
there exists a positive constant $C$ such that, for any
$f\in\vm$,
$$\|\mathfrak{M}_t(f)\|_{\vm}\le C\|f\|_{\vm},$$
where $\mathfrak{M}_t$ is as in \eqref{x2e13}.
\end{theorem}

\begin{remark}\label{3r.1}
Observe that, in \cite[Theorem 1.4]{tn}, $\vq:=(q_1,\ldots,q_n)\in (0,\fz]^n$.
However, it was pointed out to us by Professor Ferenc Weisz that, when some
$q_{i_1}\in (0,\fz)$ and $q_{i_2}:=\fz$ with
$i_1,\ i_2\in\{1,\ldots,n\}$ and $i_1<i_2$, \cite[Theorem 1.4]{tn}
is not correct. Indeed, Professor Ferenc Weisz showed that
\cite[Lemma 4.8]{tn}, which was used in the proof of \cite[Theorem 1.4]{tn},
does not hold true if $n:=2$, $p_1\in(1,\fz)$ and $p_2:=\fz$.
Thus, in Theorem \ref{x3t3}, the range of $\vq$ should be $(0,\fz)^n$.
\end{remark}

As a consequence of Theorem \ref{x3t3}, Nogayama \cite[Corollary 1.5]{tn}
also obtained the following boundedness of the iterated maximal operator
$\mathfrak{M}_t$ on classical Morrey spaces.

\begin{corollary}\label{x3c1}
Let $$0<\frac{n-1}{n}p<q\le p<\fz.$$
If $t\in(0,q)$, then there exists a positive constant
$C$ such that, for any $f\in{\mathcal{M}^{p}_{q}(\rn)}$,
$$\|\mathfrak{M}_t(f)\|_{{\mathcal{M}^{p}_{q}(\rn)}}
\le C\|f\|_{{\mathcal{M}^{p}_{q}(\rn)}},$$
where $\mathfrak{M}_t$ is as in \eqref{x2e13}.
\end{corollary}

Moreover, Nogayama in \cite[Theorems 1.8 and 1.9]{tn} established
the following two succeeding Fefferman--Stein vector-valued inequalities
of the uncentered Hardy--Littlewood maximal operator $\HL$ and
the iterated maximal operator $\mathfrak{M}_t$ on mixed
Morrey spaces.

\begin{theorem}\label{x3t6}
Let $\vq:=(q_1,\ldots,q_n)\in(1,\fz)^n$, $u\in(1,\fz]$ and $p\in(1,\fz]$
satisfy $$\frac{n}{p}\leq\sum_{j=1}^n \frac{1}{q_j}.$$
Then there exists a positive constant $C$ such that,
for any sequence $\{f_j\}_{j\in\nn}$ of measurable functions,
$$
\lf\|\lf\{\sum_{j\in\nn}\lf[\HL \lf(f_j\r)\r]^u\r\}^{\frac{1}{u}}\r\|_{\vm}
\le C\lf\|\lf(\sum_{j\in\nn}\lf|f_j\r|^u \r)^{\frac{1}{u}}\r\|_{\vm},
$$
where $\HL$ is as in \eqref{x3e1}.
\end{theorem}

\begin{theorem}\label{x3t7}
Let $\vq:=(q_1,\ldots,q_n)\in(0,\fz)^n$,
$u\in(0,\fz]$ and $p\in(0,\fz)$ satisfy
\begin{align*}
\sum_{j=1}^n \frac{1}{q_j}\ge \frac{n}{p}
\quad{\rm and }\quad\frac{n-1}{n}p<\max\{q_1,\ldots,q_n\}.
\end{align*}
If $t\in(0,\min\{q_1,\ldots q_n,u\})$, then
there exists a positive constant $C$ such that,
for any $\{f_j\}_{j\in\nn}\subset\vm$,
$$
\lf\|\lf\{\sum_{j\in\nn}[\mathfrak{M}_t (f_j)]
^u \r\}^{\frac{1}{u}}\r\|_{\vm}
\le C\lf\|\lf(\sum_{j\in\nn}\lf|f_j\r|^u \r)^{\frac{1}{u}}\r\|_{\vm},
$$
where $\mathfrak{M}_t$ is as in \eqref{x2e13}.
\end{theorem}

\begin{remark}
Observe that, in \cite[Theorem 1.9]{tn}, $\vq:=(q_1,\ldots,q_n)\in (0,\fz]^n$.
However, similarly to Remark \ref{3r.1}, when some
$q_{i_1}\in (0,\fz)$ and $q_{i_2}:=\fz$ with
$i_1,\ i_2\in\{1,\ldots,n\}$ and $i_1<i_2$, \cite[Theorem 1.9]{tn}
is not correct, which was also pointed out to us by Professor Ferenc Weisz.
Therefore, in Theorem \ref{x3t7}, the range of $\vq$ should also be $(0,\fz)^n$.
\end{remark}

As a corollary of Theorem \ref{x3t7}, the following Fefferman--Stein
vector-valued inequality of the iterated maximal
operator $\mathfrak{M}_t$ on classical Morrey spaces was
obtained by Nogayama in \cite[Theorem 1.10]{tn}.

\begin{corollary}\label{x3c2}
Let $u\in(0,\fz]$ and
$$0<\frac{n-1}{n}p<q\leq p<\infty.$$
If $t\in(0,\min\{q,u\})$, then
there exists a positive constant $C$ such that,
for any $\{f_j\}_{j\in\nn}\subset{\mathcal{M}^{p}_{q}(\rn)}$,
$$
\lf\|\lf(\sum_{j\in\nn}[\mathfrak{M}_t (f_j)]
^u \r)^{\frac{1}{u}}\r\|_{{\mathcal{M}^{p}_{q}(\rn)}}
\le C\lf\|\lf(\sum_{j\in\nn}\lf|f_j\r|^u \r)
^{\frac{1}{u}}\r\|_{\mathcal{M}^{p}_{q}(\rn)}.
$$
\end{corollary}

\subsection{Boundedness of operators on $\vm$}\label{3s3}

In this subsection, we discuss the boundedness of
Calder\'{o}n--Zygmund operators $T$
and fractional integral operators $I_{\az}$
on mixed Morrey spaces $\vm$. Then we review
a necessary and sufficient condition for the boundedness
of commutators of fractional integral operators $I_{\az}$ on $\vm$.
To this end, we first recall the following notion of
Calder\'{o}n--Zygmund operators (see, for instance, \cite{tn}).

\begin{definition}\label{x3d3}
A linear operator $T$ is called a \emph{Calder\'{o}n--Zygmund operator}
if its kernel
$$k:\ \lf\{(x,y)\in\rn\times\rn:\  x\neq y\r\}\to \mathbb{C}$$
satisfies that
\begin{enumerate}
\item[{\rm(i)}] there exists a positive constant $C$
such that, for any $x,\ y\in\rn$ with $x\neq y$,
$$\lf|k(x,y)\r|\leq \frac{C}{|x-y|^n};$$

\item[{\rm(ii)}]there exist positive constants $C$
and $\varepsilon$ such that, for any $x,\ y\in\rn$
with $|x-y|\geq 2|x-z|\neq 0$,
$$
|k(x,y)-k(z,y)|+|k(y,x)-k(y,z)|\leq C\frac{|x-z|
^{\varepsilon}}{|x-y|^{n+\varepsilon}};
$$

\item[{\rm(iii)}]if $f\in L^{\infty}(\rn)$
with compact support, then, for any $x\notin \supp(f)$,
$$
Tf(x):=\int_{\rn} k(x,y)f(y)\,dy.
$$
\end{enumerate}
\end{definition}

It was shown in \cite[Theorem 1.12]{tn} that the
Calder\'{o}n--Zygmund operator is bounded on
mixed Morrey spaces $\vm$ as follows.

\begin{theorem}\label{x3t8}
Let $\vq:=(q_1,\ldots,q_n)\in(1,\fz)^n$, $p\in(1,\fz)$ satisfy
$$
\sum_{j=1}^n \frac{1}{q_j}\ge \frac{n}{p}
$$
and $T$ be a Calder\'{o}n--Zygmund operator
defined on ${\mathcal{M}^{p}_{\min\{q_1,\ldots,q_n\}}(\rn)}$.
Then there exists a positive constant $C$ such that, for any $f \in \vm$,
$$\lf\|Tf\r\|_{\vm}\le C\|f\|_{\vm}.$$
\end{theorem}

\begin{remark}\label{x3r3}
We should point out that the statement of Theorem \ref{x3t8}
contains the following fact that
$$\vm\subset{\mathcal{M}^{p}_{\min\{q_1,\ldots,q_n\}}(\rn)}.$$
Indeed, this embedding follows from Remark \ref{x3r2}.
\end{remark}

Next we present the notion of fractional integral operators
in \cite{tn} as follows.

\begin{definition}\label{x3d4}
Let $\az\in(0,n)$. The \emph{fractional integral operator} $I_{\az}$
of order $\az$ is defined by setting, for any $f\in L_{\rm loc}^1(\rn)$ and $x\in\rn$,
$$I_{\az} (f)(x):= \int_{\rn} \frac{f(y)}{|x-y|^{n-\az}}\,dy.$$
\end{definition}

Now we state the following result about the boundedness of
fractional integral operators on mixed Morrey spaces,
which was shown by Nogayama in \cite[Theorem 1.11]{tn}.

\begin{theorem}\label{x3t9}
Let $\az\in(0,n)$, $\vq:=(q_1,\ldots,q_n)\in(1,\fz)^n$,
$\vec s:=(s_1,\ldots,s_n)\in(1,\fz)^n$ and $p,\ r\in(1,\fz)$.
Assume that $$\sum_{j=1}^n \frac{1}{q_j}\ge\frac{n}p,
\quad\sum_{j=1}^n \frac{1}{s_j}\ge\frac{n}r,
\quad\frac{1}r=\frac{1}p-\frac{\alpha}n$$
and, for any $j\in\{1,\ldots,n\}$,
$\frac{q_j}{p}=\frac{s_j}{r}$.
Then there exists a positive constant $C$ such that, for any $f\in \vm$,
$$\lf\|I_{\alpha} (f)\r\|_{\mathcal{M}^{r}_{\vec s}(\rn)}\le C\|f\|_{\vm}.$$
\end{theorem}

We next introduce the notion of the commutators of fractional integral operators
in \cite{tn19} as follows.

\begin{definition}\label{x3d5}
Let $\az\in(0,n)$, $b\in L_{\rm loc}^1(\rn)$ and
$I_{\az}$ be a fractional integral operator
of order $\az$. The \emph{commutator} $[b,I_{\alpha}]$
of $I_{\az}$ is defined by setting, for any $f\in L_{\rm loc}^1(\rn)$ and $x\in\rn$,
$$[b,I_{\alpha}](f)(x):=b(x)I_{\alpha} (f)(x)-I_{\alpha} (bf)(x).$$
\end{definition}

Via Theorem \ref{x3t9} and a sharp maximal inequality on
mixed Morrey spaces, Nogayama \cite[Theorem 1.2]{tn19} gave
a necessary and sufficient condition for the boundedness
of the commutators of fractional integral operators on mixed Morrey spaces.
To present this result, we first recall the notion of ${\rm BMO}(\rn)$.

\begin{definition}\label{x3d6}
The \emph{space ${\rm BMO}(\rn)$} is
defined to be the set of all $f\in L_{\rm loc}^1(\rn)$
such that their quasi-norms
$$\|f\|_{{\rm BMO}(\rn)}:=\sup_{Q}\frac1{|Q|}\int_Q
\lf|f(x)-f_Q\r|\,dx<\fz,$$
where the supremum is taken over all cubes $Q$ in $\rn$
and $$f_Q:=\frac1{|Q|}\int_Q f(y)\,dy.$$
\end{definition}

\begin{theorem}\label{x3t10}
Let $\az\in(0,n)$, $\vq:=(q_1,\ldots,q_n)\in(1,\fz)^n$,
$\vec s:=(s_1,\ldots,s_n)\in(1,\fz)^n$, $p\in(1,n/\az)$ and $r\in(1,\fz)$.
Assume that $$\sum_{j=1}^n \frac{1}{q_j}\ge\frac{n}p,
\quad\sum_{j=1}^n \frac{1}{s_j}\ge\frac{n}r,
\quad\frac{1}r=\frac{1}p-\frac{\alpha}n$$
and, for any $j\in\{1,\ldots,n\}$,
$\frac{q_j}{p}=\frac{s_j}{r}$.
Then the following statements are mutually equivalent:
\begin{enumerate}
\item[{\rm (i)}] $b\in {\rm BMO}(\rn)$;
\item[{\rm (ii)}] $[b,I_{\alpha}]$ is bounded from $\vm$
to $\mathcal{M}^{r}_{\vec s}(\rn)$;
\item[{\rm (iii)}] $[b,I_{\alpha}]$ is bounded from $\overline{\vm}$
to $\mathcal{M}^{r}_{\vec s}(\rn)$;
\item[{\rm (iv)}] $[b,I_{\alpha}]$ is bounded from $\overline{\vm}$
to $\mathcal{M}^{r}_{1}(\rn)$.
\end{enumerate}
Here, $\overline{\vm}$ denotes the closure of $C_c^{\fz}(\rn)$ in the norm of $\vm$.
\end{theorem}

\section{Anisotropic mixed-norm Hardy spaces $\vh$}\label{sec4}

In this section, we first present the definition of anisotropic
mixed-norm Hardy spaces and some basic facts of them
(see Subsection \ref{4s1} below).
Then, various real-variable characterizations of
these Hardy spaces, respectively, in terms of the maximal functions,
atoms, finite atoms and Lusin area functions
as well as Littlewood--Paley $g$-functions or
$g_{\lambda}^\ast$-functions, are displayed
(see Subsection \ref{4s2.1} below). As the applications
of these various real-variable characterizations, the dual spaces of $\vh$
(see Subsection \ref{4s3} below),
and the boundedness of anisotropic Calder\'on--Zygmund operators
(see Subsection \ref{4s4} below) are presented.
Some errors and gaps existing in the proof of \cite[Theorem 4.1]{hlyy}
are also corrected and sealed (see Subsection \ref{4s2.1} below).
In addition, by providing a new proof, we improve the maximal
function characterizations of $\vh$ given in \cite[Theorem 3.1]{cgn17}
(see Subsection \ref{4s2.2} below). The revised versions of
the boundedness of anisotropic Calder\'on--Zygmund operators
are obtained (see Subsection \ref{4s4} below).

\subsection{Definitions and basic properties}\label{4s1}

This subsection is devoted to recalling the notion of anisotropic
quasi-homogeneous norms and anisotropic mixed-norm Hardy spaces
as well as some basic properties of them.
We begin with stating the notion of anisotropic quasi-homogeneous norms
given in \cite{bil66, f66} (see also \cite{sw78}) as follows.
For any $b:=(b_1,\ldots,b_n)\in\rn$,
$x:=(x_1,\ldots,x_n)\in \rn$ and $t\in[0,\fz)$,
let $t^b x:=(t^{b_1}x_1,\ldots,t^{b_n}x_n)$.

\begin{definition}\label{3d1}
Let $\va:=(a_1,\ldots,a_n)\in [1,\fz)^n$.
The \emph{anisotropic quasi-homogeneous norm} $|\cdot|_{\va}$,
associated with $\va$, is a non-negative measurable function on $\rn$
defined by setting $|\vec0_n|_{\va}:=0$ and, for any $x\in \rn\setminus\{\vec0_n\}$,
$|x|_{\va}:=t_0$, where $t_0$ is the unique positive number such that $|t_0^{-\va}x|=1$,
namely,
$$\frac{x_1^2}{t_0^{2a_1}}+\cdots+\frac{x_n^2}{t_0^{2a_n}}=1.$$
\end{definition}

We also present the following notions of the anisotropic bracket
and the homogeneous dimension (see, for instance, \cite{sw78}).

\begin{definition}\label{3d2}
Let $\va:=(a_1,\ldots,a_n)\in [1,\fz)^n$. The \emph{anisotropic bracket}, associated with $\va$,
is defined by setting,  for any $x\in \rn$, $$\lg x\rg_{\va}:=\lf|(1,x)\r|_{(1,\va)}.$$
Furthermore, the \emph{homogeneous dimension} $\nu$ is defined by setting
$$\nu:=|\vec{a}|:=a_1+\cdots+a_n.$$
\end{definition}

For any $\va\in [1,\fz)^n$, $r\in (0,\fz)$ and $x\in \rn$,
we define the \emph{anisotropic ball} $B_{\va}(x,r)$,
with center $x$ and radius $r$, by
setting $$B_{\va}(x,r):=\lf\{y\in \rn:\ \lf|y-x\r|_{\va} < r\r\}.$$
Then $B_{\va}(x,r)= x+r^{\va}B_{\va}(\vec0_n,1)$ and
$|B_{\va}(x,r)|=\upsilon_n r^{\nu}$, where $\upsilon_n:=|B(\vec{0}_n,1)|$.
In what follows, we always let $B_0:=\{y\in\rn:\ |y|<1\}=B_{\va}(\vec0_n,1)$
(see \cite[Lemma 2.4(ii)]{hlyy}) and
$\mathfrak{B}$ be the set of all anisotropic balls, namely,
\begin{align}\label{3e2}
\mathfrak{B}:=\lf\{B_{\va}(x,r):\ x\in\rn,\ r\in(0,\fz)\r\}.
\end{align}
For any $B\in\mathfrak{B}$ centered at $x\in\rn$ with radius $r\in (0,\fz)$ and $\delta\in(0,\fz)$,
let
\begin{align}\label{2e2'}
B^{(\delta)}:=B^{(\delta)}_{\va}(x,r):=B_{\va}(x,\delta r).
\end{align}
In addition, for any $x\in \rn$ and $r\in (0,\fz)$, the \emph{anisotropic cube}
$Q_{\va}(x,r)$ is defined by setting $Q_{\va}(x,r):=x + r^{\va}(-1,1)^n,$
whose Lebesgue measure $|Q_{\va}(x,r)|$ equals $2^n r^\nu$.
Denote by $\mathfrak{Q}$ the set of all anisotropic cubes, namely,
\begin{align}\label{x3e2}
\mathfrak{Q}:=\lf\{Q_{\va}(x,r):\ x\in\rn,\ r\in(0,\fz)\r\}.
\end{align}

On another hand, recall that a \emph{Schwartz function}
is a $C^\infty(\rn)$ function $\varphi$ satisfying,
for any $N\in\zz_+$ and multi-index $\az\in\zz_+^n$,
$$\|\varphi\|_{N,\alpha}:=
\sup_{x\in\rn}\lf\{(1+|x|)^N
|\partial^\alpha\varphi(x)|\r\}<\infty.$$
Denote by
$\cs(\rn)$ the set of all Schwartz functions, equipped
with the topology determined by
$\{\|\cdot\|_{N,\alpha}\}_{N\in\zz_+,\az\in\zz_+^n}$,
and $\cs'(\rn)$ the \emph{dual space} of $\cs(\rn)$, equipped
with the weak-$\ast$ topology.
For any $N\in\mathbb{Z}_+$, let
$$\cs_N(\rn):=\lf\{\varphi\in\cs(\rn):\
\|\varphi\|_{\cs_N(\rn)}:=
\sup_{x\in\rn}\lf[\lg x\rg_{\va}^N\sup_{|\az|\le N}
|\partial^\alpha\varphi(x)|\r]\le 1\r\}.$$
In what follows, for any $\varphi \in \cs(\rn)$ and $t\in (0,\fz)$,
let $\varphi_t(\cdot):=t^{-\nu}\varphi(t^{-\va}\cdot)$.

To introduce the mixed-norm Hardy spaces,
we first recall the following notions of radial maximal functions,
non-tangential maximal functions and non-tangential
grand maximal functions (see, for instance, \cite{cgn17}).

\begin{definition}\label{3d4}
Let $\varphi\in\cs(\rn)$ and  $f\in\cs'(\rn)$.
The \emph{radial maximal function} $M_\varphi^0(f)$
of $f$ associated to $\varphi$
is defined by setting, for any $x\in\rn$,
\begin{equation*}
M_\varphi^0(f)(x):= \sup_{t\in (0,\fz)}
\lf|\varphi_t\ast f(x)\r|,
\end{equation*}
and the
\emph{non-tangential maximal function} $M_{\varphi}(f)$ of $f$
associated to $\varphi$ is defined by setting, for any $x\in\rn$,
$$
M_{\varphi}(f)(x):= \sup_{y\in B_{\va}(x,t),
t\in (0,\fz)}\lf|\varphi_t\ast f(y)\r|.
$$
Moreover, for any given $N\in\mathbb{N}$, the
\emph{non-tangential grand maximal function} $M_N(f)$ of
$f$ associated to $\varphi$
is defined by setting, for any $x\in\rn$,
\begin{equation*}
M_N(f)(x):=\sup_{\varphi\in\cs_N(\rn)}
M_\varphi(f)(x).
\end{equation*}
\end{definition}

In what follows, for any $\vp:=(p_1,\ldots,p_n)\in (0,\fz)^n$,
we always let
\begin{align}\label{2e10}
p_-:=\min\{p_1,\ldots,p_n\},\hspace{0.35cm}
p_+:=\max\{p_1,\ldots,p_n\}\hspace{0.35cm}
{\rm and}\hspace{0.35cm}\underline{p}\in(0,\min\{1,p_-\}).
\end{align}
Similarly, for any $\va:=(a_1,\ldots,a_n)\in [1,\fz)^n$, let
\begin{align}\label{2e9}
a_-:=\min\{a_1,\ldots,a_n\}\hspace{0.35cm}
{\rm and}\hspace{0.35cm} a_+:=\max\{a_1,\ldots,a_n\}.
\end{align}

We now present the notion of anisotropic mixed-norm
Hardy spaces as follows, which was first introduced
by Cleanthous et al. \cite[Definition 3.3]{cgn17}.

\begin{definition}\label{3d5}
Let $\va\in[1,\fz)^n$, $\vp\in(0,\fz)^n$,
$N_{\vp}:=\lfloor\nu\frac{a_+}{a_-}
(\frac{1}{\min\{1,p_-\}}+1)+\nu+2a_+\rfloor+1$ and
\begin{align}\label{2e11}
N\in\mathbb{N}\cap \lf[N_{\vp},\fz\r),
\end{align}
where $a_-$, $a_+$ are as in \eqref{2e9} and $p_-$ as in \eqref{2e10}.
The \emph{anisotropic mixed-norm Hardy space} $\vh$ is defined by setting
\begin{equation*}
\vh:=\lf\{f\in\cs'(\rn):\ M_N(f)\in\lv\r\}
\end{equation*}
and, for any $f\in\vh$, let
$\|f\|_{\vh}:=\| M_N(f)\|_{\lv}$.
\end{definition}

\begin{remark}
\begin{enumerate}
\item[{\rm (i)}] Observe that the quasi-norm of $\vh$ in Definition \ref{3d5}
depends on $N$. However, by Theorem \ref{3t2} below,
we know that the space $\vh$ is independent
of the choice of $N$ as long as $N$ same as in Theorem \ref{3t2}.

\item[{\rm (ii)}] Recall that Ho \cite{ho18} introduced the mixed
Lebesgue spaces with variable exponents. Then, based on those spaces,
the corresponding variable mixed-norm Hardy spaces may also be worth studying.
\end{enumerate}
\end{remark}

The following completeness of $\vh$ is a consequence of
Proposition \ref{5p1} and \cite[Proposition 3.7]{hlyy19}
with $A$ therein being as
\begin{align}\label{5eq2}
A:=\left(
              \begin{array}{cccc}
                2^{a_1} & 0 & \cdots & 0\\
                0 & 2^{a_2} & \cdots & 0\\
                \vdots & \vdots& &\vdots \\
                0 & 0 & \cdots & 2^{a_n} \\
              \end{array}
            \right).
\end{align}

\begin{theorem}\label{2p1}
Let $\vp$ and $N$ be as in Definition \ref{3d5}.
Then $\vh$ is complete.
\end{theorem}

The following proposition established in \cite[Theorem 6.1]{cgn17}
shows the relation between the
mixed Lebesgue spaces $\lv$ and
anisotropic mixed-norm Hardy spaces $\vh$.

\begin{proposition}\label{3p1}
Let $\vp\in(1,\fz)^n$. Then $\vh=\lv$
with equivalent norms.
\end{proposition}

\subsection{Real-variable characterizations of $\vh$}\label{4s2}

The goals of this subsection are twofold. The first one is in Subsection \ref{4s2.1} below to
display various real-variable characterizations of anisotropic
mixed-norm Hardy spaces $\vh$, respectively, in terms of the maximal
functions, atoms, finite atoms,
Lusin area functions as well as Littlewood--Paley
$g$-functions or $g_{\lambda}^\ast$-functions,
and also correct some errors or seal some gaps existing in the proof
of \cite[Theorem 4.1]{hlyy}.
The second one is in Subsection \ref{4s2.2} below to provide a new proof of the maximal
function characterizations of $\vh$, which allows the exponent $N$
to have a weaker restriction than \cite[Theorem 3.4]{cgn17}
which is re-stated in Theorem \ref{3t1} below.

\subsubsection{Various real-variable characterizations of $\vh$}\label{4s2.1}

To begin with, we first recall the following maximal
function characterizations of $\vh$ established
by Cleanthous et al. \cite[Theorem 3.4]{cgn17}.

\begin{theorem}\label{3t1}
Let $\va\in[1,\fz)^n$, $\vp\in(0,\fz)^n$
and $N$ be as in \eqref{2e11}.
Then, for any given $\varphi\in\cs(\rn)$
with $\int_{\rn}\varphi(x)\,dx\neq0$,
the following statements are mutually equivalent:
\begin{enumerate}
\item[{\rm(i)}] $f\in\vh;$
\item[{\rm(ii)}] $f\in\cs'(\rn)$ and $M_{\varphi}(f)\in\lv;$
\item[{\rm(iii)}] $f\in\cs'(\rn)$ and $M_\varphi^0(f)\in\lv.$
\end{enumerate}
Moreover, there exist two positive constants $C_1$ and $C_2$,
independent of $f$, such that
\begin{align*}
\|f\|_{\vh}\le C_1\lf\|M_\phi^0(f)\r\|_{\lv}
\le C_1 \lf\|M_{\varphi}(f)\r\|_{\lv}
\le C_2\|f\|_{\vh}.
\end{align*}
\end{theorem}

To complete the real-variable theory of the Hardy spaces $\vh$,
Huang et al. \cite[Theorem 3.16]{hlyy} established
the atomic characterizations of $\vh$.
To state this atomic characterizations, we first introduce the
notions of anisotropic mixed-norm $(\vp,r,s)$-atoms and
anisotropic mixed-norm atomic Hardy spaces as follows, which are,
respectively, \cite[Definitions 3.1 and 3.2]{hlyy}.
In what follows, for any $q\in(0,\fz]$,
denote by $L^q(\rn)$ the \emph{space of all measurable functions} $f$ such that
$$\|f\|_{L^q(\rn)}:=\lf\{\int_{\rn}\lf|f(x)\r|^q\,dx\r\}^{1/q}<\fz$$
with the usual modification made when $q=\fz$.

\begin{definition}\label{5d1}
Let $\va\in[1,\fz)^n$, $\vp:=(p_1,\ldots,p_n)\in(0,\fz)^n$, $r\in (1,\fz]$ and
\begin{align}\label{5eq1}
s\in\lf[\lf\lfloor\frac{\nu}{a_-}\lf(\frac{1}{p_-}-1\r) \r\rfloor,\fz\r)\cap\zz_+,
\end{align}
where $a_-$ is as in \eqref{2e9} and $p_-$ as in \eqref{2e10}.
An \emph{anisotropic mixed-norm $(\vp,r,s)$-atom} $a$ is
a measurable function on $\rn$ satisfying
\begin{enumerate}
\item[{\rm (i)}] $\supp a:=\{x\in\rn:\ a(x)\neq 0\} \subset B$, where
$B\in\mathfrak{B}$ with $\mathfrak{B}$ as in \eqref{3e2};

\item[{\rm (ii)}] $\|a\|_{L^r(\rn)}\le \frac{|B|^{1/r}}{\|{\mathbf 1}_B\|_{\lv}}$;

\item[{\rm (iii)}] $\int_{\mathbb R^n}a(x)x^\az\,dx=0$ for any $\az\in\zz_+^n$
with $|\az|\le s$.
\end{enumerate}
\end{definition}

In what follows, we always call an anisotropic mixed-norm
$(\vp,r,s)$-atom simply by a \emph{$(\vp,r,s)$-atom}.

\begin{definition}\label{5d2}
Let $\va\in[1,\fz)^n$, $\vp\in(0,\fz)^n$, $r\in (1,\fz]$
and $s$ be as in \eqref{5eq1}. The \emph{anisotropic mixed-norm
atomic Hardy space} $\vah$ is defined to be the
set of all $f\in\cs'(\rn)$ satisfying that there exist
$\{\lz_i\}_{i\in\nn}\subset\mathbb{C}$
and a sequence $\{a_i\}_{i\in\nn}$ of $(\vp,r,s)$-atoms
supported, respectively, in
$\{B_i\}_{i\in\nn}\subset\mathfrak{B}$ such that
\begin{align*}
f=\sum_{i\in\nn}\lz_ia_i
\quad\mathrm{in}\quad\cs'(\rn).
\end{align*}
Moreover, for any $f\in\vah$, let
\begin{align*}
\|f\|_{\vah}:=
{\inf}\lf\|\lf\{\sum_{i\in\nn}
\lf[\frac{|\lz_i|{\mathbf 1}_{B_i}}{\|{\mathbf 1}_{B_i}\|_{\lv}}\r]^
{\underline{p}}\r\}^{1/{\underline{p}}}\r\|_{\lv},
\end{align*}
where $\underline{p}$ is as in \eqref{2e10} and
the infimum is taken over all decompositions of $f$ as above

\end{definition}

\begin{remark}\label{3r.2}
Recall that, in \cite[Definition 3.2]{hlyy}, $\underline{p}:=\min\{1,p_-\}$
with $p_-$ as in \eqref{2e10}.
However, in Definition \ref{5d2} above, we correct the range of
$\underline{p}$ to be $(0,\min\{1,p_-\})$
due to the fact that Lemma \ref{3l1} below holds true only for $\vp\in(1,\fz)^n$.
\end{remark}

It is well known that the Calder\'on--Zygmund decomposition is a
key tool in the real-variable theory of function spaces.
The idea behind this decomposition is that it is often useful to
split a function or distribution into its ``good" and ``bad" part, and then
use different techniques to analysis each part. Recall that
Huang et al. \cite[Lemma 3.12]{hlyy} obtained the
following adapted Calder\'on--Zygmund decomposition, which plays a
curial role in the proof of atomic characterizations of $\vh$.
Indeed, as has been demonstrated in the proof of the atomic
decomposition for classical Hardy spaces, we need to use this lemma
to break down functions or distributions into atoms.

Let $\Phi$ be some fixed $C^\fz(\rn)$ function satisfying
$\supp \Phi\subset B(\vec{0}_n,1)$ and $\int_{\rn} \Phi(x)\,dx \neq 0$.
For any $f\in\cs'(\rn)$ and $x\in \rn$, we always let
\begin{equation}\label{3e16}
M_0(f)(x):=M_\Phi^0(f)(x),
\end{equation}
where $M_\Phi^0(f)$ is as in Definition \ref{3d4} with $\phi$ replaced by $\Phi$.
In what follows, for any given $s\in \mathbb{Z}_+$,
the \emph{symbol $\cp_s(\rn)$} denotes the linear space of all polynomials
on $\rn$ with degree not greater than $s$.

\begin{lemma}\label{5l1}
Let $\va \in [1,\fz)^n$, $\vp\in(0,\fz)^n$,
$s\in \mathbb{Z}_+$ and
$N$ be as in \eqref{2e11}.
For any $\sa\in (0,\fz)$ and $f\in \vh$,
let $$\CO := \{x\in \rn:\ M_N(f)(x)>\sa\},$$
where $M_N$ is as in Definition \ref{3d4}. Then
the following statements hold true:
\begin{enumerate}
\item[\rm(i)] There exists a sequence
$\{B_k^*\}_{k\in\nn}\subset \mathfrak{B}$
with $\mathfrak{B}$ as in \eqref{3e2},
which has finite intersection property,
such that $$\CO = \bigcup_{k\in \nn}\Qkk.$$

\item[\rm(ii)] There exist two distributions $g$ and $b$
such that $f=g+b$ in $\cs'(\rn)$.

\item[\rm(iii)] For the distribution $g$ as in {\rm (ii)} and any $x\in \rn$,
\begin{equation*}
M_0(g)(x)\ls M_N(f)(x){\mathbf 1}_{\CO^{\com}}(x)+\sum_{k\in\nn}
\frac{\sa r_k^{\nu+(s+1)a_-}}{(r_k+|x-x_k|_{\va})^{\nu+(s+1)a_-}},
\end{equation*}
where $a_-$ and $M_0$ are as in \eqref{2e9}, respectively, \eqref{3e16}, and the implicit positive
constant is independent of $f$ and $g$.
Moreover, for any $k\in \nn$, $x_k$ denotes
the center of $B_k^*$ and there exists a constant
$A^*\in (1,\fz)$, independent of $k$, such that $A^*-1$ is small enough
and $A^* r_k$ equals the radius of $B_k^*$.

\item[\rm(iv)] If $f\in L_{\loc}^1(\rn)$, then the distribution $g$ as in
{\rm (ii)} belongs to $L^{\fz}(\rn)$
and $\|g\|_{L^{\fz}(\rn)} \ls \sa$ with the implicit positive
constant independent of $f$ and $g$.

\item[\rm(v)] If $s$ is as in \eqref{5eq1} and $b$ as in {\rm(ii)}, then $b=\sum_{k\in\nn}b_k$ in $\cs'(\rn)$,
where, for any $k\in \nn,~b_k:=(f-c_k)\eta_k$, $\{\eta_k\}_{k\in \nn}$
is a partition of unity with respect to $\{B_k^*\}_{k\in \nn}$,
namely, for any $k\in\nn$, $\eta_k\in C_c^{\fz}(\rn)$,
$\supp \eta_k\subset B_k^*$, $0\le\eta_k \le 1$
and $${\mathbf 1}_{\CO}=\sum_{k\in \nn}\eta_k,$$ and
$c_k\in \cp_s(\rn)$ is a polynomial such that, for any $q\in\cp_s(\rn)$,
$$\langle f-c_k,q\eta_k\rangle=0.$$
Moreover, for any $k\in \nn$ and $x\in \rn$,
\begin{equation*}
M_0(b_k)(x)\ls M_N(f)(x){\mathbf 1}_{B_k^*}(x)+\frac{\sa r_k^{\nu+(s+1)a_-}}
{|x-x_k|_{\va}^{\nu+(s+1)a_-}}{\mathbf 1}_{({B_k^*})^{\com}}(x),
\end{equation*}
where $a_-$ and $M_0$ are as in \eqref{2e9}, respectively, \eqref{3e16}, and the implicit
positive constant is independent of $f$ and $k$.
\end{enumerate}
\end{lemma}

Now we state the atomic characterizations of $\vh$ as follows,
which was established by Huang et al.
in \cite[Theorem 3.16]{hlyy}.

\begin{theorem}\label{5t1}
Let $\va\in [1,\fz)^n,\ \vp\in (0,\fz)^n,\ r\in(\max\{p_+,1\},\fz]$
with $p_+$ as in \eqref{2e10}, $N$ be as in \eqref{2e11} and $s$
as in \eqref{5eq1}.
Then $$\vh=\vah$$ with equivalent quasi-norms.
\end{theorem}

Combining Proposition \ref{3p1} and Theorem \ref{5t1}, the
following result given in \cite[Corollary 3.18]{hlyy} was obtained.

\begin{corollary}\label{5c1}
Let $\va$ and $s$ be as in Theorem \ref{5t1}, $\vp\in(1,\fz)^n$
and $r\in(p_+,\fz]$ with $p_+$ as in \eqref{2e10}.
Then $$\lv=\vah$$ with equivalent quasi-norms.
\end{corollary}

Moreover, the finite atomic characterizations of $\vh$
was also shown by Huang et al.
in \cite[Theorem 5.9]{hlyy} as follows. To state this theorem,
we first recall the following anisotropic mixed-norm
finite atomic Hardy space given in \cite[Definition 5.1]{hlyy}.

\begin{definition}\label{5d3}
Let $\va\in[1,\fz)^n$, $\vp\in(0,\fz)^n$, $r\in (1,\fz]$
and $s$ be as in \eqref{5eq1}. The \emph{anisotropic mixed-norm
finite atomic Hardy space} $\vfah$ is defined to be the set of all
$f\in\cs'(\rn)$ satisfying that there exist $I\in\nn$,
$\{\lz_i\}_{i\in[1,I]\cap\nn}\subset\mathbb{C}$ and
a finite sequence $\{a_i\}_{i\in[1,I]\cap\nn}$ of $(\vp,r,s)$-atoms
supported, respectively, in
$\{B_i\}_{i\in[1,I]\cap\nn}\subset\mathfrak{B}$
such that
\begin{align*}
f=\sum_{i=1}^I\lambda_ia_i
\quad\mathrm{in}\quad\cs'(\rn).
\end{align*}
Moreover, for any $f\in\vfah$, let
\begin{align*}
\|f\|_{\vfah}:=
{\inf}\lf\|\lf\{\sum_{i=1}^{I}
\lf[\frac{|\lz_i|{\mathbf 1}_{B_i}}{\|{\mathbf 1}_{B_i}\|_{\lv}}\r]^
{\underline{p}}\r\}^{1/\underline{p}}\r\|_{\lv},
\end{align*}
where $\underline{p}$ is as in \eqref{2e10} and the
infimum is taken over all decompositions of $f$ as above.
\end{definition}

\begin{remark}
Similarly to Remark \ref{3r.2},
in Definition \ref{5d3} above, we correct the range of
$\underline{p}$ to be $(0,\min\{1,p_-\})$
due to the fact that Lemma \ref{3l1} below holds true only for $\vp\in(1,\fz)^n$.
\end{remark}

\begin{theorem}\label{5t2}
Let $\va\in [1,\fz)^n$, $\vp\in(0,\fz)^n$ and $s$ be as in \eqref{5eq1}.
\begin{enumerate}
\item[{\rm (i)}]
If $r\in(\max\{p_+,1\},\fz)$ with $p_+$ as in
\eqref{2e10}, then $\|\cdot\|_{\vfah}$
and $\|\cdot\|_{\vh}$ are equivalent quasi-norms on $\vfah$;
\item[{\rm (ii)}]
$\|\cdot\|_{\vfahfz}$
and $\|\cdot\|_{\vh}$ are equivalent quasi-norms on
$\vfahfz\cap C(\rn)$, here and thereafter, $C(\rn)$ denotes
the set of all continuous functions on $\rn$.
\end{enumerate}
\end{theorem}

To establish the Littlewood--Paley function characterizations of $\vh$.
Huang et al. \cite[Lemma 4.13]{hlyy} first established the
anisotropic Calder\'{o}n reproducing formula as follows. Indeed, it is
known that the Calder\'{o}n reproducing formulae are bridges to
connect the theory of function spaces and the boundedness of operators.
In what follows, for any $\phi\in\cs(\rn)$,
$\widehat{\phi}$ denotes its \emph{Fourier transform},
namely, for any $\xi\in\rn$,
\begin{align*}
\widehat \phi(\xi) := \int_{\rn} \phi(x) e^{-2\pi\imath x \cdot \xi} \, dx,
\end{align*}
where $\imath:=\sqrt{-1}$.

\begin{lemma}\label{x4l1}
Let $\va\in [1,\fz)^n$ and $s\in\zz_+$.
For any $\varphi\in C_c^{\fz}(\rn)$ satisfying $\supp\varphi\subset B_0,$
$$\int_{\rn}x^\gamma\varphi(x)\,dx=0\ \mathrm{for\ any}\
\gamma\in\zz_+^n\ \mathrm{with}\ |\gamma|\le s,$$
$|\widehat{\varphi}(\xi)|\ge C$
for any $\xi\in\{x\in\rn:\ 2^{-(1+a_+)}\le|x|\le 1\}$,
where $C\in(0,\fz)$ is a constant,
there exists a $\psi\in\cs(\rn)$ such that
\begin{enumerate}
\item[{\rm(i)}] $\supp \widehat{\psi}$
is compact and away from the origin;
\item[{\rm(ii)}] for any $\xi\in\rn\setminus\{\vec{0}_n\}$,
$\sum_{k\in\mathbb{Z}}
\widehat{\psi}(2^{k\va}\xi)\widehat{\varphi}(2^{k\va}\xi)=1$.
\end{enumerate}

Moreover, for any $f\in L^2(\rn)$,
$f=\sum_{k\in\mathbb{Z}}f\ast\psi_k\ast\varphi_k$ in $L^2(\rn)$.
The same holds true in $\cs'(\rn)$ for any $f\in \cs'_0(\rn)$.
\end{lemma}

Let $\va\in [1,\fz)^n$. Assume that $\phi\in\cs(\rn)$
satisfies the same assumptions as $\varphi$ in Lemma \ref{x4l1}
with $s$ as in \eqref{5eq1}. Then, for any $\lambda\in(0,\fz)$ and
$f\in\cs'(\rn)$, the \emph{anisotropic Lusin area function} $S(f)$,
the \emph{anisotropic Littlewood--Paley} $g$-\emph{function} $g(f)$ and
the \emph{anisotropic Littlewood--Paley} $g_\lambda^\ast$-\emph{function}
$g_\lambda^\ast(f)$ are defined, respectively, by setting, for any $x\in\rn$,
\begin{align*}
S(f)(x):=\lf[\sum_{k\in\mathbb{Z}}2^{-k\nu}\int_{B_{\va}(x,2^k)}
\lf|f\ast\phi_{k}(y)\r|^2\,dy\r]^{1/2},
\end{align*}
\begin{align*}
g(f)(x):=\lf[\sum_{k\in\mathbb{Z}}
\lf|f\ast\phi_{k}(x)\r|^2\r]^{1/2}
\end{align*}
and
\begin{align*}
g_\lambda^\ast(f)(x):=
\lf\{\sum_{k\in\mathbb{Z}}2^{-k\nu}\int_{\rn}
\lf[\frac{2^{k}}{2^{k}+|x-y|_{\va}}\r]^{\lambda\nu}
\lf|f\ast\phi_{k}(y)\r|^2\,dy\r\}^{1/2},
\end{align*}
where, for any $k\in \mathbb{Z}$, $\phi_{k}(\cdot)
:=2^{-k\nu}\phi(2^{-k\va}\cdot)$.

Recall that $f\in\cs'(\rn)$ is said to
\emph{vanish weakly at infinity} if, for any $\phi\in\cs(\rn)$,
$f\ast\phi_{k}\to0$ in $\cs'(\rn)$ as $k\to \fz$.
In what follows, we always let $\cs'_0(\rn)$ be the set of all $f\in\cs'(\rn)$
vanishing weakly at infinity.

As an application of the atomic characterizations of $\vh$,
Huang et al. in
\cite[Theorems 4.1 through 4.3]{hlyy} obtained the following
Littlewood--Paley function characterizations of $\vh$
with the help of Lemma \ref{x4l1}.

\begin{theorem}\label{5t3}
Let $\va\in [1,\fz)^n$, $\vp\in(0,\fz)^n$ and
$N$ be as in \eqref{2e11}.
Then $f\in\vh$ if and only if
$f\in\cs'_0(\rn)$ and $S(f)\in\lv$. Moreover,
there exists a positive constant $C$ such that,
for any $f\in\vh$,
$$C^{-1}\|S(f)\|_{\lv}\le\|f\|_{\vh}\le C\|S(f)\|_{\lv}.$$
\end{theorem}

\begin{theorem}\label{5t4}
Let $\va$, $\vp$ and $N$ be as in Theorem \ref{5t3}.
Then $f\in\vh$ if and only if
$f\in\cs'_0(\rn)$ and $g(f)\in\lv$. Moreover,
there exists a positive constant $C$ such that,
for any $f\in\vh$,
$$C^{-1}\|g(f)\|_{\lv}\le\|f\|_{\vh}\le C\|g(f)\|_{\lv}.$$
\end{theorem}

\begin{theorem}\label{5t5}
Let $\va$, $\vp$ and $N$ be as in Theorem \ref{5t3} and
$\lambda\in(1+\frac{2}{{\min\{p_-,2\}}}, \fz)$,
where $p_-$ is as in \eqref{2e10}.
Then $f\in\vh$ if and only if $f\in\cs'_0(\rn)$ and
$g_\lambda^{\ast}(f)\in\lv$. Moreover,
there exists a positive constant $C$ such that,
for any $f\in\vh$,
$$C^{-1}\lf\|g_\lz^\ast(f)\r\|_{\lv}\le\|f\|_{\vh}
\le C\lf\|g_\lz^\ast(f)\r\|_{\lv}.$$
\end{theorem}

We point out that, in the proof of Theorem \ref{5t3}, namely,
in the proof of \cite[Theorem 4.1]{hlyy},
we first need to show that the $\lv$
quasi-norms of the anisotropic Lusin area function $S(f)$ are independent
of the choices of $\varphi$ and $\psi$ as in Lemma \ref{x4l1}.
For this purpose, we denote by $S_\varphi(f)$ and $S_\psi(f)$
the anisotropic Lusin area functions defined, respectively,
by $\varphi$ and $\psi$.

\begin{theorem}\label{4t1}
Let $\vp\in(0,\fz)^n$, $\varphi$ and $\psi$ be as in Lemma \ref{x4l1}
with $s$ as in \eqref{5eq1}.
Then there exists a positive constant $C$ such that, for any $f\in\cs'(\rn)$,
$$C^{-1}\|S_\varphi(f)\|_{\lv}\le\|S_\psi(f)\|_{\lv}\le C\|S_\varphi(f)\|_{\lv}.$$
\end{theorem}

To prove Theorem \ref{4t1}, we first recall the notion of the
anisotropic Hardy--Littlewood maximal operator as follows.

\begin{definition}\label{3d8}
The \emph{anisotropic Hardy--Littlewood maximal function}
$M_{{\rm HL}}^{\va}(f)$ of $f\in L_{\rm loc}^1(\rn)$ is defined by setting,
for any $x\in\rn$,
\begin{align}\label{3e4}
M_{{\rm HL}}^{\va}(f)(x):=\sup_{x\in Q\in\mathfrak{Q}}
\frac1{|Q|}\int_Q\lf|f(y)\r|\,dy,
\end{align}
where $\mathfrak{Q}$ is as in \eqref{x3e2}.
\end{definition}

The following boundedness of the anisotropic Hardy--Littlewood
maximal operator $M_{{\rm HL}}^{\va}$ on mixed Lebesgue spaces
$\lv$ is just \cite[Lemma 3.5]{hlyy}.

\begin{lemma}\label{3l1}
Let $\vp\in (1,\fz)^n$. Then there exists a positive
constant C, depending on $\vp$, such that, for any $f\in \lv$,
\begin{align}\label{3.2}
\lf\|M_{{\rm HL}}^{\va}(f)\r\|_{\lv}\le C\|f\|_{\lv},
\end{align}
where $M_{{\rm HL}}^{\va}$ is as in \eqref{3e4}.
\end{lemma}

\begin{remark}
Observe that, in \cite[Lemma 3.5]{hlyy}, $\vp\in(1,\fz]^n$.
However, Professor Ferenc Weisz pointed out to us that,
when $n:=2$ and $\vp:=(p_1,\infty)$ with $p_1\in(1,\fz)$, \cite[Lemma 3.5]{hlyy}
is not correct. To show this, we find the following counterexample.

Let $\va:=(1,1)$, $I:=\{(x_1,x_2)\in(0,\fz)^2:\ x_1\ge x_2\}$, $\delta:=1-\frac1{p_1}$
and, for any $(x_1,x_2)\in I$, $f(x_1,x_2):=\frac{x_2^{\delta}}{x_1}$;
otherwise, $f\equiv0$. Then, for any $x_2\in\rn$, we have
$$\int_{\rr^2}|f(x_1,x_2)|^{p_1}\,dx_1
=x_2^{\delta p_1}\int_{x_2}^{\fz}\frac{dx_1}{x_1^{p_1}}
=\frac1{p_1-1}.$$
Thus,
\begin{align}\label{3.3}
\|f\|_{L^{(p_1,\fz)}(\rr^2)}=\lf(\frac1{p_1-1}\r)^{1/p_1}.
\end{align}
On another hand, for any $(x_1,x_2)\in I$, let $Q:=[x_1,2x_1]\times[0,x_1]$.
Obviously, $(x_1,x_2)\in Q\subset I$. Therefore, for any $(x_1,x_2)\in I$,
\begin{align*}
\HL^{\va}(f)(x_1,x_2)\ge \frac1{|Q|}\int_Q|f(t_1,t_2)|\,dt_1\,dt_2
=\frac1{x_1^2}\int_0^{x_1}t_2^\delta\,dt_2\int_{x_1}^{2x_2}\frac1{t_1}\,dt_1
\sim x_1^{\delta-1},
\end{align*}
which further implies that
$$\int_{\rr^2}\lf|\HL^{\va}(f)(x_1,x_2)\r|^{p_1}\,dx_1
\ge \int_{I}\lf|\HL^{\va}(f)(x_1,x_2)\r|^{p_1}\,dx_1
\gs \int_{x_2}^{\fz}x_1^{(\delta-1) p_1}\,dx_1
\sim \int_{x_2}^{\fz}x_1^{-1}\,dx_1\sim \fz.$$
From this and \eqref{3.3}, it follows that, in this case,
\eqref{3.2} does not hold true.

Thus, in Lemma \ref{3l1},
we restrict the range of $\vp$ to be $(1,\fz)^n$.
Moreover, due to the fact that Lemma \ref{3l1} does not hold true for some
$p_{i_1}\in(1,\fz)$ and $p_{i_2}=\fz$ with $i_1<i_2$, which was used
in the proof of \cite[Lemma 3.15]{hlyy},
the range of exponent $\widetilde{p}_-$ in \cite{hlyy}, namely, $\underline{p}$
in this survey, should be corrected to be $(0,\min\{1,p_-\})$.
In addition, we point out that if $n:=2$ and $\vp:=(\infty,p_2)$
with $p_2\in(1,\fz)$, then \eqref{3.2} still holds true;
its proof is similar to the proof of \cite[Lemma 3.5]{hlyy}.
\end{remark}

To show Theorem \ref{4t1}, the following three succeeding technical
lemmas are also necessary. Lemmas \ref{4l2} and \ref{4l3} are just,
respectively, \cite[Lemma 3.7]{hlyy} and a consequence of
\cite[Lemma 5.4]{blyz10} with $A$ therein as in \eqref{5eq2}.

\begin{lemma}\label{4l2}
Let $\vp\in (1,\fz)^n$ and $u\in(1,\fz]$.
Then there exists a positive
constant $C$ such that, for any
sequence $\{f_k\}_{k\in\nn}\subset L^1_{\rm loc}(\rn)$,
$$\lf\|\lf\{\sum_{k\in\nn}
\lf[\HL^{\va}(f_k)\r]^u\r\}^{1/u}\r\|_{\lv}
\le C\lf\|\lf(\sum_{k\in\nn}|f_k|^u\r)^{1/u}\r\|_{\lv}$$
with the usual modification made when $u=\fz$,
where $\HL^{\va}$ is as in \eqref{3e4}.
\end{lemma}

\begin{lemma}\label{4l3}
Let $a_-$ and $s$ be, respectively, as in \eqref{2e9} and
\eqref{5eq1} and $\varphi$, $\phi\in\cs(\rn)$ satisfy
that, for any $\az\in\zz_+^n$ with $|\az|\le s$,
$\int_{\rn}\varphi(x)x^\az\,dx=0$ and
$\int_{\rn}\phi(x)x^\az\,dx=0$.
Then there exists a positive constant $C$ such that,
for any $k$, $\ell\in\zz$ and $x\in\rn$,
$$\lf|\varphi_k\ast\psi_\ell(x)\r|\le C 2^{-(s+1)|k-\ell|a_-}
\frac{2^{(k\vee \ell)(s+1)\nu a_-}}{[2^{(k\vee \ell)}
+|x|_{\va}]^{(s+1)\nu a_-+\nu}},$$
here and thereafter, for any $k$, $\ell\in\zz$, $k\vee\ell:=\max\{k,\ell\}$.
\end{lemma}

The following lemma plays a key role
in the proof of Theorem \ref{4t1}.

\begin{lemma}\label{4l1}
Let $s$ be as in \eqref{5eq1} and $r\in(\frac{\nu}{\nu+(s+1)a_-},1]$. Then
there exists a positive constant $C$ such that, for any
$k,$ $\ell\in\zz$, $\{c_{Q_{\va}}\}_{Q_{\va}\in\mathfrak{Q}}\subset [0,\fz)$
with $\mathfrak{Q}$ as in \eqref{x3e2}, and $x\in\rn$,
\begin{align*}
&\sum_{l(Q_{\va})=2^{k-1}}|Q_{\va}|
\frac{2^{(k\vee\ell)(s+1) a_-}}{\lf[2^{(k\vee\ell)}
+|x-z_{Q_{\va}}|_{\va}\r]^{(s+1)a_-+\nu}}c_{Q_{\va}}\\
&\hs\hs \le C 2^{-[k-(k\vee\ell)](1/r-1)\nu}\lf\{M_{\rm HL}^{\va}
\lf(\sum_{l(Q_{\va})=2^{k-1}}\lf[c_{Q_{\va}}\r]^r\mathbf{1}_{Q_{\va}}\r)(x)\r\}^{1/r},
\end{align*}
where $l(Q_{\va})$ denotes the side-length of $Q_{\va}$
and $z_{Q_{\va}}\in Q_{\va}$.
\end{lemma}

\begin{proof}
We show this lemma by the following two cases.

\emph{Case I)} $\ell\ge k$. In this case, applying the well-known inequality that,
for any $\{\lz_i\}_{i\in\nn}\subset\mathbb{C}$ and $\theta\in[0,1]$,
\begin{align}\label{x3e10}
\lf(\sum_{i\in\nn}|\lz_i|\r)^{\theta}\le \sum_{i\in\nn}|\lz_i|^{\theta},
\end{align}
and the Tonelli theorem,
we conclude that, for any $x\in\rn$,
\begin{align}\label{x3e7}
&\lf\{\sum_{l(Q_{\va})=2^{k-1}}|Q_{\va}|\frac{2^{\ell (s+1)a_-}}{\lf[2^{\ell}
+|x-z_{Q_{\va}}|_{\va}\r]^{(s+1)a_-+\nu}}c_{Q_{\va}}\r\}^r\noz\\
&\hs\hs\ls \sum_{l(Q_{\va})=2^{k-1}}\int_{Q_{\va}}|Q_{\va}|^{r-1}
\lf[\frac1{2^{\ell}+|x-z_{Q_{\va}}|_{\va}}\r]^{\nu r}
\lf[\frac{2^{\ell}}{2^{\ell}
+|x-z_{Q_{\va}}|_{\va}}\r]^{(s+1)ra_-}\lf(c_{Q_{\va}}\r)^r\,dy\noz\\
&\hs\hs\ls \int_{\rn}\lf[\frac1{|B_{\va}(x,2^k)|}\r]^{1-r}
\lf[\frac1{2^{\ell}+|x-y|_{\va}}\r]^{\nu r}
\lf[\frac{2^{\ell}}{2^{\ell}+|x-y|_{\va}}\r]^{(s+1)ra_-}\noz\\
&\quad\hs\hs\times\sum_{l(Q_{\va})=2^{k-1}}\lf[(c_{Q_{\va}})^r
\mathbf{1}_{Q_{\va}}(y)\r]\,dy\noz\\
&\hs\hs\sim \textrm{I}_1+\textrm{I}_2,
\end{align}
where, for any $x\in\rn$,
\begin{align*}
\textrm{I}_1&:=\int_{|x-y|_{\va}\le 2^\ell}\lf[\frac1{|B_{\va}(x,2^k)|}\r]^{1-r}
\lf[\frac1{2^{\ell}+|x-y|_{\va}}\r]^{\nu r}
\lf[\frac{2^{\ell}}{2^{\ell}+|x-y|_{\va}}\r]^{(s+1)ra_-}\\
&\hs\hs\times
\sum_{l(Q_{\va})=2^{k-1}}\lf[(c_{Q_{\va}})^r\mathbf{1}_{Q_{\va}}(y)\r]\,dy
\end{align*}
and
\begin{align*}
\textrm{I}_2&:=\sum_{j\in\nn}\int_{ 2^{j+\ell-1}<|x-y|_{\va}\le 2^{j+\ell}}
\lf[\frac1{|B_{\va}(x,2^k)|}\r]^{1-r}
\lf[\frac1{2^{\ell}+|x-y|_{\va}}\r]^{\nu r}
\lf[\frac{2^{\ell}}{2^{\ell}+|x-y|_{\va}}\r]^{(s+1)ra_-}\\
&\hs\hs\times\sum_{l(Q_{\va})=2^{k-1}}\lf[(c_{Q_{\va}})^r
\mathbf{1}_{Q_{\va}}(y)\r]\,dy.
\end{align*}
For $\textrm{I}_1$, we have
\begin{align}\label{x3e8}
\textrm{I}_1&\ls\int_{|x-y|_{\va}\le 2^\ell}\lf[\frac1{|B_{\va}(x,2^k)|}\r]^{1-r}
\lf(\frac1{2^{\ell}}\r)^{\nu r}
\sum_{l(Q_{\va})=2^{k-1}}\lf[(c_{Q_{\va}})^r\mathbf{1}_{Q_{\va}}(y)\r]\,dy\noz\\
&\sim \frac1{|B_{\va}(x,2^\ell)|}
\int_{|x-y|_{\va}\le 2^\ell}\lf[\frac{|B_{\va}(x,2^\ell)|}
{|B_{\va}(x,2^k)|}\r]^{1-r}
\sum_{l(Q_{\va})=2^{k-1}}\lf[(c_{Q_{\va}})^r\mathbf{1}_{Q_{\va}}(y)\r]\,dy\noz\\
&\ls 2^{(\ell-k)(1-r)\nu}M_{\rm HL}^{\va}
\lf(\sum_{l(Q_{\va})=2^{k-1}}\lf[c_{Q_{\va}}\r]^r\mathbf{1}_{Q_{\va}}\r)(x).
\end{align}
For $\textrm{I}_2$, from the fact that $2^{j+\ell-1}<|x-y|_{\va}\le 2^{j+\ell}$,
it follows that, for any $x\in\rn$,
\begin{align*}
\textrm{I}_2&\ls\sum_{j\in\nn}\int_{ 2^{j+\ell-1}<|x-y|_{\va}\le 2^{j+\ell}}
\lf[\frac1{|B_{\va}(x,2^k)|}\r]^{1-r}\lf(\frac1{2^{j+\ell}}\r)^{\nu r}
\lf(\frac1{2^{j}}\r)^{(s+1)ra_-}\\
&\hs\hs\times\sum_{l(Q_{\va})=2^{k-1}}\lf[(c_{Q_{\va}})^r\mathbf{1}_{Q_{\va}}(y)\r]\,dy.\\
\end{align*}
Thus,
\begin{align*}
\textrm{I}_2&\ls \sum_{j\in\nn}2^{-j(s+1)ra_-}\frac1{|B_{\va}(x,2^{j+\ell})|}
\int_{2^{j+\ell-1}<|x-y|_{\va}\le 2^{j+\ell}}
\lf[\frac{|B_{\va}(x,2^{j+\ell})|}{|B_{\va}(x,2^k)|}\r]^{1-r}\\
&\hs\hs\times\sum_{l(Q_{\va})=2^{k-1}}\lf[(c_{Q_{\va}})^r\mathbf{1}_{Q_{\va}}(y)\r]\,dy\\
&\ls \sum_{j\in\nn}2^{-j[(s+1)ra_--(1-r)\nu]}2^{(\ell-k)(1-r)\nu}M_{\rm HL}^{\va}
\lf(\sum_{l(Q_{\va})=2^{k-1}}\lf[c_{Q_{\va}}\r]^r\mathbf{1}_{Q_{\va}}\r)(x),
\end{align*}
which, together with \eqref{x3e7}, \eqref{x3e8} and the
fact that $r>\frac{\nu}{\nu+(s+1)a_-}$, further implies that
\begin{align}\label{x3e9}
&\lf\{\sum_{l(Q_{\va})=2^{k-1}}|Q_{\va}|\frac{2^{\ell (s+1)a_-}}{\lf[2^{\ell}
+|x-z_{Q_{\va}}|_{\va}\r]^{(s+1)a_-+\nu}}c_{Q_{\va}}\r\}^r\noz\\
&\hs\hs\ls  2^{(\ell-k)(1-r)\nu}M_{\rm HL}^{\va}
\lf(\sum_{l(Q_{\va})=2^{k-1}}\lf[c_{Q_{\va}}\r]^r\mathbf{1}_{Q_{\va}}\r)(x).
\end{align}

\emph{Case II)} $\ell<k$. In this case, by \eqref{x3e10} and
the Tonelli theorem again, we conclude
that, for any $x\in\rn$,
\begin{align}\label{x3e12}
&\lf\{\sum_{l(Q_{\va})=2^{k-1}}|Q_{\va}|\frac{2^{k (s+1)a_-}}{\lf[2^{k}
+|x-z_{Q_{\va}}|_{\va}\r]^{(s+1)a_-+\nu}}c_{Q_{\va}}\r\}^r\noz\\
&\hs\hs\ls \sum_{l(Q_{\va})=2^{k-1}}\int_{Q_{\va}}|Q_{\va}|^{r-1}
\lf[\frac1{2^{k}+|x-z_{Q_{\va}}|_{\va}}\r]^{\nu r}
\lf[\frac{2^{k}}{2^{k}
+|x-z_{Q_{\va}}|_{\va}}\r]^{(s+1)ra_-}\lf(c_{Q_{\va}}\r)^r\,dy\noz\\
&\hs\hs\ls \int_{\rn}\lf[\frac1{|B_{\va}(x,2^k)|}\r]^{1-r}
\lf[\frac1{2^{k}+|x-y|_{\va}}\r]^{\nu r}
\lf[\frac{2^{k}}{2^{k}+|x-y|_{\va}}\r]^{(s+1)ra_-}\noz\\
&\quad\hs\hs\times\sum_{l(Q_{\va})=2^{k-1}}\lf[(c_{Q_{\va}})^r\mathbf{1}_{Q_{\va}}(y)\r]\,dy\noz\\
&\hs\hs\sim \textrm{I}_3+\textrm{I}_4,
\end{align}
where, for any $x\in\rn$,
\begin{align*}
\textrm{I}_3&:=
\int_{|x-y|_{\va}\le 2^k}\lf[\frac1{|B_{\va}(x,2^k)|}\r]^{1-r}
\lf[\frac1{2^{k}+|x-y|_{\va}}\r]^{\nu r}
\lf[\frac{2^{k}}{2^{k}+|x-y|_{\va}}\r]^{(s+1)ra_-}\noz\\
&\hs\hs\times
\sum_{l(Q_{\va})=2^{k-1}}\lf[(c_{Q_{\va}})^r\mathbf{1}_{Q_{\va}}(y)\r]\,dy
\end{align*}
and
\begin{align*}
\textrm{I}_4&:=
\sum_{j\in\nn}\int_{2^{k+j-1}<|x-y|_{\va}\le
2^{k+j}}\lf[\frac1{|B_{\va}(x,2^k)|}\r]^{1-r}
\lf[\frac1{2^{k}+|x-y|_{\va}}\r]^{\nu r}
\lf[\frac{2^{k}}{2^{k}+|x-y|_{\va}}\r]^{(s+1)ra_-}\noz\\
&\hs\hs\times
\sum_{l(Q_{\va})=2^{k-1}}\lf[(c_{Q_{\va}})^r\mathbf{1}_{Q_{\va}}(y)\r]\,dy.
\end{align*}
For $\textrm{I}_3$, we have
\begin{align}\label{x3e13}
\textrm{I}_3&\ls\int_{|x-y|_{\va}\le 2^k}\lf[\frac1{|B_{\va}(x,2^k)|}\r]^{1-r}
\lf(\frac1{2^k}\r)^{\nu r}
\sum_{l(Q_{\va})=2^{k-1}}\lf[(c_{Q_{\va}})^r\mathbf{1}_{Q_{\va}}(y)\r]\,dy\noz\\
&\sim \frac1{|B_{\va}(x,2^k)|}\int_{|x-y|_{\va}\le 2^k}
\sum_{l(Q_{\va})=2^{k-1}}\lf[(c_{Q_{\va}})^r\mathbf{1}_{Q_{\va}}(y)\r]\,dy\noz\\
&\ls M_{\rm HL}^{\va}
\lf(\sum_{l(Q_{\va})=2^{k-1}}\lf[c_{Q_{\va}}\r]^r\mathbf{1}_{Q_{\va}}\r)(x).
\end{align}
For $\textrm{I}_4$, from the facts that $2^{j+k-1}<|x-y|_{\va}\le 2^{j+k}$
and $r>\frac{\nu}{\nu+(s+1)a_-}$, it follows that, for any $x\in\rn$,
\begin{align}\label{x3e14}
\textrm{I}_4&\ls\sum_{j\in\nn}\int_{ 2^{j+k-1}<|x-y|_{\va}\le 2^{j+k}}
\lf[\frac1{|B_{\va}(x,2^k)|}\r]^{1-r}\lf(\frac1{2^{j+k}}\r)^{\nu r}
\lf(\frac1{2^{j}}\r)^{(s+1)ra_-}\noz\\
&\hs\hs\times\sum_{l(Q_{\va})=2^{k-1}}\lf[(c_{Q_{\va}})^r
\mathbf{1}_{Q_{\va}}(y)\r]\,dy\noz\\
&\sim \sum_{j\in\nn}2^{-j(s+1)ra_-}\frac1{|B_{\va}(x,2^{j+k})|}
\int_{2^{j+k-1}<|x-y|_{\va}\le 2^{j+k}}
\lf[\frac{|B_{\va}(x,2^{j+k})|}{|B_{\va}(x,2^k)|}\r]^{1-r}\noz\\
&\hs\hs\times\sum_{l(Q_{\va})=2^{k-1}}
\lf[(c_{Q_{\va}})^r\mathbf{1}_{Q_{\va}}(y)\r]\,dy\noz\\
&\ls \sum_{j\in\nn}2^{-j[(s+1)ra_--(1-r)\nu]}M_{\rm HL}^{\va}
\lf(\sum_{l(Q_{\va})=2^{k-1}}\lf[c_{Q_{\va}}\r]^r\mathbf{1}_{Q_{\va}}\r)(x)\noz\\
&\ls M_{\rm HL}^{\va}
\lf(\sum_{l(Q_{\va})=2^{k-1}}\lf[c_{Q_{\va}}\r]^r\mathbf{1}_{Q_{\va}}\r)(x).
\end{align}
From \eqref{x3e12}, \eqref{x3e13} and \eqref{x3e14}, we deduce that,
for any $x\in\rn$,
\begin{align*}
\lf\{\sum_{l(Q_{\va})=2^{k-1}}|Q_{\va}|\frac{2^{k (s+1)a_-}}{\lf[2^{k}
+|x-z_{Q_{\va}}|_{\va}\r]^{(s+1)a_-+\nu}}c_{Q_{\va}}\r\}^r
\ls M_{\rm HL}^{\va}
\lf(\sum_{l(Q_{\va})=2^{k-1}}\lf[c_{Q_{\va}}\r]^r\mathbf{1}_{Q_{\va}}\r)(x).
\end{align*}

Combining this and \eqref{x3e9}, we further conclude that, for
any given $r\in(\frac{\nu}{\nu+(s+1)a_-},1]$, any
$k,$ $\ell\in\zz$ and $x\in\rn$,
\begin{align*}
&\sum_{l(Q_{\va})=2^{k-1}}|Q_{\va}|\frac{2^{(k\vee\ell)(s+1)a_-}}{\lf[2^{(k\vee\ell)}
+|x-z_{Q_{\va}}|_{\va}\r]^{(s+1)a_-+\nu}}c_{Q_{\va}}\\
&\hs\hs\ls  2^{-[k-(k\vee\ell)](1/r-1)\nu}\lf\{M_{\rm HL}^{\va}
\lf(\sum_{l(Q_{\va})=2^{k-1}}\lf[c_{Q_{\va}}\r]^r\mathbf{1}_{Q_{\va}}\r)(x)\r\}^{1/r}.
\end{align*}
This finishes the proof of Lemma \ref{4l1}.
\end{proof}

Now we prove Theorem \ref{4t1}.

\begin{proof}[Proof of Theorem \ref{4t1}]
By symmetry, to finish the proof of this theorem,
we only need to show that, for any $f\in\cs'(\rn)$,
\begin{align}\label{x3e11}
\|S_\varphi(f)\|_{\lv}\ls\|S_\psi(f)\|_{\lv}.
\end{align}
To this end, for any $k\in\zz$ and $z\in\rn$, let
$$m_{\psi_k}(f)(z):=\lf[2^{-k\nu}\int_{B_{\va}(z,2^k)}
\lf|f\ast\psi_k(y)\r|^2\,dy\r]^{1/2}$$
and, for any $\ell\in\zz$, $x\in\rn$ and $y\in B_{\va}(x,2^\ell)$,
$$E_{\varphi_\ell}(f)(y):=f\ast \varphi_{\ell}(y).$$ Then, from Lemma
\ref{x4l1} and the Lebesgue dominated convergence theorem,
it follows that, for any $\ell\in\zz$, $x\in\rn$ and $y\in B_{\va}(x,2^\ell)$,
\begin{align}\label{x3e5}
E_{\varphi_\ell}(f)(y)&=\sum_{k\in\zz}f\ast\psi_k\ast\varphi_k\ast\varphi_{\ell}(y)
=\sum_{k\in\zz}\int_{\rn}f\ast\psi_k(z)\varphi_k\ast\varphi_{\ell}(y-z)\,dz\noz\\
&=\sum_{k\in\zz}\sum_{l(Q_{\va})=2^{k-1}}\int_{Q_{\va}}
f\ast\psi_k(z)\varphi_k\ast\varphi_{\ell}(y-z)\,dz
\end{align}
in $\cs'(\rn)$, where, for any $Q_{\va}\in\mathfrak{Q}$ with $\mathfrak{Q}$ as
in \eqref{x3e2}, $l(Q_{\va})$ denotes the side-length of $Q_{\va}$.

On another hand, by Lemma \ref{4l3}, we know that,
for any $k,$ $\ell\in\zz$, $z\in Q_{\va}$,
$x\in\rn$ and $y\in B_{\va}(x,2^\ell)$,
\begin{align*}
\lf|\varphi_k\ast\varphi_{\ell}(y-z)\r|\ls 2^{-(s+1)|k-\ell|a_-}
\frac{2^{(k\vee\ell)(s+1)a_-}}
{\lf[2^{(k\vee\ell)}+|y-z|_{\va}\r]^{(s+1)a_-+\nu}},
\end{align*}
where $s$ is as in \eqref{5eq1}.
From this and the fact that
$y\in B_{\va}(x,2^\ell)$, we further deduce that
there exists $z_{Q_{\va}}\in Q_{\va}$
such that, for any $k,$ $\ell\in\zz$ and $z\in Q_{\va}$,
\begin{align}\label{x3e4}
\lf|\varphi_k\ast\varphi_{\ell}(y-z)\r|
\ls 2^{-(s+1)|k-\ell|a_-}\frac{2^{(k\vee\ell)(s+1)a_-}}
{\lf[2^{(k\vee\ell)}+|x-z_{Q_{\va}}|_{\va}\r]^{(s+1)a_-+\nu}}.
\end{align}
In addition, observe that, for any $z\in Q_{\va}$,
$Q_{\va}\subset B_{\va}(z,2^{k})$
and $|Q_{\va}|\sim| B_{\va}(z,2^{k})|$ and, therefore,
via the H\"{o}lder inequality, we have
$$\lf|\frac1{|Q_{\va}|}\int_{Q_{\va}}f\ast\psi_k(y)\,dy\r|
\ls \lf[\frac1{|B_{\va}(z,2^{k})|}\int_{B_{\va}(z,2^{k})}
\lf|f\ast\psi_k(y)\r|^2\,dy\r]^{1/2}
\sim m_{\psi_k}(f)(z),$$
which further implies that
$$\lf|\frac1{|Q_{\va}|}\int_{Q_{\va}}f\ast\psi_k(y)\,dy\r|
\ls \inf_{z\in Q_{\va}} m_{\psi_k}(f)(z).$$
By this, \eqref{x3e5}, \eqref{x3e4} and Lemma \ref{4l1}, we find that,
for any given $r\in(\frac{\nu}{\nu+(s+1)a_-},1],$ any $\ell\in\zz$,
$x\in\rn$ and $y\in B_{\va}(x,2^\ell)$,
\begin{align}\label{x3e6}
\lf|E_{\varphi_\ell}(f)(y)\r|&\ls\sum_{k\in\zz}2^{-(s+1)|k-\ell|a_-}\sum_{l(Q_{\va})
=2^{k-1}}|Q_{\va}|\frac{2^{(k\vee\ell)(s+1)a_-}}{\lf[2^{(k\vee\ell)}
+|x-z_{Q_{\va}}|_{\va}\r]^{(s+1)a_-+\nu}}
\inf_{z\in Q_{\va}} m_{\psi_k}(f)(z)\noz\\
&\ls \sum_{k\in\zz}2^{-(s+1)|k-\ell|a_-}
2^{-[k-(k\vee\ell)](1/r-1)\nu}\noz\\
&\quad\times\lf\{M_{\rm HL}^{\va}
\lf(\sum_{l(Q_{\va})=2^{k-1}}\inf_{z\in Q_{\va}}\lf[
m_{\psi_k}(f)(z)\r]^r\mathbf{1}_{Q_{\va}}\r)(x)\r\}^{1/r}.
\end{align}
Due to the fact that $s$ is as in \eqref{5eq1}, we can
choose $r$ such that $$r\in\lf(\frac{\nu}{\nu+(s+1)a_-},\min\{1,p_-\}\r)$$
with $p_-$ as in \eqref{2e10}.
Then, from \eqref{x3e6}, we deduce that, for any $x\in\rn$,
\begin{align*}
[S_{\varphi}(f)(x)]^2&=\sum_{\ell\in\mathbb{Z}}2^{-\ell\nu}\int_{B_{\va}(x,2^\ell)}
\lf|E_{\varphi_\ell}(f)(y)\r|^2\,dy\\
&\ls \sum_{\ell\in\mathbb{Z}}\Bigg[\sum_{k\in\zz}2^{-(s+1)|k-\ell|a_-}
2^{-[k-(k\vee\ell)](1/r-1)\nu}\\
&\quad\lf.\times\lf\{M_{\rm HL}^{\va}
\lf(\sum_{l(Q_{\va})=2^{k-1}}\inf_{z\in Q_{\va}}\lf[
m_{\psi_k}(f)(z)\r]^r\mathbf{1}_{Q_{\va}}\r)(x)\r\}^{1/r}\r]^2,
\end{align*}
which, together with the H\"{o}lder inequality
and the fact that $r>\frac{\nu}{\nu+(s+1)a_-}$,
implies that
\begin{align*}
[S_{\varphi}(f)(x)]^2
&\ls \sum_{\ell\in\mathbb{Z}}\sum_{k\in\zz}2^{-(s+1)|k-\ell|a_-}
2^{-[k-(k\vee\ell)](1/r-1)\nu}\\
&\quad\times\lf\{M_{\rm HL}^{\va}
\lf(\sum_{l(Q_{\va})=2^{k-1}}\inf_{z\in Q_{\va}}\lf[
m_{\psi_k}(f)(z)\r]^r\mathbf{1}_{Q_{\va}}\r)(x)\r\}^{2/r}\\
&\ls \sum_{k\in\zz}\lf\{M_{\rm HL}^{\va}
\lf(\sum_{l(Q_{\va})=2^{k-1}}\inf_{z\in Q_{\va}}\lf[
m_{\psi_k}(f)(z)\r]^r\mathbf{1}_{Q_{\va}}\r)(x)\r\}^{2/r}\\
&\ls \sum_{k\in\zz}\lf\{M_{\rm HL}^{\va}
\lf(\lf[m_{\psi_k}(f)\r]^r\r)(x)\r\}^{2/r}.
\end{align*}
Therefore, by the fact that $r<p_-$
and Lemma \ref{4l2}, we find that
\begin{align*}
\lf\|S_{\varphi}(f)\r\|_{\lv}
&\ls \lf\|\lf(\sum_{k\in\zz}\lf\{M_{\rm HL}^{\va}
\lf(\lf[m_{\psi_k}(f)\r]^r\r)\r\}^{2/r}\r)^{r/2}\r\|_{L^{\vp/r}(\rn)}^{1/r}\\
&\ls \lf\|\lf(\sum_{k\in\zz}\lf[m_{\psi_k}(f)\r]^2\r)^{1/2}\r\|_{\lv}
\sim \lf\|S_{\psi}(f)\r\|_{\lv},
\end{align*}
which implies \eqref{x3e11} holds true and
hence completes the proof of Theorem \ref{4t1}.
\end{proof}

We point out that, in the proof of the sufficiency of \cite[Theorem 4.1]{hlyy},
there exist some errors and gaps. Therefore we give a revised proof as follows.

To show the sufficiency of Theorem \ref{5t3}, the following
lemma is necessary.

\begin{lemma}\label{3l9}
Let $\va\in[1,\fz)^n$, $\vp\in(0,\fz)^n$, $\bz\in(0,\fz)$,
$\vaz\in(0,\fz)$, $\kappa\in(0,\fz)$ and $r\in[1,\fz]\cap(p_+,\fz]$
with $p_+$ as in \eqref{2e10}. Assume that
$\{\lz_i\}_{i\in\nn}\subset\mathbb{C}$, $\{B_i\}_{i\in\nn}\subset\mathfrak{B}$
and $\{m_i^{\vaz,\kappa}\}_{i\in\nn}\subset L^r(\rn)$ satisfy that,
for any $\vaz\in(0,\fz)$, $\kappa\in(0,\fz)$ and  $i\in\nn$,
$\supp m_i^{\vaz,\kappa}\subset B_i^{(\bz)}$ with $B_i^{(\bz)}$ as in \eqref{2e2'},
\begin{align}\label{3e37}
\lf\|m_i^{\vaz,\kappa}\r\|_{L^r(\rn)}
\le\frac{|B_i|^{1/r}}{\|\mathbf{1}_{B_i}\|_{\lv}}
\end{align}
and
\begin{align*}
\lf\|\lf\{\sum_{i\in\nn}
\lf[\frac{|\lz_i|\mathbf{1}_{B_i}}{\|\mathbf{1}_{B_i}\|_{\lv}}\r]^
{{\underline{p}}}\r\}^{1/{\underline{p}}}\r\|_{\lv}<\fz.
\end{align*}
Then, for any $\kappa\in(0,\fz)$,
\begin{align}\label{3e38}
\lf\|\liminf_{\vaz\to0^+}\lf[\sum_{i\in\nn}\lf|\lz_im_i^{\vaz,\kappa}\r|^{{\underline{p}}}\r]
^{1/{\underline{p}}}\r\|_{\lv}
\le C\lf\|\lf\{\sum_{i\in\nn}
\lf[\frac{|\lz_i|\mathbf{1}_{B_i}}{\|\mathbf{1}_{B_i}\|_{\lv}}\r]^
{{\underline{p}}}\r\}^{1/{\underline{p}}}\r\|_{\lv},
\end{align}
where $\underline{p}$ is as in \eqref{2e10} and $C$ is a
positive constant independent of $\lz_i$, $B_i$,
$m_i^{\vaz,\kappa}$, $\vaz$ and $\kappa$.
\end{lemma}

\begin{proof}
By Theorem \ref{x2t3}, we find that there
exists some $g\in L^{(\vp/{\underline{p}})'}(\rn)$
with $\|g\|_{L^{(\vp/{\underline{p}})'}(\rn)}=1$ such that,
for any $\kappa\in(0,\fz)$,
\begin{align*}
\lf\|\liminf_{\vaz\to0^+}\lf[\sum_{i\in\nn}\lf|
\lz_im_i^{\vaz,\kappa}\r|^{{\underline{p}}}\r]
^{1/{\underline{p}}}\r\|_{\lv}^{\underline{p}}
&=\lf\|\liminf_{\vaz\to0^+}\sum_{i\in\nn}\lf|
\lz_im_i^{\vaz,\kappa}\r|^{{\underline{p}}}
\r\|_{L^{\vp/{\underline{p}}}(\rn)}\\
&\ls \int_{\rn}\liminf_{\vaz\to0^+}\sum_{i\in\nn}
\lf|\lz_im_i^{\vaz,\kappa}(x)\r|^{{\underline{p}}}\lf|g(x)\r|\,dx.
\end{align*}
From this, the Fatou lemma, the Tonelli theorem, the H\"{o}lder
inequality and \eqref{3e37}, we deduce that, for any $\kappa\in(0,\fz)$
and $r\in[1,\fz]\cap(p_+,\fz]$,
\begin{equation*}\begin{aligned}
\int_{\rn}\liminf_{\vaz\to0^+}\sum_{i\in\nn}\lf|
\lz_im_i^{\vaz,\kappa}(x)\r|^{{\underline{p}}}\lf|g(x)\r|\,dx
&\le \liminf_{\vaz\to0^+} \sum_{i\in \nn}|\lz_i|^{\underline{p}}
\lf\|m_i^{\vaz,\kappa}\r\|_{L^r (\rn)}^{\underline{p}}
\lf\|g\r\|_{L^{(r/{\underline{p}})'}(B_i^{(\bz)})}\\
&\le \sum_{i\in \nn}\frac{|\lz_i|
^{\underline{p}}|B_i|^{{\underline{p}}/r}}
{\|\mathbf{1}_{B_i}\|_{\lv}^{\underline{p}}}
\lf\|g\r\|_{L^{(r/{\underline{p}})'}(B_i^{(\bz)})}\\
&\ls \sum_{i\in \nn}\frac{|\lz_i|^{\underline{p}}|B_i^{(\bz)}|}
{\|\mathbf{1}_{B_i}\|_{\lv}^{\underline{p}}}\inf_{z\in B_i^{(\bz)}}
\lf[M_{\rm HL}^{\va}\lf(|g|^{(r/{\underline{p}})'}\r)(z)
\r]^{1/{(r/{\underline{p}})'}}\\
&\ls \int_{\rn} \sum_{i\in \nn}\frac{|\lz_i|^{\underline{p}}
\mathbf{1}_{B_i^{(\bz)}}(x)}{\|\mathbf{1}_{B_i}\|_{\lv}^{\underline{p}}}
\lf[M_{\rm HL}^{\va}\lf(|g|^{(r/{\underline{p}})'}\r)(x)\r]
^{1/{(r/{\underline{p}})'}}\,dx,
\end{aligned}\end{equation*}
which, together with Theorem \ref{x2t7}, \cite[Remark 3.8]{hlyy},
Lemma \ref{3l1} and the fact that $r\in (p_+,\fz]$,
further implies that
\begin{align}\label{3.4}
&\int_{\rn}\liminf_{\vaz\to0^+}\sum_{i\in\nn}
\lf|\lz_im_i^{\vaz,\kappa}(x)\r|^{{\underline{p}}}\lf|g(x)\r|\,dx\noz\\
&\hs\hs\ls \lf\|\sum_{i\in \nn}\frac{|\lz_i|^{\underline{p}}
\mathbf{1}_{B_i}}{\|\mathbf{1}_{B_i}\|_{\lv}^{\underline{p}}}
\r\|_{L^{\vp/{\underline{p}}}(\rn)}
\lf\|\lf[M_{\rm HL}^{\va}\lf(|g|^{(r/{\underline{p}})'}\r)\r]
^{1/{(r/{\underline{p}})'}}
\r\|_{L^{(\vp/{\underline{p}})'}(\rn)}\noz\\
&\hs\hs\ls \lf\|\lf\{\sum_{i\in\nn}
\lf[\frac{|\lz_i|\mathbf{1}_{B_i}}{\|\mathbf{1}_{B_i}
\|_{\lv}}\r]^{{\underline{p}}}\r\}^{1/{\underline{p}}}
\r\|_{L^{\vp}(\rn)}^{\underline{p}}
\lf\|g\r\|_{L^{(\vp/{\underline{p}})'}(\rn)}.
\end{align}
By this and the fact that $\lf\|g\r\|_{L^{(\vp/{\underline{p}})'}(\rn)}=1$,
we know that \eqref{3e38} holds true.
This finishes the proof of Lemma \ref{3l9}.
\end{proof}

\begin{remark}
Let $\vp:=(p_1,\ldots,p_n)\in(0,\fz)^n$.
Notice that, if $\underline{p}:=\min\{1,p_-\}$ with
$p_-$ as in \eqref{2e10}, then there may exist some
$p_{i_0}$ such that $p_{i_0}/\underline{p}=1$ and hence
$(p_{i_0}/\underline{p})'=\fz$. By this and the fact that
Lemma \ref{3l1} holds true only for $\vp\in(1,\fz)^n$, we know that
\eqref{3.4} does not hold true in this case.
Thus, we need to restrict the range of $\underline{p}$ to be $(0,\min\{1,p_-\})$.
\end{remark}

To prove the sufficiency of Theorem \ref{5t3},
we also need the following lemma, which is just \cite[Lemma 4.9]{hlyy}.

\begin{lemma}\label{4l2x}
Let $\va\in [1,\fz)^n$. Then there exists a set
$$\mathcal{Q}:=\lf\{Q_\alpha^k\subset\rn:\ k\in\mathbb{Z},
\,\alpha\in E_k\r\}$$
of open subsets, where $E_k$ is some index set, such that
\begin{enumerate}
\item[{\rm (i)}] for any $k\in\zz$,
$\lf|\rn\setminus\bigcup_{\alpha}Q_\alpha^k\r|=0$
and, when $\alpha\neq\beta$,
$Q_\alpha^k\cap Q_\beta^k=\emptyset$;
\item[{\rm(ii)}] for any $\alpha,\ \beta,\ k,\ \ell$ with $\ell\ge k$,
either $Q_\alpha^k\cap Q_\beta^\ell=\emptyset$ or
$Q_\alpha^\ell\subset Q_\beta^k$;
\item[{\rm(iii)}] for any $(\ell,\beta)$ and $k<\ell$,
there exists a unique $\alpha$ such that
$Q_\beta^\ell\subset Q_\alpha^k$;
\item[{\rm(iv)}] there exist some $w\in\zz\setminus\zz_+$
and $u\in\nn$ such that, for any $Q_\alpha^k$
with $k\in\mathbb{Z}$ and $\alpha\in E_k$,
there exists $x_{Q_\alpha^k}\in Q_\alpha^k$
such that, for any $x\in Q_\alpha^k$,
$$x_{Q_\alpha^k}+2^{(wk-u)\va}B_0
\subset Q_\alpha^k\subset x+2^{(wk+u)\va}B_0,$$
where $B_0$ denotes the unit ball of $\rn$.
\end{enumerate}
\end{lemma}

In what follows, we call
$\mathcal{Q}:=
\{Q_\alpha^k\}_{k\in\mathbb{Z},\,\alpha\in E_k}$
from Lemma \ref{4l2x} \emph{dyadic cubes} and
$k$ the \emph{level}, denoted by $\ell(Q_\alpha^k)$,
of the dyadic cube $Q_\alpha^k$
for any $k\in\mathbb{Z}$ and $\alpha\in E_k$.

We now prove the sufficiency of Theorem \ref{5t3}.

\begin{proof}[Proof of the sufficiency of Theorem \ref{5t3}]
Let $f\in\cs'_0(\rn)$ and $S(f)\in\lv$ with $\vp\in(0,\fz)^n$.
Then, from Theorem \ref{4t1}, it follows that $S_{\psi}(f)\in\lv$.
Thus, we only need to prove that $f\in\vh$ and
\begin{align}\label{4e7}
\|f\|_{\vh}\ls\|S_{\psi}(f)\|_{\lv}.
\end{align}
To this end, for any $k\in\mathbb{Z}$, let
$\Omega_k:=\{x\in\rn:\ S_{\psi}(f)(x)>2^k\}$ and
$$\mathcal{Q}_k:=\lf\{Q\in\mathcal{Q}:
\ |Q\cap\Omega_k|>\frac{|Q|}2\ \ {\rm and}\
\ |Q\cap\Omega_{k+1}|\le\frac{|Q|}2\r\}.$$
It is easy to see that, for any $Q\in\mathcal{Q}$,
there exists a unique $k\in\mathbb{Z}$
such that $Q\in\mathcal{Q}_k$.
For any given $k\in\zz$, denote by $\{Q_i^k\}_i$ the collection of all \emph{maximal dyadic cubes}
in $\mathcal{Q}_k$,
namely, there exists no $Q\in\mathcal{Q}_k$
such that $Q_i^k\subsetneqq Q$ for any $i$.

For any $Q\in\mathcal{Q}$, let
\begin{align}\label{3e1}
\widehat{Q}:=\lf\{(y,t)\in\rr^{n+1}_+:\
y\in Q\ \ {\rm and}\
\ t\sim 2^{w\ell(Q)+u}\r\},
\end{align}
here and thereafter, $\rr^{n+1}_+:=\rn\times (0,\fz)$
and $t\sim 2^{w\ell(Q)+u}$ always means
\begin{align}\label{4e8}
2^{w\ell(Q)+u+1}\le t<2^{w[\ell(Q)-1]+u+1},
\end{align}
where $w$ and $u$ are as in Lemma \ref{4l2x}(iv)
and $\ell(Q)$ denotes the level of $Q$.
Clearly, $\{\widehat{Q}\}_{Q\in\mathcal{Q}}$ are mutually disjoint and
\begin{align}\label{4.26x}
\rr^{n+1}_+=\bigcup_{k\in\mathbb{Z}}\bigcup_i B_{k,\,i},
\end{align}
where, for any $k\in\mathbb{Z}$ and $i$,
$B_{k,\,i}:=\bigcup_{Q\subset Q_i^k,Q\in\mathcal{Q}_k}\widehat{Q}$.
Then, by Lemma \ref{4l2x}(ii), we easily know that $\{B_{k,i}\}_{k\in\zz,\,i}$
are mutually disjoint.

Let $\psi$ and $\varphi$ be as in Lemma \ref{x4l1}. Then
$\varphi$ has the vanishing moments up to order $s$ as in \eqref{5eq1}. By
Lemma \ref{x4l1}, the properties of tempered distributions
(see \cite[Theorem 2.3.20]{lg14} or \cite[Theorem 3.13]{sw71}),
we find that, for any $f\in\cs'_0(\rn)$ with
$S_{\psi}(f)\in \lv$, and for any $x\in\rn$,
\begin{align}\label{3e12}
f(x)
=\sum_{k\in\mathbb{Z}}f\ast\psi_k\ast\varphi_k(x)
=\int_{\rr^{n+1}_+}
f\ast\psi_t(y)\varphi_t(x-y)\,dy\,dm(t)
\end{align}
in $\cs'(\rn)$, where $m$ denotes
the \emph{integer counting measure} on $\rn$, namely,
for any set $E\subset \rr$, $m(E)$ is the number of integers
contained in $E$.
For any $k\in\mathbb{Z}$, $i$ and $x\in\rn$, let
\begin{align*}
h_i^k(x)
:=\int_{B_{k,\,i}}f\ast\psi_t(y)\varphi_t(x-y)\,dy\,dm(t).
\end{align*}
Next we show that
\begin{align}\label{3e21}
\sum_{k\in\zz}\sum_i h_i^k
\quad{\rm converges~~in}\quad\cs'(\rn).
\end{align}

To this end, we first claim that

\begin{enumerate}
\item[{\rm (i)}]
for any given $r\in(1,\fz)$,
any $k\in\zz,$ $i$ and $x\in\rn$,
$$h_i^k(x)
=\sum_{Q\subset Q_i^k,Q\in\mathcal{Q}_k}
\int_{\widehat{Q}}f\ast\psi_t(y)\varphi_t(x-y)\,dy\,dm(t)
=:\sum_{Q\subset Q_i^k,Q\in\mathcal{Q}_k}e_{Q}(x)
$$
converges in $L^r(\rn)$ and hence in $\cs'(\rn)$;

\item[{\rm (ii)}] for any $k\in\zz,$ $i$,
$h_i^k$ is a multiple of a $(\vp,r,s)$-atom.

\end{enumerate}

Indeed, using \cite[(3.23)]{lyy16LP}
with the dilation $A$ as in \eqref{5eq2}, we conclude that,
for any $x\in\rn$,
\begin{align}\label{4e11}
\lf[S_{\varphi}\lf(\sum_{Q\subset Q_i^k,Q\in\mathcal{Q}_k}e_Q\r)(x)\r]^2
\ls\sum_{Q\subset Q_i^k,Q\in\mathcal{Q}_k}
\lf[M_{\rm HL}^{\va}(c_Q{\mathbf{1}}_Q)(x)\r]^2,
\end{align}
where, for any $Q\subset Q_i^k$ and $Q\in\mathcal{Q}_k$,
$$c_Q:=\lf[\int_{\widehat{Q}}|\psi_t\ast f(y)|^2
\,dy\frac{dm(t)}{t^{\nu}}\r]^{1/2}.$$

We now show the above two assertions and we first show assertion (i). 
To this end, for any $k\in\zz$, $Q\in\mathcal{Q}_k$ and
$x\in Q$, by Lemma \ref{4l2x}(iv), we find that
$$M_{\rm HL}^{\va}\lf({\mathbf{1}}_{\Omega_k}\r)(x)\gs\frac1{2^{[w\ell(Q)+u]\nu}}
\int_{x_Q+2^{[w\ell(Q)+u]\va}B_0}{\mathbf{1}}_{\Omega_k}(z)\,dz\gs2^{-2u\nu}
\frac{|\Omega_k\cap Q|}{|Q|}\gs 2^{-2u\nu-1},$$
which implies that
\begin{align*}
\bigcup_{Q\subset Q_i^k,Q\in\mathcal{Q}_k}Q
\subset\widehat{\Omega}_k:=\lf\{x\in\rn:\
M_{\rm HL}^{\va}\lf({\mathbf{1}}_{\Omega_k}\r)(x)\gs 2^{-2u\nu-1}\r\}.
\end{align*}
In addition, for any $k\in\zz$, $Q\in\mathcal{Q}_k$ and $x\in Q$,
by $Q\subset\widehat{\Omega}_k$, we know that
$$M_{\rm HL}^{\va}\lf({\mathbf{1}}_{Q\cap(\widehat{\Omega}_k
\setminus\Omega_{k+1})}\r)(x)
\ge\frac1{|Q|}\int_{Q}{\mathbf{1}}_{Q\cap(\widehat{\Omega}_k
\setminus\Omega_{k+1})}(z)\,dz
\ge\frac{|Q|-|Q|/2}{|Q|}\ge\frac{{\mathbf{1}}_Q(x)}2.$$
From this, \cite[Theorem 3.2]{blyz10} with the
dilation $A$ as in \eqref{5eq2}, \eqref{4e11}, \cite[Lemma 3.7]{hlyy}
and an argument similar to that used in
the proof of \cite[(3.26)]{lyy16LP},
it follows that, for any given $r\in(1,\fz)$, any $k\in\zz$ and $i$,
\begin{align}\label{4e14}
\lf\|\sum_{Q\subset Q_i^k,Q\in\mathcal{Q}_k}
e_{Q}\r\|_{L^r(\rn)}
\ls\lf\|\lf[\sum_{Q\subset Q_i^k,Q\in\mathcal{Q}_k}
\lf(c_Q\r)^2{\mathbf{1}}_{Q\cap(\widehat{\Omega}_k
\setminus\Omega_{k+1})}\r]^{1/2}\r\|_{L^r(\rn)}.
\end{align}

On another hand, for any $k\in\zz$, $Q\in\mathcal{Q}_k,$ $x\in Q$
and $(y,t)\in\widehat{Q}$, by \cite[Lemmas 4.9(iv) and 2.5(ii)]{hlyy}
and \eqref{4e8}, we easily know that
$$x-y\in 2^{[w\ell(Q)+u]\va}B_0+2^{[w\ell(Q)+u]\va}B_0
\subset 2^{[w\ell(Q)+u+1]\va}B_0\subset t^{\va}B_0,$$
here and thereafter, $B_0:=B_{\va}(\vec{0}_n,1)$.
From this and the disjointness of
$\{\widehat{Q}\}_{Q\subset Q_i^k}$,
we deduce that, for any  $k\in\zz$, $i$ and $x\in\rn$,
\begin{align}\label{4e15}
\sum_{Q\subset Q_i^k,Q\in\mathcal{Q}_k}
\lf(c_Q\r)^2{\mathbf{1}}_{Q\cap(\widehat{\Omega}_k
\setminus\Omega_{k+1})}(x)
&=\sum_{Q\subset Q_i^k,Q\in\mathcal{Q}_k}
\int_{\widehat{Q}}
|\psi_t\ast f(y)|^2\,dy\frac{dm(t)}{t^{\nu}}
{\mathbf{1}}_{Q\cap(\widehat{\Omega}_k
\setminus\Omega_{k+1})}(x)\noz\\
&\ls\lf[S_{\psi}(f)(x)\r]^2{\mathbf{1}}_{Q_i^k\cap(\widehat
{\Omega}_k\setminus\Omega_{k+1})}(x).
\end{align}
Combining this and Lemma \ref{4l2x}(iv), we further
conclude that, for any given $r\in(1,\fz)$, any $k\in\zz$ and $i$,
\begin{align}\label{4e16}
&\lf\|\lf\{\sum_{Q
\subset Q_i^k,Q\in\mathcal{Q}_k}
\lf(c_Q\r)^2{\mathbf{1}}_{Q\cap(\widehat{\Omega}_k
\setminus\Omega_{k+1})}\r\}^
{1/2}\r\|^r_{L^r(\rn)}\noz\\
&\hs\le\int_{\rn}\lf[S_{\psi}(f)(x)\r]^r
{\mathbf{1}}_{Q_i^k\cap(\widehat{\Omega}_k
\setminus\Omega_{k+1})}(x)\,dx\noz\\
&\hs\ls 2^{kr}\lf|Q_i^k\r|
\ls 2^{kr}\lf|2^{(wk+u)\va}B_0\r|
\ls 2^{kr}2^{(wk+u)\nu}< \fz.
\end{align}
For any $N\in\mathbb{N}$, let
$\mathcal{Q}_{k,N}:=
\{Q\in\mathcal{Q}_k:\ |\ell(Q)|>N\}$.
Then, replacing
$\sum_{Q\subset Q_i^k,Q\in\mathcal{Q}_k}e_{Q}$
by $\sum_{Q\subset Q_i^k,Q\in\mathcal{Q}_{k,N}}e_Q$
in \eqref{4e14}, we obtain, for any $N\in\nn$, $k\in\zz$ and $i$,
\begin{align*}
\lf\|\sum_{Q\subset Q_i^k,Q\in
\mathcal{Q}_{k,N}}e_{Q}\r\|_{L^r(\rn)}^r
&\ls\lf\|\lf[\sum_{Q\subset Q_i^k,Q\in
\mathcal{Q}_{k,N}}\lf(c_Q\r)^2{\mathbf{1}}_{Q\cap
(\widehat{\Omega}_k\setminus\Omega_{k+1})}
\r]^{1/2}\r\|_{L^r(\rn)}^r.
\end{align*}
From this, \eqref{4e16} and the
Lebesgue dominated convergence theorem, we deduce that,
for any given $r\in(1,\fz)$, any $k\in\zz$ and $i$,
$$\lf\|\sum_{Q\subset Q_i^k,Q\in
\mathcal{Q}_{k,N}}e_{Q}\r\|_{L^r(\rn)}\rightarrow0$$
as $N\rightarrow\fz$, and hence
$$\lf\|
\int_{\cup_{Q\subset Q_i^k,Q\in\mathcal{Q}_{k,N}}\widehat{Q}}
f\ast\psi_t(y)\varphi_t(x-y)\,dy\,dm(t)\r\|_{L^r(\rn)}\rightarrow0$$
as $N\rightarrow\fz$. Thus,
$h_i^k=\sum_{Q\subset Q_i^k,Q\in\mathcal{Q}_k}e_{Q}$
in $L^r(\rn)$. This finishes the proof of assertion (i) above.
By this, \eqref{4e14}, the estimation of \eqref{4e16} and
Lemma \ref{4l2x}(iv), we know that
\begin{equation}\label{4e17}
\lf\|h_i^k\r\|_{L^r(\rn)}\ls\lf\{\int_{\rn}\lf[S_{\psi}(f)(x)\r]^r
{\mathbf{1}}_{Q_i^k\cap(\widehat{\Omega}_k
\setminus\Omega_{k+1})}(x)\,dx\r\}^{1/r}\ls2^k\lf|Q_i^k\r|^{1/r}
\le \wz{C_1} 2^k\lf|B_i^k\r|^{1/r},
\end{equation}
where $\wz{C_1}$ is a positive constant
independent of $f$, $k$ and $i$ and, for any $k\in\zz$ and $i$,
$$B_i^k:=x_{Q_i^k}+2^{(w[\ell(Q_i^k)-1]+u+3)\va}B_0.$$

We now show assertion (ii). To this end, observe that,
for any $x\in\supp h_i^k$ with $k\in\zz$, $h_i^k(x)\neq0$
implies that there exists a $Q\subset Q_i^k$ and
$Q\in\mathcal{Q}_k$ such that $e_{Q}(x)\neq0$.
Then there exists $(y,t)\in\widehat{Q}$ such that
$t^{-\va}(x-y)\in B_0$.
By this, Lemma \ref{4l2x}(iv),
\eqref{4e8} and \cite[Lemma 2.5(ii)]{hlyy}, we have
$$x\in y+t^{\va}B_0\subset x_Q+2^{(w\ell(Q)+u)\va}B_0
+2^{(w[\ell(Q)-1]+u+1)\va}B_0\subset x_Q+2^{(w[\ell(Q)-1]+u+2)\va}B_0.$$
Thus,
$$\supp e_Q\subset x_Q+2^{(w[\ell(Q)-1]+u+2)\va}B_0.$$
From this, the fact that
$h_i^k=\sum_{Q\subset Q_i^k,Q\in\mathcal{Q}_k}e_{Q}$,
(ii) and (iv) of Lemma \ref{4l2x} and
\cite[Lemma 2.5(ii)]{hlyy},
we further deduce that, for any $k\in\zz$ and $i$,
\begin{align}\label{4e12}
\supp h_i^k
&\subset\bigcup_{Q\subset Q_i^k,Q\in\mathcal
{Q}_k}\lf(x_Q+2^{(w[\ell(Q)-1]+u+2)\va}B_0\r)\noz\\
&\subset x_{Q_i^k}+2^{w[\ell(Q_i^k)+u]\va}B_0+
2^{(w[\ell(Q_i^k)-1]+u+2)\va}B_0
\subset B_i^k.
\end{align}

Recall that $\varphi$ has the vanishing moments up to
order $s\ge\lfloor\frac{\nu}{a_-}(\frac1{p_-}-1)\rfloor$
and so does $e_Q$. For any
$k\in\mathbb{Z},$ $i$, $\gamma\in\zz_+^n$ with
$|\gamma|\le s$ and $x\in\rn$,
let $g(x):=x^\gamma{\mathbf{1}}_{B_i^k}(x)$.
Clearly, $g\in L^{r'}(\rn)$ with $r\in(1,\fz)$
and $1/r+1/r'=1$.
Thus, by \eqref{4e12} and the facts that
$(L^{r'}(\rn))^\ast=L^r(\rn)$ and
$$\supp e_Q\subset x_Q+
2^{(w[\ell(Q)-1]+u+2)\va}B_0,$$
we conclude that, for any $k$, $i$ and $\gamma$ as above,
\begin{equation*}
\int_{\rn}h_i^k(x)x^\gamma\,dx
=\langle h_i^k,g\rangle
=\sum_{Q\subset Q_i^k,Q\in
\mathcal{Q}_k}\langle e_{Q},g\rangle
=\sum_{Q\subset Q_i^k,Q\in\mathcal{Q}_k}
\int_{\rn}e_{Q}(x)x^\gamma\,dx=0,
\end{equation*}
namely, $h_i^k$ has the vanishing moments up to order $s$,
which, combined with \eqref{4e17} and \eqref{4e12},
implies that $h_i^k$ is a multiple of a $(\vp,r,s)$-atom
supported in $B_i^k$. This finishes the proof of assertion (ii) above.

Now we prove \eqref{3e21}. For any $k\in\zz$ and $i$,
let
\begin{align}\label{3e24}
\lambda_i^k:=\wz{C_1} 2^k\lf\|{\mathbf{1}}_{B_i^k}\r\|_{\lv}\quad
{\rm and}\quad a_i^k:=(\lambda_i^k)^{-1}h_i^k,
\end{align}
 where
$\wz{C_1}$ is as in \eqref{4e17}.
Then it is easy to see that, for any $k\in\zz$ and $i$,
$a_i^k$ is a $(\vp,r,s)$-atom. Notice that $i\in\nn$ or
there exists a $D\in\nn$ such that $i\in\{1,\ldots,D\}$.
When $i\in\nn$, by \eqref{3e24}, to prove \eqref{3e21},
it suffices to show that
\begin{align}\label{3e22}
\lim_{l\to\fz}\lf\|\sum_{l\le|k|\le m}\sum_{l\le i\le m} \lz_i^k a_i^k\r\|_{\vh}=0.
\end{align}
Assuming that \eqref{3e22} holds true for the moment, then,
for any $\phi\in\cs(\rn)$, let $\wz{\phi}(\cdot):=\phi(-\cdot)$.
Obviously, for any $l,$ $m\in\nn$,
$$\lf|\lf\langle\sum_{l\le|k|\le m}
\sum_{l\le i\le m} \lz_i^k a_i^k,\phi\r\rangle\r|
=\lf|\lf(\sum_{l\le|k|\le m}\sum_{l\le i\le m}
\lz_i^k a_i^k\r)\ast\wz{\phi}(\vec 0_n)\r|.$$
Combining this, the proof of \cite[Lemma 4.8]{hlyy}
with $f,$ $\phi$ and $k$ therein
replaced, respectively, by $\sum_{l\le|k|\le m}
\sum_{l\le i\le m} \lz_i^k a_i^k,$ $\wz{\phi}$ and $0$,
and \eqref{3e22}, we further conclude that
\begin{align*}
\lim_{l\to\fz}\lf|\lf\langle \sum_{l\le|k|\le m}\sum_{l\le i\le m}
\lz_i^k a_i^k,\phi\r\rangle\r|
&\ls \lim_{l\to\fz}\lf\|M_N\lf(\sum_{l\le|k|\le m}\sum_{l\le i\le m}
\lz_i^k a_i^k\r)\r\|_{\lv}\\
&\sim \lim_{l\to\fz}\lf\|\sum_{l\le|k|\le m}\sum_{l\le i\le m}
\lz_i^k a_i^k\r\|_{\vh}\to 0,
\end{align*}
where $N$ is as in \eqref{2e11}. From this and the completeness
of $\cs'(\rn)$, we deduce that \eqref{3e21} holds true.
Therefore, to show \eqref{3e21}, it remains to prove \eqref{3e22}.
To do this, for any $k\in\mathbb{Z}$ and $i\in\nn$, by the fact that
$|\Qik\cap\Omega_k|\ge\frac{|\Qik|}{2}$, we find that, for any $x\in\rn$,
\begin{align*}
M_{\rm HL}^{\va}\lf({\mathbf{1}}_{\Qik\cap\Omega_k}\r)(x)
\gs \frac1{|\Qik|}\int_{\Qik}{\mathbf{1}}_{\Qik\cap\Omega_k}(y)\,dy
\sim\frac{|{\Qik\cap\Omega_k}|}{|\Qik|}\gs \frac1{2},
\end{align*}
which, together with \cite[(2.5) and Lemma 3.5]{hlyy}, implies that
\begin{align}\label{3e20}
\lf\|{\mathbf{1}}_{\Qik}\r\|_{\lv}
&=\lf\|{\mathbf{1}}_\Qik\r\|_{L^{2\vp/p_-}(\rn)}^{2/p_-}
\ls \lf\|M_{\rm HL}^{\va}\lf({\mathbf{1}}_{\Qik\cap\Omega_k}\r)
\r\|_{L^{2\vp/p_-}(\rn)}^{2/p_-}\noz\\
&\ls \lf\|{\mathbf{1}}_{\Qik\cap\Omega_k}\r\|_{L^{2\vp/p_-}(\rn)}^{2/p_-}
\sim\lf\|{\mathbf{1}}_{\Qik\cap\Omega_k}\r\|_{\lv}.
\end{align}
From the fact that, for any $l,$ $m\in\nn$,
$\sum_{l\le|k|\le m}\sum_{l\le i\le m}
\lz_i^k a_i^k\in\vah$,
Theorem \ref{5t1},
the mutual disjointness of $\{Q_i^k\}_{k\in\mathbb{Z},\,i\in\nn}$ and
\cite[Lemma 5.9(iv)]{hlyy}, it follows that
\begin{align*}
&\lf\|\sum_{l\le|k|\le m}\sum_{l\le i\le m}
\lz_i^k a_i^k\r\|_{\vh}\\
&\hs\ls\lf\|\lf\{\sum_{l\le|k|\le m}\sum_{l\le i\le m}
\lf[\frac{\lz_i^k{\mathbf{1}}_{\Bik}}{\|{\mathbf{1}}_{\Bik}\|_{\lv}}\r]^
{\underline{p}}\r\}^{1/\underline{p}}\r\|_{\lv}
\sim\lf\|\lf[\sum_{l\le|k|\le m}\sum_{l\le i\le m}\lf(2^{k}
{\mathbf{1}}_{\Bik}\r)^{{\underline{p}}}\r]^{{1/\underline{p}}}\r\|_{\lv}\\
&\hs\sim\lf\|\lf[\sum_{l\le|k|\le m}\sum_{l\le i\le m}\lf(2^{k}
{\mathbf{1}}_{\Qik}\r)^{\underline{p}}\r]^{{1/\underline{p}}}\r\|_{\lv}
\sim\lf\|\sum_{l\le|k|\le m}\sum_{l\le i\le m}
\lf(2^{k}{\mathbf{1}}_{\Qik}\r)^{\underline{p}}\r\|
_{L^{\vp/\underline{p}}(\rn)}^{{1/\underline{p}}},
\end{align*}
which, combined with \eqref{3e20}, implies that
\begin{align*}
&\lf\|\sum_{l\le|k|\le m}\sum_{l\le i\le m}
\lz_i^k a_i^k\r\|_{\vh}\\
&\hs\sim\lf\|\sum_{l\le|k|\le m}\sum_{l\le i\le m}
\lf(2^{k}{\mathbf{1}}_{\Qik\cap\Omega_k}\r)^{\underline{p}}\r\|
_{L^{\vp/\underline{p}}(\rn)}^{{1/\underline{p}}}
\ls\lf\|\lf[\sum_{l\le|k|\le m}\lf(2^k{\mathbf{1}}_{\Omega_k}\r)^
{\underline{p}}\r]^{1/{\underline{p}}}\r\|_{\lv}\\
&\hs\sim \lf\|\lf[\sum_{l\le|k|\le m}\lf(2^k{\mathbf{1}}
_{\Omega_k\setminus\Omega_{k+1}}\r)^
{\underline{p}}\r]^{1/{\underline{p}}}\r\|_{\lv}
\sim \lf\|S_{\psi}(f)\lf(\sum_{l\le|k|\le m}{\mathbf{1}}
_{\Omega_k\setminus\Omega_{k+1}}\r)
^{1/{\underline{p}}}\r\|_{\lv}\to 0
\end{align*}
as $l\to \fz$. This further implies that \eqref{3e22} holds true
and so does \eqref{3e21} in the case $i\in\nn$.
When $i\in\{1,\ldots,D\}$, to prove \eqref{3e21}, it suffices to show that
\begin{align}\label{3e23}
\lim_{l\to\fz}\lf\|\sum_{l\le|k|\le m}\sum_{i=1}^{D} \lz_i^k a_i^k\r\|_{\vh}=0.
\end{align}
Applying a similar argument to that used in the proof of \eqref{3e22} above,
we know that \eqref{3e23} also holds true. Thus,
\begin{align*}
\sum_{k\in\zz}\sum_i h_i^k=\sum_{k\in\zz}\sum_i \lz_i^k a_i^k\quad
{\rm converges~~~in}\quad\cs'(\rn).
\end{align*}

Now, for any $x\in\rn$, let
\begin{align*}
\eta(x)
:=\sum_{k\in\zz}\sum_i h_i^k(x)
=\sum_{k\in\mathbb{Z}}\sum_i\int_{B_{k,i}}
f\ast\psi_t(y)\varphi_t(x-y)\,dy\,dm(t)
\quad{\rm in}\quad \cs'(\rn),
\end{align*}
here and thereafter, for any $k\in\zz$ and $i$,
$B_{k,i}$ is as in \eqref{4.26x}.
To finish  the proof of the sufficiency of Theorem \ref{5t3},
we next show that
\begin{align}\label{3e28}
f=\eta\quad{\rm in}\quad \cs'(\rn).
\end{align}
Indeed, by assertion (i) above, \eqref{3e1} and \eqref{4e8}, we know that,
for any given $r\in(1,\fz)$, any $k\in\zz,\ i$ and $x\in\rn$,
\begin{align}\label{3.1}
h_i^k(x)
&=\lim_{N\to\fz}\int_0^\fz\int_{\rn}
f\ast\psi_t(y)\varphi_t(x-y)\mathbf{1}_{\cup_{\gfz{Q\subset Q_i^k,Q\in\mathcal{Q}_k}{|\ell(Q)|\le N}}\widehat{Q}}(y,t)\,dy\,dm(t)\noz\\
&=\lim_{N\to\fz}\int_{\az(N)}^{\bz(N)}\int_{\rn}
f\ast\psi_t(y)\varphi_t(x-y)\mathbf{1}_{B_{k,i}}(y,t)\,dy\,dm(t)
\end{align}
converges in $L^r(\rn)$ and hence in $\cs'(\rn)$,
where, for any $N\in\nn$, $\az(N):=2^{w N+u+1}$ and
$\bz(N):=2^{-w(N+1)+u+1}$ with $w$ and $u$ as in Lemma \ref{4l2x}(iv).
For the convenience of the notation, we rewrite $\eta$ as,
for any $x\in\rn$,
\begin{align*}
\eta(x)=\sum_{j\in\nn}\int_{R^{(j)}}
f\ast\psi_t(y)\varphi_t(x-y)\,dy\,dm(t),
\end{align*}
where $\{R^{(j)}\}_{j\in\nn}$ is any rearrangement
of $\{B_{k,\,i}\}_{k\in\zz,i}$. For any $M\in\nn$
and $x\in\rn$, let
\begin{align*}
\eta_M(x):=f(x)-\sum_{j=1}^M\int_{R^{(j)}}
f\ast\psi_t(y)\varphi_t(x-y)\,dy\,dm(t).
\end{align*}
Then, using \eqref{4.26x}, \eqref{3e12} and \eqref{3.1}, we have,
for any $M\in\nn$ and $x\in\rn$,
\begin{align}\label{3e29}
\eta_M(x)&=\lim_{N\to\fz}\int_{\az(N)}^{\bz(N)}\int_{\rn}
f\ast\psi_t(y)\varphi_t(x-y)\mathbf{1}_{\cup_{j=1}^\fz R^{(j)}}(y,t)\,dy\,dm(t)\noz\\
&\hs\hs-\lim_{N\to\fz}\int_{\az(N)}^{\bz(N)}\int_{\rn}
f\ast\psi_t(y)\varphi_t(x-y)\mathbf{1}_{\cup_{j=1}^M R^{(j)}}(y,t)\,dy\,dm(t)\noz\\
&=\lim_{N\to\fz}\int_{\az(N)}^{\bz(N)}\int_{\rn}
f\ast\psi_t(y)\varphi_t(x-y)\mathbf{1}_{\cup_{j=M+1}^\fz R^{(j)}}(y,t)\,dy\,dm(t)
\end{align}
converges in $\cs'(\rn)$.

Note that $\vh$ is continuously embedded into $\cs'(\rn)$
(see \cite[Lemma 3.6]{hlyy19}). Thus, to prove \eqref{3e28},
it suffices to show that $\|\eta_M\|_{\vh}\to 0$ as $M\to \fz$.
To do this, we borrow some ideas from the proof of the atomic
characterizations of $\vh$ (see \cite[Theorem 3.16]{hlyy}).
Indeed, for any $\vaz\in(0,\fz)$, $M\in\nn$ and $x\in\rn$, let
\begin{align*}
\eta_M^{\vaz,\kappa}(x):=\int_{\vaz}^{\kappa/\vaz}\int_{\rn}
f\ast\psi_t(y)\varphi_t(x-y)
\mathbf{1}_{\cup_{j=M+1}^\fz R^{(j)}}(y,t)\,dy\,dm(t),
\end{align*}
where $\kappa$ is any given positive constant.
Therefore, by the Lebesgue dominated convergence theorem,
for any given $\kappa\in(0,\fz)$, any $\vaz\in(0,\fz)$, $j\in\nn$ and $x\in\rn$,
we have
\begin{align*}
\eta_M^{\vaz,\kappa}(x)
=\sum_{j=M+1}^\fz \int_{\vaz}^{\kappa/\vaz}\int_{\rn}
f\ast\psi_t(y)\varphi_t(x-y)\mathbf{1}_{R^{(j)}}(y,t)\,dy\,dm(t)
=:\sum_{j=M+1}^\fz h_j^{\vaz,\kappa}(x)
\end{align*}
in $\cs'(\rn)$.

In addition, for any given $\kappa\in(0,\fz)$, any
$\vaz\in(0,\fz)$ and $Q\in\mathcal{Q}$, let
$$\widehat{Q}_{\vaz,\kappa}:=\lf\{(y,t)\in\rn\times(\vaz,\kappa/\vaz):\
y\in Q\ \ {\rm and}\
\ t\sim 2^{w\ell(Q)+u}\r\},$$
where $w$ and $u$ are as in Lemma \ref{4l2x}(iv)
and $\ell(Q)$ denotes the level of $Q$.
Obviously, for any given $\kappa\in(0,\fz)$ and any $\vaz\in(0,\fz)$,
$\{\widehat{Q}_{\vaz,\kappa}\}_{Q\in\mathcal{Q}}$ are mutually disjoint and
\begin{align*}
\rn\times(\vaz,\kappa/\vaz)=\bigcup_{k\in\mathbb{Z}}\bigcup_i B_{k,\,i}^{\vaz,\kappa},
\end{align*}
where, for any $k\in\zz$ and $i$,
$B_{k,\,i}^{\vaz,\kappa}
:=\bigcup_{Q\subset Q_i^k,Q\in\mathcal{Q}_k}\widehat{Q}_{\vaz,\kappa}$.
Then, by Lemma \ref{4l2x}(ii),
we easily know that, for any given $\kappa\in(0,\fz)$ and any $\vaz\in(0,\fz)$,
$\{B_{k,i}^{\vaz,\kappa}\}_{k\in\zz,\,i}$ are mutually disjoint.

Now we claim that, for any given $\kappa\in(0,\fz)$, any $\vaz\in(0,\fz)$ and $x\in\rn$,
\begin{align}\label{4e11'}
\lf[S_{\varphi}\lf(\sum_{Q\subset Q_i^k,Q\in\mathcal{Q}_k}e_Q^{\vaz,\kappa}\r)(x)\r]^2
\ls\sum_{Q\subset Q_i^k,Q\in\mathcal{Q}_k}
\lf[M_{\rm HL}^{\va}(c_Q^{\vaz,\kappa}{\mathbf{1}}_Q)(x)\r]^2,
\end{align}
where, for any $Q\subset Q_i^k$, $Q\in\mathcal{Q}_k$
and $x\in\rn$,
$$e_Q^{\vaz,\kappa}(x)
:=\int_{\widehat{Q}_{\vaz,\kappa}}f\ast\psi_t(y)\varphi_t(x-y)\,dy\,dm(t)$$
and
$$c_Q^{\vaz,\kappa}:=\lf[\int_{\widehat{Q}_{\vaz,\kappa}}|\psi_t\ast f(y)|^2
\,dy\frac{dm(t)}{t^{\nu}}\r]^{1/2}.$$
Assuming this inequality holds true for the moment, then,
for any given $\kappa\in(0,\fz)$ and any $\vaz\in(0,\fz)$,
by some arguments similar to these used in the proofs of
the above assertions (i) and (ii)  with $\widehat{Q}$ and \eqref{4e11} therein
replaced, respectively, by $\widehat{Q}_{\vaz,\kappa}$ and \eqref{4e11'},
we conclude that, for any $j\in\nn$ and $x\in\rn$,
\begin{align*}
h_j^{\vaz,\kappa}(x)
=\sum_{Q\subset Q_i^k,Q\in\mathcal{Q}_k} e_Q^{\vaz,\kappa}(x)
\end{align*}
converges in $\cs'(\rn)$ and, for any given $r\in(1,\fz)$,
$h_j^{\vaz,\kappa}$ is a multiple of a $(\vp,r,s)$-atom, namely,
there exist $\{\lz_j\}_{j\in\nn}\subset\mathbb{C}$
and a sequence of $(\vp,r,s)$-atoms, $\{a_j^{\vaz,\kappa}\}_{j\in\nn}$,
supported, respectively, in $\{B_j\}_{j\in\nn}\subset\mathfrak{B}$
such that, for any $j\in\nn$, $h_j^{\vaz,\kappa}=\lz_j a_j^{\vaz,\kappa}$,
where, for any $j\in\nn$, $\lz_j$ and $B_j$ are
independent of $\vaz$ and $\kappa$.
Thus, for any given $\kappa\in(0,\fz)$,
any $\vaz\in(0,\fz)$, $M\in\nn$ and $x\in\rn$,
\begin{align}\label{x3e16}
\eta_M^{\vaz,\kappa}(x)
=\sum_{j=M+1}^\fz \lz_j a_j^{\vaz,\kappa}(x)
\quad{\rm in}\quad \cs'(\rn)
\end{align}
and
\begin{align}\label{x3e18}
\lf\|\lf\{\sum_{j=M+1}^{\fz}
\lf[\frac{|\lz_j|\mathbf{1}_{B_j}}{\|\mathbf{1}_{B_j}
\|_{\lv}}\r]^{\underline{p}}\r\}^{1/\underline{p}}\r\|_{\lv}
< \fz.
\end{align}

For any $f\in\cs'(\rn)$, let $M_0(f)$ be as in \eqref{3e16}.
Then, by the fact that, for any given $\kappa\in(0,\fz)$ and any $\vaz\in(0,\fz)$,
$\{a_j^{\vaz,\kappa}\}_{j\in\nn}$ is a sequence of $(\vp,r,s)$-atoms
and \cite[(3.41)]{hlyy},
we know that, for any $j\in\nn$ and $x\in\rn$,
\begin{align}\label{x3e15}
M_0(a_j^{\vaz,\kappa})(x)\ls M_{\rm HL}^{\va}(a_j^{\vaz,\kappa})(x)\mathbf{1}_{B_j^{(2)}}(x)
+\frac1{\|\mathbf{1}_{B_j}\|_{\lv}}
\lf[M_{\rm HL}^{\va}(\mathbf{1}_{B_j})(x)\r]^
{\frac{\nu+(s+1)a_-}{\nu}},
\end{align}
where, for any $j\in\nn$, $B_j^{(2)}$ is as in \eqref{2e2'} with $\delta:=2$.
Let $r\in (\max\{p_+,1\},\fz)$. Then, by \cite[(3.38)]{hlyy} and Lemma \ref{3l1},
we find that, for any given $\kappa\in(0,\fz)$ and any $\vaz\in(0,\fz)$,
$$\lf\|M_0(a_j^{\vaz,\kappa})\mathbf{1}_{B_j^{(2)}}\r\|_{L^r(\rn)}
\ls\lf\|M_{\rm HL}^{\va}(a_j^{\vaz,\kappa})
\mathbf{1}_{B_j^{(2)}}\r\|_{L^r(\rn)}
\ls\frac{|B_j|^{1/r}}{\|\mathbf{1}_{B_j}\|_{\lv}},$$
which, combined with Lemma \ref{3l9}, further implies that
\begin{align}\label{x3e17}
&\lf\|\liminf_{\vaz\to0^+} \lf\{\sum_{j=M+1}^{\fz}\lf[\lf|\lz_j\r|M_0(a_j^{\vaz,\kappa})
\mathbf{1}_{B_j^{(2)}}\r]^{{\underline{p}}}\r\}^{1/{\underline{p}}}\r\|_{\lv}\noz\\
&\hs\hs\ls \lf\|\lf\{\sum_{j=M+1}^{\fz}
\lf[\frac{|\lz_i|\mathbf{1}_{B_j}}{\|\mathbf{1}_{B_j}\|_{\lv}}\r]^
{{\underline{p}}}\r\}^{1/{\underline{p}}}\r\|_{\lv}.
\end{align}
Let $\vaz:=\az(N)$ and $$\kappa:=2^{-w+2(u+1)}$$ with $N\in\nn$,
$w$ and $u$ as in Lemma \ref{4l2x}(iv). Then, by \eqref{3e29},
we know that
\begin{align*}
M_0(\eta_M)&=M_0\lf(\lim_{N\to\fz}\eta_M^{\az(N),\kappa}\r)
=\sup_{t\in(0,\fz)}\lf|\lim_{N\to\fz}\Phi_t\ast\eta_M^{\az(N),\kappa}\r|
\le \liminf_{N\to\fz}\sup_{t\in(0,\fz)}\lf|\Phi_t\ast\eta_M^{\az(N),\kappa}\r|\\
&=\liminf_{N\to\fz}M_0(\eta_M^{\az(N),\kappa}),
\end{align*}
where $\Phi$ is as in \eqref{3e16}.
From this, Theorem \ref{3t1}, \eqref{x3e16} and \eqref{x3e15},
we deduce that
\begin{align*}
\lf\|\eta_M\r\|_{\vh}&
\le\lf\|\liminf_{N\to\fz}M_0\lf(\eta_M^{\az(N),\kappa}\r)\r\|_{\lv}
\le\lf\|\liminf_{N\to\fz} \sum_{j=M+1}^{\fz}|\lz_j| M_0(a_j^{\az(N),\kappa})\r\|_{\lv}\\
&\ls \lf\|\liminf_{N\to\fz} \sum_{j=M+1}^{\fz}
|\lz_j| M_0(a_j^{\az(N),\kappa})\mathbf{1}_{B_j^{(2)}}\r\|_{\lv}\\
&\quad+\lf\|\sum_{j=M+1}^{\fz}\frac{|\lz_j|}{\|\mathbf{1}_{B_j^{(2)}}\|_{\lv}}
\lf[M_{\rm HL}^{\va}(\mathbf{1}_{B_j})\r]
^{\frac{\nu+(s+1)a_-}{\nu}}\r\|_{\lv},
\end{align*}
which, together with \eqref{x3e17}, Lemma \ref{4l2}
and \eqref{x3e18}, further implies that
\begin{align*}
\lf\|\eta_M\r\|_{\vh}
&\ls\lf\|\liminf_{N\to\fz}\lf\{\sum_{j=M+1}^{\fz}\lf[\lf|\lz_j\r|
M_{\rm HL}^{\va}(a_j^{\az(N),\kappa})
\mathbf{1}_{B_j^{(2)}}\r]^{{\underline{p}}}\r\}^{1/{\underline{p}}}\r\|_{\lv}\noz\\
&\quad+\lf\|\sum_{j=M+1}^{\fz}\lf\{\frac{|\lz_j|}{\|\mathbf{1}_{B_j}\|_{\lv}}
\lf[M_{\rm HL}(\mathbf{1}_{B_j})\r]^
{\frac{\nu+(s+1)a_-}{\nu}}\r\}^{\frac{\nu}{\nu+(s+1)a_-}}\r\|_{L^
{\frac{\nu+(s+1)a_-}{\nu}\vp}(\rn)}^
{\frac{\nu+(s+1)a_-}{\nu}}\\
&\ls\lf\|\lf\{\sum_{j=M+1}^{\fz}
\lf[\frac{|\lz_j|\mathbf{1}_{B_j}}{\|\mathbf{1}_{B_j}
\|_{\lv}}\r]^{\underline{p}}\r\}^{1/\underline{p}}\r\|_{\lv}\to 0
\end{align*}
as $M\to \fz$. This implies that \eqref{3e28} holds true. Therefore,
\begin{align*}
f=\sum_{k\in\zz}\sum_i \lz_i^k a_i^k\quad
{\rm in}\quad\cs'(\rn).
\end{align*}
By this, Theorem \ref{5t1},
the mutual disjointness of $\{Q_i^k\}_{k\in\mathbb{Z},\,i}$
and \cite[Lemma 5.9(iv)]{hlyy}, we conclude that
\begin{align*}
\|f\|_{\vh}
&\ls\lf\|\lf\{\sum_{k\in\zz}\sum_{i}
\lf[\frac{\lz_i^k\mathbf{1}_{\Bik}}{\|\mathbf{1}_{\Bik}\|_{\lv}}\r]^
{\underline{p}}\r\}^{1/\underline{p}}\r\|_{\lv}
\sim\lf\|\lf[\sum_{k\in\zz}\sum_{i}\lf(2^{k}
\mathbf{1}_{\Bik}\r)^{{\underline{p}}}\r]^{{1/\underline{p}}}\r\|_{\lv}\\
&\sim\lf\|\lf[\sum_{k\in\zz}\sum_{i}\lf(2^{k}
\mathbf{1}_{\Qik}\r)^{\underline{p}}\r]^{{1/\underline{p}}}\r\|_{\lv},
\end{align*}
which, combined with \eqref{3e20}, further implies that
\begin{align*}
\|f\|_{\vh}
&\sim\lf\|\sum_{k\in\zz}\sum_{i}
\lf(2^{k}\mathbf{1}_{\Qik}\r)^{\underline{p}}\r\|
_{L^{\vp/\underline{p}}(\rn)}^{{1/\underline{p}}}
\sim\lf\|\sum_{k\in\zz}\sum_{i}
\lf(2^{k}\mathbf{1}_{\Qik\cap\Omega_k}\r)^{\underline{p}}\r\|
_{L^{\vp/\underline{p}}(\rn)}^{{1/\underline{p}}}\\
&\ls\lf\|\lf[\sum_{k\in\zz}\lf(2^k\mathbf{1}_{\Omega_k}\r)^
{\underline{p}}\r]^{1/{\underline{p}}}\r\|_{\lv}
\sim\lf\|\lf[\sum_{k\in\zz}
\lf(2^k\mathbf{1}_{\Omega_k\setminus\Omega_{k+1}}\r)^
{\underline{p}}\r]^{1/{\underline{p}}}\r\|_{\lv}\\
&\sim\lf\|S_{\psi}(f)\lf[\sum_{k\in\zz}
\mathbf{1}_{\Omega_k\setminus\Omega_{k+1}}
\r]^{1/{\underline{p}}}\r\|_{\lv}
\sim\lf\|S_{\psi}(f)\r\|_{\lv}.
\end{align*}
Thus, $f\in\vh$ and hence \eqref{4e7} holds true, which then
completes the proof of the sufficiency of Theorem \ref{5t3}.

To finish the whole proof, it remains to prove \eqref{4e11'}.
Indeed, for any given $\kappa\in(0,\fz)$, any $\vaz\in(0,\fz)$ and $x\in\rn$,
we have
\begin{align*}
\lf[S_{\varphi}\lf(\sum_{Q\subset Q_i^k,Q\in\mathcal{Q}_k}e_Q^{\vaz,\kappa}\r)(x)\r]^2
&=\int_{\Gamma(x)}\lf|\varphi_t\ast\lf(\sum_{Q\subset
Q_i^k,Q\in\mathcal{Q}_k}e_Q^{\vaz,\kappa}\r)(y)\r|^2\,\frac{dy\,dm(t)}{t^\nu}\\
&\le \sum_{P\in\mathcal{Q},\widehat{P}\cap\Gamma(x)\neq\emptyset}
\int_{\widehat{P}} \lf[\sum_{Q\subset Q_i^k,Q\in\mathcal{Q}_k}
\lf|\varphi_t\ast e_Q^{\vaz,\kappa}(y)\r|\r]^2\,\frac{dy\,dm(t)}{t^\nu},
\end{align*}
where, for any $x\in\rn$, $$\Gamma(x):=\{(y,t)\in\rr_+^{n+1}:\ |y-x|_{\va}<t\}$$
and, for any $P\in\mathcal{Q}$, $\widehat{P}$ is as in \eqref{3e1}.

Let $x\in\supp e_Q^{\vaz,\kappa}$. Note that $\supp \varphi\subset B_0$.
Then there exists $(y,t)\in\widehat{Q}_{\vaz,\kappa}$ such that
$t^{-\va}(x-y)\in B_0$. From this, Lemma \ref{4l2x}(iv),
\eqref{4e8} and \cite[Lemma 2.5(ii)]{hlyy}, it follows that
$$x\in y+t^{\va}B_0\subset x_Q+2^{(w\ell(Q)+u)\va}B_0
+2^{(w[\ell(Q)-1]+u+1)\va}B_0\subset x_Q+2^{(w[\ell(Q)-1]+u+2)\va}B_0.$$
Thus,
\begin{align*}
\supp e_Q^{\vaz,\kappa}\subset
x_Q+2^{(w[\ell(Q)-1]+u+2)\va}B_0=:R_Q.
\end{align*}
Moreover, by the fact that $\supp \varphi\subset B_0$,
Lemma \ref{4l2x}(iv), \eqref{4e8} and \cite[Lemma 2.5(ii)]{hlyy},
we know that, for any $(y,t)\in\widehat{P}$,
\begin{align*}
\supp \varphi_t(y-\cdot)&\subset y+t^{\va}B_0\subset x_P+2^{(w\ell(P)+u)\va}B_0
+2^{(w[\ell(P)-1]+u+1)\va}B_0\\
&\subset x_P+2^{(w[\ell(P)-1]+u+2)\va}B_0=:R_P.
\end{align*}
For any $x\in\rn$ and $(y,t)\in\widehat{P}$ with
$\widehat{P}\cap\Gamma(x)\neq\emptyset$, we will show
below that, for any given $\kappa\in(0,\fz)$ and any $\vaz\in(0,\fz)$, when $R_Q\cap R_P=\emptyset$,
$\varphi_t\ast e_Q^{\vaz,\kappa}(y)\equiv0$, otherwise,
\begin{align}\label{x3e20}
\lf|\varphi_t\ast e_Q^{\vaz,\kappa}(y)\r|
\ls c_Q^{\vaz,\kappa}M_{\rm HL}^{\va}(\mathbf{1}_Q)(x)
2^{(s+1)w|\ell(P)-\ell(Q)|a_-}.
\end{align}
Assuming that this holds true for the moment, then,
via the Cauchy--Schwarz inequality, we have
\begin{align*}
\lf[\sum_{Q\subset Q_i^k,Q\in\mathcal{Q}_k}
\lf|\varphi_t\ast e_Q^{\vaz,\kappa}(y)\r|\r]^2
&\ls \lf[\sum_{\gfz{Q\subset Q_i^k,Q\in\mathcal{Q}_k}{R_Q\cap R_P\neq\emptyset}}
c_Q^{\vaz,\kappa}M_{\rm HL}^{\va}(\mathbf{1}_Q)(x)
2^{(s+1)w|\ell(P)-\ell(Q)|a_-}\r]^2\\
&\ls \sum_{\gfz{Q\subset Q_i^k,Q\in\mathcal{Q}_k}{R_Q\cap R_P\neq\emptyset}}
(c_Q^{\vaz,\kappa})^2\lf[M_{\rm HL}^{\va}(\mathbf{1}_Q)(x)\r]^2
2^{(s+1)w|\ell(P)-\ell(Q)|a_-}\\
&\hs\hs \times\sum_{\gfz{Q\subset Q_i^k,Q\in\mathcal{Q}_k}
{R_Q\cap R_P\neq\emptyset}}
2^{(s+1)w|\ell(P)-\ell(Q)|a_-}.
\end{align*}
In addition, from \cite[(4.18)]{blyz10} with $A$ as in
\eqref{5eq2}, we deduce that, for any $P\in\mathcal{Q}$,
$$\sum_{\gfz{Q\subset Q_i^k,Q\in\mathcal{Q}_k}{R_Q\cap R_P\neq\emptyset}}
2^{(s+1)w|\ell(P)-\ell(Q)|a_-}\ls 1$$
and, for any $Q\subset Q_i^k$ and $Q\in\mathcal{Q}_k$,
$$\sum_{P\in\mathcal{Q},R_Q\cap R_P\neq\emptyset}
2^{(s+1)w|\ell(P)-\ell(Q)|a_-}\ls 1.$$
Therefore, for any given $\kappa\in(0,\fz)$, any $\vaz\in(0,\fz)$ and $x\in\rn$, we have
\begin{align*}
&\lf[S_{\varphi}\lf(\sum_{Q\subset Q_i^k,Q\in\mathcal{Q}_k}e_Q^{\vaz,\kappa}\r)(x)\r]^2\\
&\hs\hs\ls \sum_{\gfz{P\in\mathcal{Q}}{\widehat{P}\cap\Gamma(x)\neq\emptyset}}
\int_{\widehat{P}} \sum_{\gfz{Q\subset Q_i^k,Q\in\mathcal{Q}_k}
{R_Q\cap R_P\neq\emptyset}}
(c_Q^{\vaz,\kappa})^2\lf[M_{\rm HL}^{\va}(\mathbf{1}_Q)(x)\r]^2
2^{(s+1)w|\ell(P)-\ell(Q)|a_-}\,\frac{dy\,dm(t)}{t^\nu}\\
&\hs\hs\ls\sum_{Q\subset Q_i^k,Q\in\mathcal{Q}_k}
(c_Q^{\vaz,\kappa})^2\lf[M_{\rm HL}^{\va}(\mathbf{1}_Q)(x)\r]^2
\lf[\sum_{P\in\mathcal{Q},R_Q\cap R_P\neq\emptyset}2^{(s+1)w|\ell(P)-\ell(Q)|a_-}\r]\\
&\hs\hs\ls \sum_{Q\subset Q_i^k,Q\in\mathcal{Q}_k}
\lf[M_{\rm HL}^{\va}(c_Q^{\vaz,\kappa}\mathbf{1}_Q)(x)\r]^2,
\end{align*}
which implies \eqref{4e11'}. 

To finish the whole proof,
we still need to show \eqref{x3e20}. To this end, observe that,
when $R_Q\cap R_P=\emptyset$, from Lemma \ref{4l2x}(iv),
it follows that, for any $(y,t)\in\widehat{P}$,
$$y\in P\subset x_P+2^{[w\ell(P)+u]\va}B_0\subset R_P,$$
which implies that, for any $z\in R_Q$,
$$|y-z|_{\va}\ge2^{w[\ell(P)-1]+u+2}>t.$$
Thus, for any given $\kappa\in(0,\fz)$, any $\vaz\in(0,\fz)$ and $(y,t)\in\widehat{P}$,
$$\varphi_t\ast e_Q^{\vaz,\kappa}(y)
=\int_{R_Q}\varphi_t(y-z)e_Q^{\vaz,\kappa}(z)\,dz=0.$$

Now we deal with the non-trivial case $R_Q\cap R_P\neq\emptyset$ 
by considering the following two subcases.

\emph{Case I)} $\ell(P)\le \ell(Q)$. In this case,
for any $(y,t)\in\widehat{P}$ and $z\in R_Q$, we have
\begin{align*}
z':=t^{-\va}z\in t^{-\va}x_Q+t^{-\va} 2^{(w[\ell(Q)-1]+u+2)\va}B_0
\subset t^{-\va}x_Q+2^{(w[\ell(Q)-\ell(P)-1]+1)\va}B_0=:R_{QP}.
\end{align*}
By this, Lemma \ref{3l8} and the facts
that $-w>0$ and $w[\ell(Q)-\ell(P)]\le 0$,
we find that, for any $z'\in R_{QP}$,
\begin{align}\label{x3e21}
\lf|z'-t^{-\va}x_Q\r|
&=\lf|2^{(-w+1)\va}\lf[2^{(w-1)\va}\lf(z'-t^{-\va}x_Q\r)\r]\r|\noz\\
&\le 2^{(-w+1)a_+}\lf|2^{(w-1)\va}\lf(z'-t^{-\va}x_Q\r)\r|
\ls 2^{[\ell(Q)-\ell(P)]wa_-}.
\end{align}

On another hand, from the H\"{o}lder inequality, it follows that,
for any given $\kappa\in(0,\fz)$, any $\vaz\in(0,\fz)$ and $x\in\rn$,
\begin{align*}
\lf|e_Q^{\vaz,\kappa}(x)\r|^2
\le (c_Q^{\vaz,\kappa})^2\int_{\widehat{Q}_{\vaz,\kappa}}
|\varphi_{\tau}(x-y)|^2\tau^{\nu}\,dy\,dm(\tau)
\ls (c_Q^{\vaz,\kappa})^2 \sum_{\tau\sim 2^{w\ell(Q)+u}}|Q|\tau^{-\nu}
\ls (c_Q^{\vaz,\kappa})^2,
\end{align*}
which, together with the vanishing moments of $e_Q^{\vaz,\kappa}$
and \eqref{x3e21}, further implies that,
for any given $\kappa\in(0,\fz)$, any $\vaz\in(0,\fz)$ and $(y,t)\in\widehat{P}$,
\begin{align}\label{x3e22}
\lf|e_Q^{\vaz,\kappa}\ast\varphi_t(y)\r|
&=\lf|\int_{\rn}e_Q^{\vaz,\kappa}(z)\varphi_t(y-z)\,dz\r|
=\lf|\int_{\rn}e_Q^{\vaz,\kappa}(t^{\va}z)\varphi(t^{-\va}y-z)\,dz\r|\noz\\
&=\lf|\int_{R_{QP}}e_Q^{\vaz,\kappa}(t^{\va}z)\lf[\varphi(t^{-\va}y-z)
-\sum_{|\az|\le s}\frac{\partial^{\az}\varphi(t^{-\va}y-t^{-\va}x_Q)}{\az!}
(t^{-\va}x_Q-z)^{\az}\r]\,dz\r|\noz\\
&\ls \int_{R_{QP}}\lf|e_Q^{\vaz,\kappa}(t^{\va}z)\r| \lf|t^{-\va}x_Q-z\r|^{s+1}\,dz
\ls c_Q^{\vaz,\kappa} 2^{[\ell(Q)-\ell(P)](s+1)wa_-}2^{[\ell(Q)-\ell(P)]w\nu}.
\end{align}

Note that $\ell(P)\le\ell(Q)$ and $R_Q\cap R_P\neq \emptyset$. Thus,
\begin{align}\label{x3e3}
R_Q:=x_Q+2^{(w[\ell(Q)-1]+u+2)\va}B_0
\subset x_P+2^{(w[\ell(P)-1]+u+4)\va}B_0=:R_P'.
\end{align}
Moreover, for any $x\in\rn$ such that $\widehat{P}\cap\Gamma(x)\neq\emptyset$,
we know that there exists $(y_0,t_0)\in \widehat{P}\cap\Gamma(x)$ such that
$|x-y_0|_{\va}<t_0<2^{w[\ell(P)-1]+u+1}$; thus, from Lemma \ref{4l2x}(iv)
and \cite[Lemma 2.5(ii)]{hlyy}, we deduce that
\begin{align*}
x&\in y_0+2^{(w[\ell(P)-1]+u+1)\va}B_0
\subset x_P+2^{[w\ell(P)+u]\va}B_0+2^{(w[\ell(P)-1]+u+1)\va}B_0\\
&\subset x_P+2^{(w[\ell(P)-1]+u+2)\va}B_0=R_P\subset R_P'.
\end{align*}
By this, \eqref{x3e3} and \cite[Lemma 2.7]{blyz10}, we conclude that,
for any $x\in\rn$ such that $\widehat{P}\cap\Gamma(x)\neq\emptyset$,
\begin{align*}
2^{[\ell(Q)-\ell(P)]w\nu}\sim \frac{|R_Q|}{|R_P'|}
\sim \frac{|R_Q\cap R_P'|}{|R_P'|}\ls M_{\rm HL}^{\va}(\mathbf{1}_{R_Q})(x)
\ls M_{\rm HL}^{\va}(\mathbf{1}_{Q})(x),
\end{align*}
which, combined with \eqref{x3e22}, implies \eqref{x3e20} holds true
in this case.

\emph{Case II)} $\ell(P)> \ell(Q)$. In this case, for any $(y,t)\in\widehat{P}$
and $z\in R_P$, it is easy to see that
\begin{align*}
z':=2^{[-w\ell(Q)-u]\va}z\in 2^{[-w\ell(Q)-u]\va}x_P
+2^{(w[\ell(P)-\ell(Q)-1]+2)\va}B_0=:R_{PQ}.
\end{align*}
From this, Lemma \ref{3l8} and the facts
that $-w>0$ and $w[\ell(P)-\ell(Q)]>0$,
it follows that, for any $z'\in R_{PQ}$,
\begin{align}\label{x3e23}
\lf|z'-2^{[-w\ell(Q)-u]\va}x_P\r|
&=\lf|2^{(-w+2)\va}\lf[2^{(w-2)\va}\lf(z'-2^{[-w\ell(Q)-u]\va}x_P\r)\r]\r|\noz\\
&\le 2^{(-w+2)a_+}\lf|2^{(w-2)\va}\lf(z'-2^{[-w\ell(Q)-u]\va}x_P\r)\r|
\ls 2^{[\ell(P)-\ell(Q)]wa_-}.
\end{align}
Let $\wz{e_Q^{\vaz,\kappa}}(\cdot):=e_Q^{\vaz,\kappa}(2^{[w\ell(Q)+u]\va}\cdot)$.
For any $\az\in\zz_+^n$, using the H\"{o}lder inequality, we find that,
for any given $\kappa\in(0,\fz)$, any $\vaz\in(0,\fz)$ and $z\in\rn$,
\begin{align*}
\lf|\partial^{\az}\wz{e_Q^{\vaz,\kappa}}(z)\r|^2
&=\lf|\int_{\widehat{Q}_{\vaz,\kappa}}f\ast\psi_{\tau}(y)
\partial^{\az}\varphi_{\tau}(2^{[w\ell(Q)+u]\va}z-y)\,dy\,dm(\tau)\r|^2\\
&\le (c_Q^{\vaz,\kappa})^2\int_{\widehat{Q}_{\vaz,\kappa}}\lf|
\partial^{\az}\varphi_{\tau}(2^{[w\ell(Q)+u]\va}z-y)\r|^2\tau^{\nu}\,dy\,dm(\tau)\\
&\ls (c_Q^{\vaz,\kappa})^2 \sum_{\tau\sim 2^{w\ell(Q)+u}}|Q|\tau^{-\nu}\ls (c_Q^{\vaz,\kappa})^2.
\end{align*}
By this, the vanishing moments of $\varphi$ and \eqref{x3e23},
we know that, for any given $\kappa\in(0,\fz)$, any $\vaz\in(0,\fz)$
and $(y,t)\in\widehat{P}$,
\begin{align}\label{x3e24}
\lf|e_Q^{\vaz,\kappa}\ast\varphi_t(y)\r|
&=\lf|2^{[w\ell(Q)+u]\nu}\int_{\rn}e_Q^{\vaz,\kappa}(2^{[w\ell(Q)+u]\va}z)
\varphi_t(y-2^{[w\ell(Q)+u]\va}z)\,dz\r|\noz\\
&=\lf|2^{[w\ell(Q)+u]\nu}\int_{\rn}\widetilde{e_Q^{\vaz,\kappa}}(z)
\varphi_t(y-2^{[w\ell(Q)+u]\va}z)\,dz\r|\noz\\
&=\lf|2^{[w\ell(Q)+u]\nu}\int_{R_{PQ}}\lf[\wz{e_Q^{\vaz,\kappa}}(z)
-\sum_{|\az|\le s}\frac{\partial^{\az}\wz{e_Q^{\vaz,\kappa}}
(2^{[-w\ell(Q)-u]\va}x_P)}{\az!}
\lf(z-2^{[-w\ell(Q)-u]\va}x_P\r)^{\az}\r]\r.\noz\\
&\hs\hs\times \varphi_t(y-2^{[w\ell(Q)+u]\va}z)\,dz\Bigg|\noz\\
&=2^{[w\ell(Q)+u]\nu}c_Q^{\vaz,\kappa}\int_{R_{PQ}}
\lf|z-2^{[-w\ell(Q)-u]\va}x_P\r|^{s+1}\lf|
\varphi_t(y-2^{[w\ell(Q)+u]\va}z)\r|\,dz\noz\\
&\ls c_Q^{\vaz,\kappa} 2^{[\ell(P)-\ell(Q)](s+1)wa_-}.
\end{align}

Note that $\ell(P)>\ell(Q)$ and $R_Q\cap R_P\neq \emptyset$. Thus,
\begin{align}\label{x3e19}
R_P:=x_P+2^{(w[\ell(P)-1]+u+2)\va}B_0
\subset x_Q+2^{(w[\ell(Q)-1]+u+4)\va}B_0=:R_Q'.
\end{align}
Moreover, for any $x\in\rn$ such that $\widehat{P}\cap\Gamma(x)\neq\emptyset$,
we know that there exists $(y_1,t_1)\in \widehat{P}\cap\Gamma(x)$ such that
$|x-y_1|_{\va}<t_1<2^{w[\ell(P)-1]+u+1}$. Thus, by Lemma \ref{4l2x}(iv)
and \cite[Lemma 2.5(ii)]{hlyy}, we find that, for any
$x\in\rn$ such that $\widehat{P}\cap\Gamma(x)\neq\emptyset$,
\begin{align*}
x&\in y_1+2^{(w[\ell(P)-1]+u+1)\va}B_0
\subset x_P+2^{[w\ell(P)+u]\va}B_0+2^{(w[\ell(P)-1]+u+1)\va}B_0\\
&\subset x_P+2^{(w[\ell(P)-1]+u+2)\va}B_0=R_P.
\end{align*}
By this, \eqref{x3e19} and \cite[Lemma 2.7]{blyz10}, we conclude that,
for any $x\in\rn$ such that $\widehat{P}\cap\Gamma(x)\neq\emptyset$,
\begin{align*}
1\sim \frac{|R_P|}{|R_P|}\sim \frac{|R_Q'\cap R_P|}{|R_P|}
\ls M_{\rm HL}^{\va}(\mathbf{1}_{R_Q'})(x)
\ls M_{\rm HL}^{\va}(\mathbf{1}_{Q})(x),
\end{align*}
which, together with \eqref{x3e24}, implies \eqref{x3e20} also holds true
in this case. This finishes the proof of \eqref{4e11'} and hence of the
sufficiency of Theorem \ref{5t3}.
\end{proof}

Using Theorem \ref{5t4}, in \cite[Proposition 4.5]{hlyy}, Huang
 et al. further established the relation between $\vh$ and $H_A^p(\rn)$ as follows,
where $H_A^p(\rn)$ denotes the anisotropic Hardy space introduced by Bownik in
\cite[p.\,17, Definition 3.11]{mb03}.

\begin{proposition}\label{5p1}
Let $\va:=(a_1,\ldots,a_n)\in[1,\fz)^n$,
$\vp:=(\overbrace{p,\ldots,p}^{n\ \rm times})$ with $p\in(0,\fz)$
and $A$ be as in \eqref{5eq2}.
Then $\vh$ and the anisotropic Hardy space $\vAh$
coincide with equivalent quasi-norms.
\end{proposition}

\subsubsection{A new proof for maximal function characterizations of $\vh$}\label{4s2.2}

This subsection is devoted to providing a new proof of maximal
function characterizations of $\vh$, which improves Theorem \ref{3t1}.
We first state the main results of this subsection as follows.

\begin{theorem}\label{3t2}
Let $\vp\in(0,\fz)^n$, $N\in\nn\cap[\lfloor\frac1{p_-}\rfloor+2\nu+3,\fz)$
and $\varphi$ be as in Theorem \ref{3t1}.
Then the following statements are mutually equivalent:
\begin{enumerate}
\item[{\rm(i)}] $f\in\vh;$
\item[{\rm(ii)}] $f\in\cs'(\rn)$ and $M_\varphi(f)\in\lv;$
\item[{\rm(iii)}] $f\in\cs'(\rn)$ and $M_\varphi^0(f)\in\lv.$
\end{enumerate}
Moreover, there exist two positive constants $C_3$ and $C_4$, independent
of $f$, such that
\begin{align*}
\|f\|_{\vh}\le C_3\lf\|M_\varphi^0(f)
\r\|_{\lv}\le C_3\lf\|M_\varphi(f)\r\|_{\lv}
\le C_4\|f\|_{\vh}.
\end{align*}
\end{theorem}

\begin{remark}\label{3r2}
Note that, in Theorem \ref{3t1},
the exponent $N$ belongs to
$$\mathbb{N}\cap\lf[\lf\lfloor\nu\frac{a_+}{a_-}
\lf(\frac{1}{\min\{1,p_-\}}+1\r)+\nu+2a_+\r\rfloor+1,\fz\r),$$
which is a proper subset of
$$\nn\cap\lf[\lf\lfloor\frac1{p_-}\r\rfloor+2\nu+3,\fz\r).$$
In this sense, Theorem \ref{3t2} improves Theorem \ref{3t1}.
\end{remark}

To show Theorem \ref{3t2}, we need several technical lemmas.
We begin with introducing some notation as follows.
In what follows, for any $x\in\rn$, let
\begin{equation*}
\rho_{\va}(x):=\sum_{j\in\mathbb{Z}}
2^{\nu j}\mathbf{1}_{2^{(j+1)\va}B_0\setminus 2^{j\va}B_0}(x)\hspace{0.25cm}
{\rm when}\ x\neq\vec0_n,\hspace{0.35cm} {\rm or\ else}
\hspace{0.25cm}\rho_{\va}(\vec0_n):=0.
\end{equation*}

We first recall the following notions of some maximal functions,
which are used later to prove Theorem \ref{3t2}.

\begin{definition}\label{3d7}
Let $K\in\zz$, $L\in[0,\fz)$ and $N\in\nn$. For any $\varphi\in\cs(\rn)$,
the \emph{maximal functions} $M_\varphi^{0(K,L)}(f)$,
$M_\varphi^{(K,L)}(f)$ and $T_\varphi^{N(K,L)}(f)$ of
$f\in\cs'(\rn)$ are defined, respectively, by
setting, for any $x\in\rn$,
\begin{align*}
M_\varphi^{0(K,L)}(f)(x):=\sup_{k\in\mathbb{Z},\,k\le K}
\lf|(f\ast\varphi_k)(x)\r|\lf[\max\lf\{1,\rho_{\va}\lf(2^{-\va K}x\r)\r\}\r]^
{-L}\lf[1+2^{-\nu(k+K)}\r]^{-L},
\end{align*}
\begin{align*}
M_\varphi^{(K,L)}(f)(x):=\sup_{k\in\mathbb{Z},\,k\le K}
\sup_{y\in B_{\va}(x,2^k)}\lf|(f\ast\varphi_k)(y)\r|\lf[\max\lf\{1,\rho_{\va}
\lf(2^{-\va K}y\r)\r\}\r]^{-L}\lf[1+2^{-\nu(k+K)}\r]^{-L}
\end{align*}
and
\begin{align*}
T_\varphi^{N(K,L)}(f)(x):=\sup_{k\in\mathbb{Z},\,k\le K}
\sup_{y\in\rn}\frac{|(f\ast\varphi_k)(y)|}{[\max\{1,\rho_{\va}
(2^{-\va k}(x-y))\}]^{N}}\frac{[1+2^{-\nu(k+K)}]^{-L}}
{[\max\{1,\rho_{\va}(2^{-\va K}y)\}]^{L}},
\end{align*}
here and thereafter, for any $\varphi\in\cs(\rn)$ and $k\in\zz$,
let $\varphi_{k}(\cdot):=2^{-k\nu}\varphi(2^{-k\va}\cdot)$.
Moreover, the \emph{maximal functions} $M_N^{0(K,L)}(f)$
and $M_N^{(K,L)}(f)$ of $f\in\cs'(\rn)$ are defined, respectively, by
setting, for any $x\in\rn$,
\begin{align*}
M_N^{0(K,L)}(f)(x):=
\sup_{\varphi\in\cs_N(\rn)}M_\varphi^{0(K,L)}(f)(x)
\end{align*} and
\begin{align*}
M_N^{(K,L)}(f)(x):=
\sup_{\varphi\in\cs_N(\rn)}M_\varphi^{(K,L)}(f)(x).
\end{align*}
\end{definition}

The following Lemmas \ref{3l2} and \ref{3l3} are just \cite[Lemma 2.3]{abr16}
with $A$ as in \eqref{5eq2} and \cite[Remark 2.8(iii)]{hlyy}, respectively.

\begin{lemma}\label{3l2}
There exists a positive constant $C$ such that, for any
$K\in\zz$, $L\in[0,\fz)$, $\lz\in(0,\fz)$, $N\in\nn\cap[\frac1{\lz},\fz)$,
$\varphi\in\cs(\rn)$, $f\in\cs'(\rn)$ and $x\in\rn$,
$$\lf[T_\varphi^{N(K,L)}(f)(x)\r]^{\lz}\le CM_{{\rm HL}}^{\va}
\lf(\lf[M_\varphi^{(K,L)}(f)\r]^{\lz}\r)(x),$$
where $T_\varphi^{N(K,L)}$ and $M_\varphi^{(K,L)}$ are as in Definition \ref{3d1}
and $M_{{\rm HL}}^{\va}$ denotes the anisotropic
Hardy--Littlewood maximal operator as in \eqref{3e4}.
\end{lemma}

\begin{lemma}\label{3l3}
Let $\vp\in [1,\fz)^n$. Then, for any $r\in (0,\fz)$ and $f\in\lv$,
$$\lf\||f|^r\r\|_{\lv}=\|f\|_{L^{r \vp}(\rn)}^r.$$
In addition, for any $\mu\in{\mathbb C}$, $\theta\in [0,\min\{1,p_-\}]$
with $p_-$ as in \eqref{2e10} and $f,\ g\in\lv$,
$\|\mu f\|_{\lv}=|\mu|\|f\|_{\lv}$ and
\begin{align*}
\|f+g\|_{\lv}^{\theta}\le \|f\|_{\lv}^{\theta}+\|g\|_{\lv}^{\theta}.
\end{align*}
\end{lemma}

Applying the monotone convergence theorem (see \cite[p.\,62, Corollary 1.9]{ss05}) and
\cite[p.\,304, Theorem 2]{bp61}, we have the following monotone
convergence property of $\lv$ and we omit the details.

\begin{lemma}\label{3l5}
Let $\vp\in[1,\fz)^n$ and $\{g_i\}_{i\in\nn}\subset\lv$
be any sequence of non-negative functions satisfying that
$g_i$, as $i\to\fz$, increases pointwisely almost everywhere
to some $g\in\lv$. Then
$$\|g-g_i\|_{\lv}\to0\hspace{0.5cm} as\ \ i\to\fz.$$
\end{lemma}

By Lemmas \ref{3l1}, \ref{3l2} and \ref{3l3}, we easily
obtain the following conclusion; the details are omitted.

\begin{lemma}\label{3l4}
Let $\vp\in(0,\fz)^n$. Then there exists a positive constant $C$ such that,
for any $K\in\zz$, $L\in[0,\fz)$, $N\in\nn\cap(\frac1{p_-},\fz)$,
$\varphi\in\cs(\rn)$ and $f\in\cs'(\rn)$,
\begin{align*}
\lf\|T_\varphi^{N(K,L)}(f)\r\|_{\lv}
\le C\lf\|M_\varphi^{(K,L)}(f)\r\|_{\lv},
\end{align*}
where $T_\varphi^{N(K,L)}$ and $M_\varphi^{(K,L)}$ are as in Definition \ref{3d1}.
\end{lemma}

To prove Theorem \ref{3t2}, we also need
the following three technical lemmas, which are just
\cite[p.\,45, Lemma 7.5, p.\,46, Lemma 7.6 and p.\,11, Lemma 3.2]{mb03}
with $A$ as in \eqref{5eq2}, respectively.
In what follows, for any $t\in\mathbb{R}$,
we denote by $\lceil t\rceil$ the \emph{least integer not less
than $t$}.

\begin{lemma}\label{3l6}
Let $\varphi\in\cs(\rn)$ and
$\int_{\rn}\varphi(x)\,dx\neq0$. Then,
for any given $N\in\mathbb{N}$ and $L\in[0,\fz)$, there
exist an $I=N+2(\nu+1)+L\lceil\frac{\nu}{a_-}\rceil$ and
a positive constant $C$, depending on $N$ and $L$,
such that, for any $K\in\zz_+$, $f\in\cs'(\rn)$ and $x\in\rn$,
$$M_I^{0(K,L)}(f)(x)\le C T_\varphi^{N(K,L)}(f)(x),$$
where $M_I^{0(K,L)}$ and $T_\varphi^{N(K,L)}$ are as in Definition \ref{3d1}.
\end{lemma}

\begin{lemma}\label{3l7}
Let $\varphi$ be as in Lemma \ref{3l6}.
Then, for any given $\lz\in(0,\fz)$ and $K\in\zz_+$,
there exist $L\in(0,\fz)$ and a positive constant $C$,
depending on $K$ and $\lz$, such that,
for any $f\in\cs'(\rn)$ and $x\in\rn$,
\begin{align}\label{3e6}
M_\varphi^{(K,L)}(f)(x)
\le C\lf[\max\lf\{1,\rho_{\va}(x)\r\}\r]^{-\lz},
\end{align}
where $M_\varphi^{(K,L)}$ is as in Definition \ref{3d1}.
\end{lemma}

\begin{lemma}\label{3l8}
There exists a positive constant $C$ such that, for any $x\in \rn$,
\begin{equation*}
C^{-1}[\rho_{\va}(x)]^{a_-/\nu} \le \lf|x\r|
\le  C [\rho_{\va}(x)]^{a_+/\nu}\quad  when\quad \rho_{\va}(x)\in[1,\fz),
\end{equation*}
and
\begin{equation*}
C^{-1} [\rho_{\va}(x)]^{a_+/\nu} \le \lf|x\r|
\le  C [\rho_{\va}(x)]^{a_-/\nu}\quad when\quad \rho_{\va}(x)\in[0,1).
\end{equation*}
\end{lemma}

We now prove Theorem \ref{3t2}.

\begin{proof}[Proof of Theorem \ref{3t2}]
Clearly, (i) implies (ii) and (ii) implies (iii). Therefore,
to prove this theorem,
it suffices to show that (ii) implies (i) and that (iii) implies (ii).

We first prove that (ii) implies (i). To this end,
from Lemma \ref{3l6} with $N\in\nn\cap(\frac1{p_-},\fz)$ and $L:=0$, it follows that,
for any $I\in\nn\cap[\lfloor\frac1{p_-}\rfloor+2\nu+3,\fz)$, $K\in\zz_+$, $f\in\cs'(\rn)$
and $x\in\rn$, $$M_I^{0(K,0)}(f)(x)\ls T_\varphi^{N(K,0)}(f)(x).$$
This, combined with Lemma \ref{3l4}, implies that, for any $K\in\mathbb{Z}_+$
and $f\in\cs'(\rn)$,
\begin{align}\label{3e9}
\lf\|M_I^{0(K,0)}(f)\r\|_{\lv}\ls
\lf\|M_\varphi^{(K,0)}(f)\r\|_{\lv}.
\end{align}
Letting $K\to\fz$ in \eqref{3e9} and using Lemma \ref{3l5}, we easily know that
$$\lf\|M_I^0(f)\r\|_{\lv}\ls\lf\|M_\varphi(f)\r\|_{\lv}.$$
By this and \cite[p.\,17, Proposition 3.10]{mb03} with $A$ as in \eqref{5eq2},
we further conclude that, if (ii) holds true, then (i) also holds true.

Now we show that (iii) implies (ii). To this end, let $M_\varphi^0(f)\in \lv$.
By Lemma \ref{3l7} with $\lz\in(a_+/p_-,\fz)$ and $K\in\zz_+$,
we know that there exists some $L\in(0,\fz)$ such that \eqref{3e6} holds true.
Therefore, for any $K\in\zz_+$, $M_\varphi^{(K,L)}(f)\in \lv$.
Indeed, when $\lz\in(a_+/p_-,\fz)$, from Lemma \ref{3l3} with
$\theta:=\underline{p}$, and Lemma \ref{3l8}, we deduce that
\begin{align*}
\lf\|M_\varphi^{(K,L)}(f)\r\|_{\lv}^{\underline{p}}
&\le\lf\|M_\varphi^{(K,L)}(f){\mathbf 1}_{2^{\va}B_0}\r\|_{\lv}^{\underline{p}}
+\sum_{k\in\nn}\lf\|M_\varphi^{(K,L)}(f)
{\mathbf 1}_{2^{(k+1)\va}B_0\setminus 2^{k\va}B_0}\r\|_{\lv}^{\underline{p}}\\
&\ls\lf\|{\mathbf 1}_{2^{\va}B_0}\r\|_{\lv}^{\underline{p}}+\sum_{k\in\nn}
2^{-\nu\lz k\underline{p}}
\lf\|{\mathbf 1}_{2^{(k+1)\va}B_0\setminus 2^{k\va}B_0}\r\|_{\lv}^{\underline{p}}\\
&\ls\lf\|{\mathbf 1}_{B(\vec 0_n,1)}\r\|_{\lv}^{\underline{p}}+\sum_{k\in\nn}
2^{-\nu\lz k\underline{p}}
\lf\|{\mathbf 1}_{B(\vec 0_n,\,2^{ka_+})}\r\|_{\lv}^{\underline{p}}\\
&\ls \sum_{k\in\zz_+} 2^{-\nu\lz k\underline{p}}
2^{\nu k a_+\underline{p}/p_-}<\fz,
\end{align*}
where, for any $r\in(0,\fz)$,
$B(\vec 0_n,r):=\{y\in\rn:\ |y|<r\}$.
Thus, $M_\varphi^{(K,L)}(f)\in \lv$.

In addition, by Lemmas \ref{3l6} and \ref{3l4}, we conclude that, for
any given $L\in(0,\fz)$, there exist some $I\in\nn$ and a positive constant $C_5$
such that, for any $K\in\zz_+$ and $f\in\cs'(\rn)$,
$$\lf\|M_I^{0(K,L)}(f)\r\|_{\lv}\le C_5
\lf\|M_\varphi^{(K,L)}(f)\r\|_{\lv}.$$
For any fixed $K\in\zz_+$, let
\begin{align*}
G_K:=\lf\{x\in\rn:\ M_I^{0(K,L)}(f)(x)\le C_6
M_\varphi^{(K,L)}(f)(x)\r\},
\end{align*}
where $C_6:=2C_5$. Then
\begin{align}\label{3e10}
\lf\|M_\varphi^{(K,L)}(f)\r\|_{\lv}\ls
\lf\|M_\varphi^{(K,L)}(f)\r\|_{L^{\vp}(G_K)},
\end{align}
due to the fact that
$$\lf\|M_\varphi^{(K,L)}(f)\r\|_{L^{\vp}(G_K
^\complement)}\le C_6^{-1}\lf\|M_I^{0(K,L)}(f)\r\|
_{L^{\vp}(G_K^\complement)}
\le C_5/C_6\lf\|M_\varphi^{(K,L)}(f)\r\|_{\lv}.$$

For any given $L\in(0,\fz)$, repeating the proof of
\cite[(4.17)]{lyy16} with $p$ therein replaced by $p_-$
and $A$ as in \eqref{5eq2},
we find that, for any $r\in(0,p_-)$, $K\in\zz_+$,
$f\in\cs'(\rn)$ and $x\in G_K$,
\begin{align*}
\lf[M_\varphi^{(K,L)}(f)(x)\r]^r
\ls M_{{\rm HL}}^{\va}\lf(\lf[M_\varphi^{0(K,L)}(f)\r]^r\r)(x).
\end{align*}

From this, \eqref{3e10} and Lemmas \ref{3l3} and \ref{3l1},
we further deduce that,
for any $K\in\zz_+$ and $f\in\cs'(\rn)$,
\begin{align}\label{3e11}
\lf\|M_\varphi^{(K,L)}(f)\r\|_{\lv}^r
&\ls\lf\|M_\varphi^{(K,L)}(f)\r\|_{L^{\vp}(G_K)}^r
\sim\lf\|\lf[M_\varphi^{(K,L)}(f)\r]^r\r\|_{L^{
\vp/r}(G_K)}\noz\\
&\ls\lf\|M_{{\rm HL}}^{\va}\lf(\lf[M_\varphi^{0(K,L)}(f)\r]^
r\r)\r\|_{L^{\vp/r}(\rn)}\noz\\
&\ls\lf\|\lf[M_\varphi^{0(K,L)}(f)\r]^r\r\|_{L^{
\vp/r}(\rn)}
\sim\lf\|M_\varphi^{0(K,L)}(f)\r\|_{\lv}^r.
\end{align}
Letting $K\to\fz$ in \eqref{3e11}, by Lemma \ref{3l5},
we have
\begin{align*}
\lf\|M_\varphi (f)\r\|_{\lv}\ls\lf\|M_\varphi^0(f)\r\|_{\lv}.
\end{align*}
This shows that (iii) implies (ii) and
hence finishes the proof of Theorem \ref{3t2}.
\end{proof}

\subsection{Dual spaces of $\vh$}\label{4s3}

In this subsection, we mainly discuss the dual spaces of $\vh$.
Indeed, the dual spaces of $\vh$ were asked by
Cleanthous et al. in \cite{cgn17}
and part of them were obtained by Huang et al. in \cite{hlyy18}.
To present this, we first introduce the notion of the anisotropic mixed-norm
Campanato space $\lq$ given in \cite{hlyy18}.

\begin{definition}\label{5d4}
Let $\va\in [1,\fz)^n$, $\vp\in(0,\fz]^n$, $q\in[1,\fz]$ and $s\in\zz_+$.
The \emph{anisotropic mixed-norm Campanato space} $\lq$ is defined to be
the set of all measurable functions $g$ such that, when $q\in[1,\fz)$,
$$\|g\|_{\lq}:=\sup_{B\in\mathfrak{B}}\inf_{P\in\cp_s(\rn)}
\frac{|B|}{\|{\mathbf 1}_B\|_{\lv}}\lf[\frac1{|B|}\int_B\lf|g(x)-P(x)\r|^q\,dx\r]^{1/q}<\fz$$
and
$$\|g\|_{\mathcal{L}^{\va}_{\vec{p},\,\fz,\,s}(\rn)}
:=\sup_{B\in\mathfrak{B}}\inf_{P\in\cp_s(\rn)}
\frac{|B|}{\|{\mathbf 1}_B\|_{\lv}}\lf\|g-P\r\|_{L^{\fz}(B)}<\fz,$$
where $\mathfrak{B}$ is as in \eqref{3e2}.
\end{definition}

\begin{remark}\label{3r1}
\begin{enumerate}
\item[{\rm (i)}] Obviously, $\|\cdot\|_{\lq}$ is a seminorm and
$\cp_s(\rn)\subset \lq$. Indeed,
$$\|g\|_{\lq}=0\quad {\rm if\ and\ only\ if}\quad g\in\cp_s(\rn).$$
Thus, if we identify $g_1$ with $g_2$ when $g_1-g_2\in\cp_s(\rn)$, then
$\lq$ becomes a Banach space. In what follows, we always identify
$g\in\lq$ with $\{g+P:\, P\in\cp_s(\rn)\}$.
\item[{\rm (ii)}] When $\va:=(\overbrace{1,\ldots,1}^{n\ \rm times})$ and
$\vp:=(\overbrace{p,\ldots,p}^{n\ \rm times})$ with some $p\in(0,1]$, for any
$B\in\mathfrak{B}$, we have $\|\mathbf{1}_B\|_{\lv}=|B|^{1/p}$. Then the
space $\lq$ is just the classical Campanato space $L_{\frac1p-1,\,q,\,s}(\rn)$
introduced by Campanato in \cite{c64}, which includes the classical
space $\mathop{\mathrm{BMO}}(\rn)$, introduced by John and Nirenberg
in \cite{jn61}, as a special case.
\end{enumerate}
\end{remark}

Via Theorems \ref{5t1} and \ref{5t2}, the following
duality theorem was established in \cite[Theorem 3.10]{hlyy18}.

\begin{theorem}\label{5t6}
Let $\vp\in (0,1]^n$ and $\va$, $r$
and $s$ be as in Definition \ref{5d1}.
Then the dual space of $\vh$, denoted by $(\vh)^*$, is $\lr$
with $1/r+1/r'=1$ in the following sense:
\begin{enumerate}
\item[{\rm (i)}] Suppose that $g\in\lr$. Then the linear functional
\begin{align}\label{5eq3}
L_g:\ f\longmapsto L_g(f):=\int_{\rn}f(x)g(x)\,dx,
\end{align}
initially defined for any $f\in\vfah$, has a bounded extension to $\vh$.

\item[{\rm (ii)}] Conversely, any continuous linear functional on $\vh$
arises as in \eqref{5eq3} with a unique $g\in\lr$.
\end{enumerate}
Moreover, $\lf\|g\r\|_{\lr}\sim \lf\|L_g\r\|_{(\vh)^*},$ where the positive equivalence
constants are independent of $g$.
\end{theorem}

\begin{remark}\label{3r3}
\begin{enumerate}
\item[{\rm (i)}]
When $\va:=(\overbrace{1,\ldots,1}^{n\ \rm times})$ and
$\vp:=(\overbrace{p,\ldots,p}^{n\ \rm times})$ with some $p\in(0,1]$,
the two spaces $\vh$ and $\lr$ become, respectively, the classical Hardy space $H^p(\rn)$ and the classical Campanato space $L_{\frac1p-1,\,r',\,s}(\rn)$.
In this case,
Theorem \ref{3t1} was obtained by Taibleson and Weiss \cite{tw80}, which includes
the famous duality result, obtained by Fefferman and Stein in \cite{fs72},
$(H^1(\rn))^*={\mathop{\mathrm{BMO}}}(\rn)$, as a special case.
\item[{\rm (ii)}] Note that,
when $\va:=(\overbrace{1,\ldots,1}^{n\ \rm times})$,
the space $\vh$ is just the isotropic mixed-norm Hardy space.
We point out that, even in this case, Theorem \ref{5t6} is also new.
\item[{\rm (iii)}]
When $\vp\in(1,\fz)^n$, by Proposition \ref{3p1}, we know that
$\vh=\lv$ with equivalent norms. This, together with
Theorem \ref{x2t18}, further implies that, for any $\vp\in(1,\fz)^n$,
$L^{\vp'}(\rn)$ is the dual space of $\vh$.
However, when
$\vp:=(p_1,\ldots,p_n)\in(0,\fz)^n$ with $p_{i_0}\in(0,1]$ and $p_{j_0}\in(1,\fz)$
for some $i_0$, $j_0\in\{1,\ldots,n\}$, the dual space of $\vh$ is still unknown so far.
\end{enumerate}
\end{remark}

\subsection{Applications to boundedness of sublinear operators}\label{4s4}

This subsection is devoted to displaying some applications
of Hardy spaces $\vh$, which were obtained in \cite{hlyy}. More precisely,
in this subsection, we first recall a criterion on the
boundedness of sublinear operators from $\vh$ into a quasi-Banach
space. Then, applying this criterion,
the boundedness of anisotropic convolutional $\delta$-type and
non-convolutional $\bz$-order Calder\'on--Zygmund operators
from $\vh$ to itself [or to $\lv$] was obtained. In addition,
we improve \cite[Theorems 6.8 and 6.9]{hlyy}.

Recall that a complete vector space is called a \emph{quasi-Banach space} $\mathcal{B}$ if its quasi-norm $\|\cdot\|_{\mathcal{B}}$ satisfies
\begin{enumerate}
\item[{\rm (i)}] $\|\psi\|_{\mathcal{B}}=0$ if and only if $\psi$ is the zero element of $\mathcal{B}$;
\item[{\rm (ii)}] there exists a positive constant $K\in[1,\fz)$ such that, for any
$\psi,\ \phi\in\mathcal{B}$,
$$\|\psi+\phi\|_{\mathcal{B}}\le K(\|\psi\|_{\mathcal{B}}+\|\phi\|_{\mathcal{B}}).$$
\end{enumerate}
Note that, when $K=1$, a quasi-Banach space $\mathcal{B}$ is just a Banach space.
In addition, for any given $\gamma\in(0,1]$, a \emph{$\gamma$-quasi-Banach space}
${\mathcal{B}_{\gamma}}$ is a quasi-Banach space equipped
with a quasi-norm $\|\cdot\|_{\mathcal{B}_{\gamma}}$ satisfying that
there exists a constant $C\in[1,\fz)$ such that, for any $\kappa\in \nn$ and $\{\psi_i\}_{i=1}^{\kappa}\subset\mathcal{B}_{\gamma}$,
$$\lf\|\sum_{i=1}^\kappa \psi_i\r\|_{\mathcal{B}_{\gamma}}^{\gamma}\le
C \sum_{i=1}^\kappa \lf\|\psi_i\r\|_{\mathcal{B}_{\gamma}}^{\gamma}$$
holds true (see \cite{hlyy, ky14, ylk17, zy08, zy09}).

Let $\mathcal{B}_{\gamma}$ be a $\gamma$-quasi-Banach space with
$\gamma\in(0,1]$ and $\mathcal{Y}$ a linear space. An operator
$T$ from $\mathcal{Y}$ to $\mathcal{B}_{\gamma}$ is said to be
$\mathcal{B}_{\gamma}$-\emph{sublinear} if
there exists a positive constant $C$ such that, for any $\kappa\in \nn$,
$\{\mu_{i}\}_{i=1}^{\kappa}\subset \mathbb{C}$ and
$\{\psi_{i}\}_{i=1}^\kappa\subset\mathcal{Y}$,
$$\lf\|T\lf(\sum_{i=1}^\kappa \mu_i \psi_i\r)\r\|_{\mathcal{B}
_{\gamma}}^{\gamma}\le C\sum_{i=1}^\kappa |\mu_i|^{\gamma}\lf\|T(\psi_i)\r\|_{\mathcal{B}_{\gamma}}^{\gamma}$$
and, for any $\psi,$ $\phi\in \mathcal{Y}$,
$\|T(\psi)-T(\phi)\|_{\mathcal{B}_{\gamma}}\le C\|T(\psi-\phi)\|
_{\mathcal{B}_{\gamma}}$(see \cite{hlyy, ky14, ylk17, zy08, zy09}).
Obviously, for any $\gamma\in (0,1]$, the linearity of $T$ implies
its $\mathcal{B}_{\gamma}$-sublinearity.

We first state the following criterion on the boundedness of
sublinear operators from $\vh$ into a quasi-Banach
space $\mathcal{B}_{\gamma}$, which was proved by
Huang et al. in \cite[Theorem 6.2]{hlyy}
via Theorem \ref{5t2}.

\begin{theorem}\label{5t7}
Assume that $\va\in [1,\fz)^n$, $\vp\in (0,\fz)^n$, $r\in(\max\{p_+,1\},\fz]$
with $p_+$ as in \eqref{2e10}, $\gamma\in (0,1]$, $s$ is
as in \eqref{5eq1} and $\mathcal{B}_{\gamma}$ a $\gamma$-quasi-Banach space.
If either of the following two statements holds true:
\begin{enumerate}
\item[{\rm (i)}] $r\in(\max\{p_+,1\},\fz)$ and
$T:\ \vfah\to\mathcal{B}_{\gamma}$
is a $\mathcal{B}_{\gamma}$-sublinear operator satisfying that
there exists a positive constant $C$ such that,
for any $f\in \vfah$,
\begin{align*}
\lf\|T(f)\r\|_{\mathcal{B}_{\gamma}}\le C\|f\|_{\vfah};
\end{align*}
\item[{\rm (ii)}]
$T:\ \vfahfz\cap C(\rn)\to\mathcal{B}_{\gamma}$
is a $\mathcal{B}_{\gamma}$-sublinear operator satisfying that
there exists a positive constant $C$ such that,
for any $f\in \vfahfz\cap C(\rn)$,
$$\lf\|T(f)\r\|_{\mathcal{B}_{\gamma}}\le C\|f\|_{\vfahfz},$$
\end{enumerate}
then $T$ uniquely extends to a bounded $\mathcal{B}_{\gamma}$-sublinear operator from $\vh$
into $\mathcal{B}_{\gamma}$. Moreover, there exists a positive constant $C$ such that,
for any $f\in \vh$,
$$\lf\|T(f)\r\|_{\mathcal{B}_{\gamma}}\le C\|f\|_{\vh}.$$
\end{theorem}

Using Theorem \ref{5t7}, the following useful conclusion
was also shown in \cite[Corollary 6.3]{hlyy}.

\begin{corollary}\label{5c2}
Let $\vp\in(0,1]^n$ and $\va,\ r,\ \gamma,\ s$ and $\mathcal{B}_{\gamma}$ be as
in Theorem \ref{5t7}. If either of the following two statements holds true:
\begin{enumerate}
\item[{\rm (i)}] $r\in(1,\fz)$ and $T$ is a $\mathcal{B}_{\gamma}$-sublinear
operator from $\vfah$ to $\mathcal{B}_{\gamma}$ satisfying
$$\sup\lf\{\lf\|T(a)\r\|_{\mathcal{B}_{\gamma}}:\
a\ {\rm is\ any}\ (\vp,r,s){\text-}{\rm atom}\r\}<\fz;$$
\item[{\rm(ii)}] $T$ is a $\mathcal{B}_{\gamma}$-sublinear
operator defined on all continuous $(\vp,\fz,s)$-atoms satisfying
$$\sup\lf\{\lf\|T(a)\r\|_{\mathcal{B}_{\gamma}}:\
 a\ {\rm is\ any\ continuous}\ (\vp,\fz,s){\text-}{\rm atom}\r\}<\fz,$$
\end{enumerate}
then $T$ uniquely extends to a bounded $\mathcal{B}_{\gamma}$-sublinear
operator from $\vh$ into $\mathcal{B}_{\gamma}$.
\end{corollary}

We now recall a class of anisotropic convolutional
$\delta$-type Calder\'{o}n--Zygmund operators and
anisotropic non-convolutional $\beta$-order
Calder\'{o}n--Zygmund operators in \cite{hlyy}
as follows. In what follows, for any
$E\subset \rn\times\rn$ and $m\in\zz_+$, denote by $C^m(E)$
the set of all functions on $E$ whose derivatives with order
not greater than $m$ are continuous.

\begin{definition}\label{8d0}
For any $\delta\in (0,a_+]$,
a linear operator $T$ is called an \emph{anisotropic convolutional
$\delta$-type Calder\'{o}n--Zygmund operator} if $T$
is bounded on $L^2(\rn)$ with kernel
$k\in \cs'(\rn)$ coinciding with a locally integrable
function on $\rn\setminus\{\vec{0}_n\}$ and satisfying that
there exists a positive constant $C$ such that,
for any $x,\ y\in \rn$ with $|x|_{\va}>2|y|_{\va}\neq 0$,
\begin{align*}
\lf|k(x-y)-k(x)\r|\le C\frac{|y|_{\va}^{\delta}}{|x|_{\va}^{\nu+\delta}}
\end{align*}
and, for any $f\in L^2(\rn)$ and $x\in\rn$, $T(f)(x):={\rm p.\,v.}\ k\ast f(x)$.
\end{definition}

\begin{definition}\label{8d1}
Let $\va\in[1,\fz)^n$. For any given $\beta\in (0,\fz)\setminus\nn$, a linear operator
$T$ is called an \emph{anisotropic non-convolutional $\beta$-order Calder\'{o}n--Zygmund operator} if $T$ is bounded
on $L^2(\rn)$ and its kernel
$$\mathcal{K}:\ \Omega:=\lf\{(x,y)\in\rn\times\rn:\  x\neq y\r\}\to \mathbb{C}$$
satisfies that $\mathcal{K}\in C^{\lfloor\bz \rfloor}(\Omega)$ and
there exists a positive constant $C$ such that, for any $\az\in\zz_+^n$ with
$|\alpha|= \lfloor\bz\rfloor$ and $x,\ y,\ z\in \rn$ with $ |x-y|_{\va}>2|y-z|_{\va}\neq 0$,
\begin{align*}
\lf|\lf[\pa^{\az}\mathcal{K}(x,\cdot)\r](y)
-\lf[\pa^{\az}\mathcal{K}(x,\cdot)\r](z)\r|
\le C\frac{|y-z|_{\va}^{\bz a_+}}{|x-y|_{\va}^{\nu+\bz a_+}}
\min\lf\{|y-z|_{\va}^{-\lfloor\bz\rfloor a_+}
,\,|y-z|_{\va}^{-\lfloor\bz\rfloor a_-}\r\}
\end{align*}
and, for any $f\in L^2(\rn)$ with compact support, and $x\notin \supp f$,
$$T(f)(x)=\int_{\supp f}\mathcal{K}(x,y)f(y)\,dy.$$
\end{definition}

For any $l\in\nn$, an operator $T$ is said to have the \emph{vanishing moments
up to order $l$} if, for any $a\in L^2(\rn)$ with compact support and
satisfying that, for any $\gamma\in\zz_+^n$ with $|\gamma|\le l$,
$\int_{\rn}x^{\gamma}a(x)\,dx=0$, then
$$\int_{\rn}x^{\gamma}T(a)(x)\,dx=0.$$

In \cite[Theorems 6.4 and 6.5]{hlyy}, Huang et al. established the
following boundedness of anisotropic convolutional $\delta$-type
Calder\'{o}n--Zygmund operators from $\vh$ to itself
or to $\lv$.

\begin{theorem}\label{5t8}
Let $\va\in [1,\fz)^n$, $\vp\in (0,\fz)^n$, $\delta\in(0,a_+]$
and $p_-\in(\frac\nu{\nu+\delta},\fz)$ with $p_-$ as in \eqref{2e10}.
Let $T$ be an anisotropic convolutional $\dz$-type Calder\'on--Zygmund operator.
Then there exists a positive constant $C$
such that, for any $f\in \vh$,
$$\lf\|T(f)\r\|_{\vh}\le C\|f\|_{\vh}.$$
\end{theorem}

\begin{theorem}\label{5t9}
Let $\va,\ \vp,\ \delta,\ p_-$ and $T$ be as in Theorem \ref{5t8}.
Then there exists a positive constant $C$
such that, for any $f\in \vh$,
$$\lf\|T(f)\r\|_{\lv}\le C\|f\|_{\vh}.$$
\end{theorem}

Moreover, the following boundedness of anisotropic $\beta$-order
Calder\'{o}n--Zygmund operators $T$ from $\vh$ to itself or to $\lv$
was also established in \cite[Theorems 6.8 and 6.9]{hlyy}.

\begin{theorem}\label{5t10}
Let $\va\in[1,\fz)^n,$ $\vp\in(0,2)^n$, $\bz\in(0,\fz)\setminus\nn$,
$$p_-\in\lf(\frac{\nu}{\nu+\bz a_-},\frac{\nu}{\nu+\lfloor\bz\rfloor a_-}\r]$$
with $p_-$ as in $\eqref{2e10}$ and $a_-$ as in \eqref{2e9}, and $T$ be
an anisotropic $\beta$-order Calder\'{o}n--Zygmund operator having the
vanishing moments up to order $\lfloor\bz\rfloor$.
Then there exists a positive constant $C$
such that, for any $f\in \vh$,
$$\lf\|T(f)\r\|_{\vh}\le C\|f\|_{\vh}.$$
\end{theorem}

\begin{theorem}\label{5t11}
Let $\va,\ \vp,\ \bz$ and $p_-$ be the same as in Theorem \ref{5t10}
and $T$ an anisotropic
$\beta$-order Calder\'{o}n--Zygmund operator. Then there exists a positive
constant $C$ such that, for any $f\in \vh$,
$$\lf\|T(f)\r\|_{\lv}\le C\|f\|_{\vh}.$$
\end{theorem}

Finally, we point out that, in Definition \ref{8d1},
the range of $\beta\in (0,\fz)\setminus\nn$ can be
revised as $\beta\in (0,\fz)$, and then we can
improve the restriction of $\bz$ in Theorems \ref{5t10}
and \ref{5t11} into $\beta\in (0,\fz)$. To be precise,
we have the following improved versions of Definition
\ref{8d1} and Theorems \ref{5t10} and \ref{5t11}.

\begin{definition}\label{8d1`}
Let $\va\in[1,\fz)^n$. For any given $\beta\in (0,\fz)$, a linear operator
$T$ is called an \emph{anisotropic non-convolutional $\beta$-order Calder\'{o}n--Zygmund operator} if $T$ is bounded
on $L^2(\rn)$ and its kernel
$$\mathcal{K}:\ \Omega:=\{(x,y)\in\rn\times\rn:\  x\neq y\}\to \mathbb{C}$$
satisfies that $\mathcal{K}\in C^{\lceil \bz\rceil-1}(\Omega)$ and
there exists a positive constant $C$ such that, for any $\az\in\zz_+^n$ with
$|\alpha|= \lceil \bz\rceil-1$ and $x,\ y,\ z\in \rn$
with $ |x-y|_{\va}>2|y-z|_{\va}\neq 0$,
\begin{align*}
&\lf|\lf[\pa^{\az}\mathcal{K}(x,\cdot)\r](y)
-\lf[\pa^{\az}\mathcal{K}(x,\cdot)\r](z)\r|\\
&\hs\hs\le C\frac{|y-z|_{\va}^{\bz a_+}}{|x-y|_{\va}^{\nu+\bz a_+}}
\min\lf\{|y-z|_{\va}^{-(\lceil \bz\rceil-1) a_+}
,\,|y-z|_{\va}^{-(\lceil \bz\rceil-1) a_-}\r\}
\end{align*}
and, for any $f\in L^2(\rn)$ with compact support, and $x\notin \supp f$,
$$T(f)(x)=\int_{\supp f}\mathcal{K}(x,y)f(y)\,dy.$$
\end{definition}

\begin{theorem}\label{5t10`}
Let $\va\in[1,\fz)^n,$ $\vp\in(0,2)^n$, $\bz\in(0,\fz)$,
$$p_-\in\lf(\frac{\nu}{\nu+\bz a_-},\frac{\nu}{\nu+(\lceil \bz\rceil-1) a_-}\r]$$
with $p_-$ as in $\eqref{2e10}$ and $a_-$ as in \eqref{2e9}, and $T$ be
an anisotropic $\beta$-order Calder\'{o}n--Zygmund operator having the
vanishing moments up to order $\lceil \bz\rceil-1$.
Then there exists a positive constant $C$
such that, for any $f\in \vh$,
$$\lf\|T(f)\r\|_{\vh}\le C\|f\|_{\vh}.$$
\end{theorem}

\begin{theorem}\label{5t11`}
Let $\va,\ \vp,\ \bz$ and $p_-$ be the same as in Theorem \ref{5t10`} and $T$
an anisotropic $\beta$-order Calder\'{o}n--Zygmund operator.
Then there exists a positive constant $C$ such that, for any $f\in \vh$,
$$\lf\|T(f)\r\|_{\lv}\le C\|f\|_{\vh}.$$
\end{theorem}

To prove Theorems \ref{5t10`} and \ref{5t11`}, we only need to replace
$\lfloor \bz \rfloor$ by $\lceil \bz\rceil-1$ in the proofs of
\cite[Theorems 6.8 and 6.9]{hlyy}, respectively. Thus, the details
are omitted.

At the end of this section, we point out that all the results
about the mixed-norm Hardy space $H_{\va}^{\vp}(\rn)$ associated to a
vector $\va\in[1,\fz)^n$
can be extended to a more general anisotropic setting, namely, the mixed-norm
Hardy space $H_A^{\vp}(\rn)$ associated to an expansive matrix $A$.
We refer the reader to \cite{hlyy19} for the details.

\section*{Acknowledgments}
The authors would like to express his deep thanks to Professor Ferenc Weisz
for his pointing out that, when $n:=2$ and $\vp:=(p_1,\infty)$ with $p_1\in(1,\fz)$,
\cite[Lemma 3.5]{hlyy} is not correct.
Long Huang would like to express his deep thanks to Dr. Jun Liu for his many
discussions and suggestions on this survey.
The research of the authors is supported by the National
Natural Science Foundation of China
(Grant Nos.~11571039, 11761131002 and 11671185).


\end{document}